%% file: tmfcoop.tex
\documentclass{amsart}

\input{tmfcoop_preamble}

\usepackage{graphicx}
\usepackage{caption}
\usepackage{subcaption}
\usepackage{wrapfig}

\title[On the ring of cooperations for $\tmf$]{On the ring of cooperations for $2$-primary connective topological modular forms}

\author{M.~Behrens}\address{University of Notre Dame}\email{mbehren1@nd.edu}
\author{K.~Ormsby}\address{Reed College}\email{ormsbyk@reed.edu}
\author{N.~Stapleton}\address{University of Kentucky}\email{nat.j.stapleton@gmail.com}
\author{V.~Stojanoska}\address{University of Illinois Urbana-Champaign}\email{vesna@illinois.edu}

\dedicatory{This paper is dedicated to the memory of Mark Mahowald}

\thanks{The first author was partially supported by NSF CAREER grant DMS-1050466 and NSF grant DMS-1611786, the second author was partially supported by NSF Postdoctoral Fellowship DMS-1103873 and NSF grant DMS-1406327, the third author was partially supported by NSF grant DMS-0943787, and the fourth author was partially supported by NSF grant DMS-1307390/1606479.}

\begin{document}

\begin{abstract}
We analyze the ring $\tmf_*\tmf$ of cooperations for the connective spectrum of topological modular forms (at the prime $2$) through a variety of perspectives: (1) the $E_2$-term of the Adams spectral sequence for $\tmf \wedge \tmf$ admits a decomposition in terms of $\Ext$ groups for $\bo$-Brown-Gitler modules, (2) the image of $\tmf_*\tmf$ in $\TMF_*\TMF_\QQ$ admits a description in terms of $2$-variable modular forms, and (3) modulo $v_2$-torsion, $\tmf_*\tmf$ injects into a certain product of copies of $\pi_*\TMF_0(N)$, for various values of $N$.  We explain how these different perspectives are related, and leverage these relationships to give complete information on $\tmf_*\tmf$ in low degrees.  We reprove a result of Davis-Mahowald-Rezk, that a piece of $\tmf \wedge \tmf$ gives a connective cover of $\TMF_0(3)$, and show that another piece gives a connective cover of $\TMF_0(5)$.  To help motivate our methods, we also review the existing work on $\bo_*\bo$, the ring of cooperations for ($2$-primary) connective $K$-theory, and in the process give some new perspectives on this classical subject matter. 
\end{abstract}

\maketitle	

\tableofcontents

\input{tmfcoop_intro}

\input{tmfcoop_bo}

\input{tmfcoop_review}

\input{tmfcoop_ass}

\input{tmfcoop_2var}

\input{tmfcoop_level}

\input{tmfcoop_cover}

\bibliographystyle{amsalpha}
\bibliography{tmfcoop}
\end{document}

%% file: tmfcoop_preamble.tex
\usepackage{amsmath}
\usepackage{amsthm}
\usepackage{amsopn}
\usepackage{amssymb}
\usepackage[all, cmtip]{xy}
\usepackage{amscd}
\usepackage{lscape,xcolor}
\usepackage{graphicx}
\usepackage{hyperref}
\usepackage{stmaryrd}

\def \mmod{/\mkern-3mu /}
\def \MRlevel3{MRlevel3}

\def \MMJDetectSq4{MMJDetectSq4}
\def \BrownGitler{BrownGitler}
\def \CohenZeta{CohenZeta}
\def \MahowaldInf{MahowaldInf}
\def \Mahowaldbor{boresolutions}

\def \BHHM{BHHM}

\allowdisplaybreaks[1]

\newcommand{\mc}[1]{\mathcal{#1}}
\newcommand{\ul}[1]{\underline{#1}}

\newcommand{\mr}[1]{\mathrm{#1}}
\newcommand{\mbf}[1]{\mathbf{#1}}

\newcommand{\abs}[1]{\lvert #1 \rvert}

\newcommand{\bra}[1]{\langle #1 \rangle}
\newcommand{\br}[1]{\overline{#1}}

\newcommand{\td}[1]{\widetilde{#1}}

\newcommand{\ZZ}{\mathbb{Z}}

\newcommand{\QQ}{\mathbb{Q}}

\newcommand{\FF}{\mathbb{F}}
\newcommand{\GG}{\mathbb{G}}
\newcommand{\MS}{\mathbb{S}}

\def \Z3{\mathbb{Z}_{(3)}}
\def \Z{\mathbb{Z}}
\def \F{\mathbb{F}}
\def \Q{\mathbb{Q}}
\def \N{\mathbb{N}}
\def \W{\mathbb{W}}

\def \Einf {\mathbf{E}_{\infty}} 

\def \ba{\bar a}
\def \bc{\bar c}

\newcommand{\AF}{\mr{AF}}

\newcommand{\TMF}{\mathrm{TMF}}
\newcommand{\Tmf}{\mathrm{Tmf}}
\newcommand{\tmf}{\mathrm{tmf}}
\newcommand{\bo}{\mathrm{bo}}
\newcommand{\bsp}{\mathrm{bsp}}
\newcommand{\HZ}{\mr{H}\Z}
\def \HF2{\mr{H}\FF_2}
\newcommand{\bu}{\mr{bu}}
\newcommand{\MU}{\mr{MU}}
\newcommand{\KU}{\mr{KU}}
\newcommand{\al}{\alpha}
\newcommand{\KO}{\mr{KO}}

\newcommand{\BP}{\mr{BP}}
\newcommand{\K}{\mr{K}}
\newcommand{\E}{\mr{E}}

\newcommand{\bou}{\ul{\bo}}
\newcommand{\tmfu}{\ul{\tmf}}

\def \M{\mathcal{M}}
\def \Mss{\mathcal{M}^{ss}}
\def \Mfg{\mathcal{M}_{fg}}
\def \Htwo{\hat{\mathcal{H}}(2)}

\def \G {\mathbb{G}}

\newcommand{\Sh}[1]{\mathcal{#1}}

\newcommand{\Mcl}[1]{\mathcal{M}_0(#1)}
\newcommand{\Msl}[1]{\mathcal{M}_1(#1)}

\theoremstyle{definition}

 \newtheorem{thm}[equation]{Theorem}
 \newtheorem{cor}[equation]{Corollary}
 \newtheorem{lem}[equation]{Lemma}
 \newtheorem{prop}[equation]{Proposition}
 \newtheorem{defn}[equation]{Definition}
 \newtheorem{ex}[equation]{Example}
 
 \newtheorem{rmk}[equation]{Remark}

\newtheorem*{thm*}{Theorem}
\newtheorem*{cor*}{Corollary}
\newtheorem*{lem*}{Lemma}
\newtheorem*{prop*}{Proposition}
\newtheorem*{defn*}{Definition}
\newtheorem*{ex*}{Example}
\newtheorem*{exs*}{Examples}
\newtheorem*{rmk*}{Remark}
\newtheorem*{claim*}{Claim}

\numberwithin{equation}{section}
\numberwithin{figure}{section}
\DeclareMathOperator{\Ext}{Ext}
\DeclareMathOperator{\Hom}{Hom}
\DeclareMathOperator{\aut}{Aut}

\DeclareMathOperator{\Map}{Map}
\DeclareMathOperator*{\holim}{holim}

\DeclareMathOperator{\Spf}{Spf}
\DeclareMathOperator{\Aut}{Aut}
\DeclareMathOperator{\Spec}{Spec}
\DeclareMathOperator{\Proj}{Proj}

\newcommand{\s}{\wedge}
\newcommand{\Bigvee}[1]{\underset{#1}{\bigvee}}
\newcommand{\lra}[1]{\overset{#1}{\longrightarrow}}
\DeclareMathOperator{\sq}{Sq}
\newcommand{\si}[1]{\Sigma^{#1}}
\newcommand{\xib}{{\bar{\xi}}}

%% file: tmfcoop_intro.tex

\section{Introduction}\label{sec:intro}

The Adams-Novikov spectral sequence based on a connective spectrum $ E $ ($ E $-ANSS) is perhaps the best available tool for computing stable homotopy groups. For example, $ H\F_p $ and $\BP$ give the classical Adams spectral sequence and the Adams-Novikov spectral sequence, respectively.

To begin to compute with the $E$-ANSS, one needs to know the structure of the smash powers $ E^{\wedge k} $. When $ E $ is one of $ H\F_p $, $ \mathrm{MU} $, or $ \BP $, the situation is simpler than in general, since in this case $ E\wedge E $ is an infinite wedge of suspensions of $ E $ itself, which allows for an algebraic description of the $ E_2 $-term. This is not the case for $ \bu, \bo, $ or $ \tmf $, in which case the $ E_2 $ page is harder to describe, and in fact, has not yet been described in the case of $ \tmf $.

Mahowald and his collaborators have studied the $2$-primary $ \bo $-ANSS to great effect: it gives the only known approach to the calculation of the telescopic $2$-primary $ v_1 $-periodic homotopy in the sphere spectrum \cite{LellmannMahowald,boresolutions}. The starting input in that calculation is a complete description of $ \bo\wedge \bo $ as an infinite wedge of spectra, each of which is a smash product of $\bo$ with a suitable finite complex (as in \cite{Milgram-connectivektheory} and others). The finite complexes involved are the so-called integral Brown-Gitler spectra.  (See also the related work of \cite{ClarkeCrossleyWhitehouse1,ClarkeCrossleyWhitehouse2, MR2434436}.)

Mahowald has worked on a similar description for $ \tmf\wedge \tmf $, but concluded that no analogous result could hold. In this paper we use his insights to explore four different perspectives on $2$-primary $\tmf$-cooperations. While we do not arrive at a complete and closed-form description of $\tmf \wedge \tmf$, we believe our results have the potential to be very useful as a computational tool.  

The four perspectives are the following.
\begin{enumerate}
\item The $E_2$ term of the $2$-primary Adams spectral sequence for $\tmf \wedge \tmf$ admits a splitting in terms of $\bo$-Brown-Gitler modules:
$$ \Ext(\tmf \wedge \tmf) \cong \bigoplus_i \Ext(\Sigma^{8i} \tmf \wedge \bo_i). $$
\item Modulo torsion, $\TMF_*\TMF$ is isomorphic to a subring of the ring of integral two variable modular forms.
\item $K(2)$-locally, the ring spectrum $(\TMF \wedge \TMF)_{K(2)}$ is given by an equivariant function spectrum:
$$ (\TMF \wedge \TMF)_{K(2)} \simeq \Map^c(\GG_2/G_{48}, E_2)^{hG_{48}}. $$
\item By our Theorem \ref{thm:faithful}, $\TMF_*\TMF$ injects into a certain product of homotopy groups of topological modular forms with level structures:  
$$ \TMF \wedge \TMF \hookrightarrow \prod_{\substack{i \in \ZZ, \\ j \ge 0}} \TMF_0(3^j) \times \TMF_0(5^j). $$
\end{enumerate}
The purpose of this paper is to describe and investigate the relationship between these different perspectives. As an application of our method, in Theorems \ \ref{thm:tdtmf03} and  \ref{thm:tdtmf05} we construct connective covers $\td{\tmf}_0(3)$  and $\td{\tmf}_0(5)$ of the periodic spectra $\TMF_0(3)$ and $\TMF_0(5)$, respectively, recovering and extending previous results of Davis, Mahowald, and Rezk \cite{MRlevel3}, \cite{MahowaldConnective}.

Others have also investigated the ring of cooperations for elliptic cohomology.  Clarke and Johnson \cite{MR1232203} conjectured that $\TMF_0(2)_*\TMF_0(2)$ was given by the ring of $2$-variable modular forms for $\Gamma_0(2)$ over $\ZZ[1/2]$.  Versions of this conjecture were subsequently verified by Baker \cite{MR1307488} (in the case of $\TMF[1/6]$) and Laures \cite{Laures} (for all $\TMF(\Gamma)[S^{-1}]$ associated to congruence subgroups, where $S$ is a large enough set of primes to make the theory Landweber exact).  This previous work clearly feeds into perspective (2) (indeed Laures' work is cited as an initial step to establishing perspective (2)).  In retrospect, Baker's work also contains observations related to perspective (4): in \cite{MR1307488} he observes that the ring of $2$-variable modular forms can be regarded as a certain space of functions on a space of isogenies of elliptic curves.  Finally, since the writing of this paper, Culver has produced a similar (but more complete) analysis of $\tmf_1(3)_*\tmf_1(3)$ (a.k.a. $BP\bra{2}_*BP\bra{2}$) \cite{Culver}.

\subsection{A tour of the paper}
For the reader's convenience, we take some time here to outline the contents of the paper.

\subsection*{Section 2} This section reviews Brown-Gitler comodules and Brown-Gitler spectra, splittings associated to these, and exact sequences which relate the various comodules.

Sections 2.1-2.2 begin with a review of mod 2, integral, and $\bo$-Brown-Gitler spectra.  Our interest stems from the fact (Section 2.3) that the $E_2$-term of the Adams spectral sequence for $\bo \wedge \bo$ (respectively $\tmf\wedge\tmf$) splits as a direct sum of $\Ext$-groups for the integral (respectively $\bo$) Brown-Gitler spectra.  Section 2.4 recalls some exact sequences used in \cite{BHHM} which allow for an inductive approach for computing $\Ext$ of $\bo$-Brown-Gitler comodules, and introduces related sequences which allow for an inductive approach to $\Ext$ groups of integral Brown-Gitler comodules.  

\subsection*{Section 3} This section is devoted to the motivating example of $\bo\wedge\bo$.  Sections 3.1-3.3 are primarily expository, based upon the foundational work of Adams, Lellmann, Mahowald, and Milgram.  We make an effort to consolidate their theorems and recast them in modern notation and terminology, and hope that this will prove a useful resource to those trying to learn the classical theory of $\bo$-cooperations and $v_1$-periodic stable homotopy.  To the best of our knowledge, Sections 3.4-3.5 provide new perspectives on this subject.

Section 3.1 is devoted to the homology of the $\HZ_i$ and certain $\Ext_{A(1)_*}$-computations relevant to the Adams spectral sequence computation of $\bo_*\bo$.

We shift perspectives in Section 3.2 and recall Adams's description of $\KU_*\KU$ in terms of numerical polynomials.  This allows us to study the image of $\bu_*\bu$ in $\KU_*\KU$ as a warm-up for our study of the image of $\bo_*\bo$ in $\KO_*\KO$.  

We undertake this latter study in Section 3.3, where we ultimately describe a basis of $\KO_0\bo$ in terms of the ``9-Mahler basis" for 2-adic numerical polynomials with domain $2\mathbb{Z}_2$.  By studying the Adams filtration of this basis, we are able to use the above results to fully describe $\bo_*\bo$ mod $v_1$-torsion elements.

In Section 3.4, we link the above two perspectives, studying the image of $\bo_*\HZ_i$ in $\KO_*\KO$. Theorem \ref{thm:HZjImage} provides a complete description of this image (mod $v_1$-torsion) in terms of the 9-Mahler basis.

We conclude with Section 3.5 which studies a certain map
\[
  \KO\wedge\KO \xrightarrow{\prod \widetilde{\psi}^{3^k}} \prod_{k\in \ZZ} \KO
\]
constructed from Adams operations.  We show that this map is an injection after applying $\pi_*$ and exhibit how it interacts with the Brown-Gitler decomposition of $\bo\wedge\bo$.

\subsection*{Section 4} In Section 4, we recall certain essential features of $\TMF$ and $\tmf$, the periodic and connective topological modular forms spectra.  

Section 4.1 reviews the Goerss-Hopkins-Miller sheaf of $E_\infty$-ring spectra, $\mathcal{O}^{top}$, on the moduli stack of smooth elliptic curves $\mathcal{M}$.  One can use this sheaf to construct $\TMF$ (sections on $\mathcal{M}$ itself), $\TMF_1(n)$ (sections on the moduli stack of $\Gamma_1(n)$-structures after inverting $n$), and $\TMF_0(n)$ (sections on the moduli stack of $\Gamma_0(n)$-structures after inverting $n$).  We consider the maps
\[
  f,q:\TMF[1/n]\to \TMF_0(n)
\]
induced by forgetting the level structure and taking the quotient by it, respectively. We use these maps to produce a $\TMF[1/n]$-module map
\[
  \Psi_n:\TMF[1/n]\wedge \TMF[1/n]\to \TMF_0(n)
\]
important in our subsequent studies.

Section 4.2 reviews Lawson and Naumann's work on the construction of $\BP\langle 2\rangle$ as the $E_\infty$-ring spectrum $\tmf_1(3)$.  We use formal group laws and some computer calculations to compute the maps
\[
  \BP_*\to \tmf_1(3)_*,\quad \BP_*\BP\to \tmf_1(3)_*\tmf_1(3).
\]

We isolate the lowest Adams filtration portion of this map in Section 4.4 via our computation of  $\pi_* f:\TMF_*\to \TMF_1(3)_*$ in Section 4.3.

Finally, we review the $K(2)$-local version of $\TMF\wedge\TMF$ in Section 4.5.

\subsection*{Section 5} With the stage set, our work begins in earnest in Section 5.  Here we study the Adams spectral sequence for $\tmf\wedge\tmf$.  

We study the rational behavior of this spectral sequence in Section 5.1, observing that it collapses after inverting $v_0$. This provides a precise computation of the map
\[
  v_0^{-1}\Ext(\tmf\wedge \Sigma^{8j}\bo_j)\to v_0^{-1}\Ext(\tmf\wedge\tmf).
\]

In Section 5.2, the exact sequences of Section 2.4 are used to perform an inductive computation of $\Ext(\tmf\wedge\Sigma^{8j}\bo_j)$ relative to $\Ext(\tmf\wedge\bo_1^{\wedge k})$.  We produce detailed charts for $\Ext(\tmf\wedge\Sigma^{8j}\bo_j)$ for $j\le 6$.

Sections 5.3 and 5.4 are concerned with identifying the generators of the lattice
\[
  \Ext(\tmf\wedge\Sigma^{8j}\bo_j)/v_0\text{-torsion}
\]
inside of the ``vector space''
\[
  v_0^{-1}\Ext(\tmf\wedge\Sigma^{8j}\bo_j).
\]
In Section 5.3, we produce an inductive method compatible with the exact sequences of Section 2.4.  Section 5.4 completes the task of computing said generators.

\subsection*{Section 6} In Section 6, we study the role of 2-variable modular forms in $\tmf$-cooperations.  Baker and Laures have proved that, after inverting 6, $\TMF$-cooperations are precisely the 2-variable $\Gamma(1)$ modular forms (meromorphic at the cusp).  

After reviewing the Baker-Laures work in Section 6.1, we adapt it to the study of $\TMF_*\TMF$ modulo torsion in Section 6.2.  In particular, we prove that 2-integral 2-variable $\Gamma(1)$-modular forms (again meromorphic at the cusp) are exactly the $0$-line of a descent spectral sequence for $\TMF_*\TMF$.

The efficacy of this result becomes apparent in Section 6.3 where we prove that $\tmf_*\tmf$ modulo torsion injects into the ring of 2-integral 2-variable modular forms with nonnegative Adams filtration.  Moreover, the injection is a rational isomorphism; once again we are primed to identify the generators of a lattice inside a vector space.

Sections 6.4 and 6.5 undertake the task of detecting 2-variable modular forms in the Adams spectral sequence for $\tmf\wedge\tmf$, resulting in a table of 2-variable modular form generators of $\Ext(\tmf\wedge\tmf)/\text{torsion}$ in dimensions $\le 64$.

\subsection*{Section 7}  Our final section studies the level structure approximation map
\[
  \Psi:\tmf\wedge\tmf\to \prod_{i\in\ZZ,j\ge 0} \TMF_0(3^j)\times \TMF_0(5^j).
\]
The first theorem of Section 7.1 is that the analogous map
\[
  \psi:\TMF\wedge\TMF\to \prod_{i\in\ZZ,j\ge 0} \TMF_0(3^j)\times \TMF_0(5^j)
\]
induces an injection on homotopy groups. The proof is quite involved. It includes a reduction to a $K(2)$-local variant of the theorem, whose proof in turn requires the key technical Lemma \ref{lem:faithful} on detecting homotopy fixed points of profinite groups using dense subgroups.  

In Section~\ref{sec:psicomp}, we compute the effect of the maps
\begin{align*}
\Psi_3: \pi_* \tmf \wedge \tmf & \rightarrow \pi_* \TMF_0(3), \\
\Psi_5: \pi_* \tmf \wedge \tmf & \rightarrow \pi_* \TMF_0(5) \\
\end{align*}
on a certain submodules of $\pi_*\tmf \wedge \tmf$.  

In Section~\ref{sec:diffext}, we observe that these computations allow us to deduce differentials and hidden extensions in the corresponding portion of the ASS for $\tmf \wedge \tmf$ using the known homotopy of $\TMF_0(3)$ and $\TMF_0(5)$.  

Davis, Mahowald, and Rezk \cite{MRlevel3}, \cite{MahowaldConnective} observed that one can build a connective cover 
$$ \td{\tmf}_0(3) \rightarrow \TMF_0(3) $$
out of $\tmf \wedge \bo_1$ and a piece of $\tmf \wedge \bo_2$.  In Section~\ref{sec:cover}, we reprove this result, and relate this connective cover to our map $\Psi_3$.  We also show that similar methods allow us build a connective cover
$$ \td{\tmf}_0(5) \rightarrow \TMF_0(5) $$
out of the other part of $\tmf \wedge \bo_2$, $\tmf \wedge \bo_3$, and a piece of $\tmf \wedge \bo_4$.  Note that neither $\td{\tmf}_0(3)$ nor $\td{\tmf}_0(5)$ coincide with the Hill-Lawson connective covers of $\TMF_0(N)$.  The connective covers we consider are $7$-connected and $23$-connected, respectively.

\subsection{Acknowledgments}  We would like to thank the referee, for their helpful suggestions and comments.  We dedicate this paper to the memory of Mark Mahowald, whose visionary ideas permeate almost every aspect of this work.  We also would like to tip our hat to Andy Baker, Don Davis, Gerd Laures, and Charles Rezk, whose previous work on $\tmf$-cooperations and related matters have made this paper possible.

\subsection{Notation and conventions}

In this paper, unless we say explicitly otherwise, we shall always be implicitly working $2$-locally.  We denote homology by $H_*$, and it  will be taken with mod $2$ coefficients, unless specified otherwise.  We let $A = H^*H$ denote the mod 2 Steenrod algebra, and
\[
  A_*=H_*H\cong \FF_2[\xi_1,\xi_2,\ldots]
\]
denotes its dual.  In any Hopf algebra, we let $\overline{x}$ denote the antipode of $x$.  We let $A(i)$ denote the subalgebra of $A$ generated by $\mathrm{Sq}^1,\cdots,\mathrm{Sq}^{2^i}$.  Let $A\mmod A(i)$ be the Hopf algebra quotient of $A$ by $A(i)$ and let $(A\mmod A(i))_*$ be the dual of this Hopf algebra.

We will use $\Ext(X)$ to abbreviate $\Ext_{A_*}(\FF_2, H_* X)$, the $E_2$-term of the Adams spectral sequence (ASS) for $\pi_* X$ and will let $C^*_{A_*}(H^*X)$ denote the corresponding cobar complex.  Given an element $x \in \pi_* X$, we shall let $[x]$ denote the coset of the ASS $E_2$-term which detects $x$.  We let $AF(x)$ denote the Adams filtration of $x$.

We write $\bu$ for the connective complex $\K$-theory spectrum, $\bo$ for the connective real $K$-theory spectrum, and $\bsp$ for the connective symplectic $K$-theory spectrum, so that $\Sigma^4 \bsp$ is the $3$-connected cover of $\bo$.

%% file: tmfcoop_bo.tex

\section{Brown-Gitler comodules and spectra}

Mod $2$ Brown-Gitler spectra were introduced in \cite{\BrownGitler} to study obstructions to immersing manifolds, but immediately found use in studying the stable homotopy groups of spheres (e.g. \cite{\MahowaldInf}, \cite{\CohenZeta} and many other places).  Mahowald, Milgram, and others have used integral Brown-Gitler modules/spectra to decompose the ring of cooperations of $\bo$ \cite{boresolutions}, \cite{Milgram-connectivektheory}, and much of the work of Davis, Mahowald, and Rezk on $\tmf$-resolutions has been based on the use of  $\bo$-Brown-Gitler spectra \cite{MRlevel3},\cite{MahowaldConnective}, \cite{BHHM}.  In this section we collect the things about Brown-Gitler comodules and spectra that pertain to the subject matter of this paper.

\subsection{Brown-Gitler comodules}

Consider the subalgebras of the dual Steenrod algebra
\[ (A \mmod  A(i))_* =\F_2[\bar\xi_1^{2^{i+1}},\bar\xi_2^{2^i},\dots,\bar\xi_{i+1}^2,\bar\xi_{i+2},\dots]. \]
The first few of these arise as the homology of spectra:
\begin{align*}
H_* \HF2 & \cong A_* = (A\mmod A(-1))_*, \\
H_* \HZ & \cong (A\mmod A(0))_*, \\
H_* \bo & \cong (A\mmod A(1))_*, \\
H_* \tmf & \cong (A\mmod A(2))_*.
\end{align*}
The algebra $(A\mmod A(i))_*$ admits an increasing filtration 
by defining $ wt(\bar\xi_j)=2^{j-1} $; every element has filtration divisible by $ 2^{i+1} $. The Brown-Gitler subcomodule $ N_i(j) $ is defined to be the $\F_2$-subspace spanned by all monomials of weight  less than or equal to $ 2^{i+1}j $, which is also an $ A_* $-subcomodule as the coaction cannot increase weight. 

\subsection{Brown-Gitler spectra}

The modules $ N_{-1}(j) $ through $ N_1(j) $ are known to be realizable by the mod-$ 2 $ (classical), integral, and $ \bo $-Brown-Gitler spectra respectively \cite{GoerssJonesMahowald}, which we will denote by $ (H\F_2)_j $, $ H\Z_j $, and $ \bo_j $, since we have
\begin{align*}
\HF2 & \simeq \varinjlim (\HF2)_j, \\
\HZ & \simeq \varinjlim \HZ_j, \\
\bo & \simeq \varinjlim \bo_j.
\end{align*}
To clarify notation we shall underline a spectrum to refer to the corresponding subcomodule of the dual Steenrod algebra, so that we have
\begin{align*}
(\ul{\HF2})_j & := H_* (\HF2)_j = N_{-1}(j), \\
\ul{\HZ}_j & := H_* \HZ_j = N_0(j), \\
\ul{\bo}_j & := H_* \bo_j = N_1(j).
\end{align*}
It is not known if $\tmf$-Brown-Gitler spectra $\tmf_j$ exist in general, though we will still define
 $$ \ul{\tmf}_j := N_2(j).  $$ 
The spectrum $N_3(1)$ is not realizable, by the Hopf-invariant one theorem.  

\subsection{Algebraic and topological splittings}

There are algebraic splittings of $A(i)_*$-comodules
\[ (A\mmod A(i))_* \cong \bigoplus_j \Sigma^{2^{i+1}j}N_{i-1}(j). \]
This splitting is given by the sum of maps:
\begin{equation}\label{eq:splittingmap}
\begin{split}
\Sigma^{2^{i+1}j}N_{i-1}(j) & \rightarrow  (A\mmod A(i))_* \\
\xib_1^{i_1} \xib_2^{i_2}\cdots & \mapsto  \xib_1^{a} \xib_2^{i_1} \xib_3^{i_2} \cdots,
\end{split}
\end{equation}
where the exponent $a$ above is chosen such that the monomial has weight $2^{i+1}j$.
It follows that there are algebraic splittings
\begin{align}
\Ext(\HZ \wedge \HZ) & \cong \bigoplus \Ext_{A(0)_*}(\si{2j} (\HF2)_j), \\
\Ext(\bo \wedge \bo) & \cong \bigoplus \Ext_{A(1)_*}(\si{4j} \HZ_j), \label{eq:boextsplit}
\\
\Ext(\tmf \wedge \tmf) & \cong \bigoplus \Ext_{A(2)_*}(\si{8j} \bo_j). \label{eq:tmfextsplit}
\end{align}
These algebraic splittings can be realized topologically for $i\leq 1$ \cite{\Mahowaldbor}:
\begin{align*}
\HZ\wedge \HZ & \simeq \bigvee_{j} \Sigma^{2j} \HZ \wedge (\HF2)_j ,\\
\bo\wedge \bo &\simeq \bigvee_{j} \Sigma^{4j} \bo \wedge \HZ_j.
\end{align*}
However, the corresponding splitting fails for $\tmf$ as was shown by Davis, Mahowald, and Rezk \cite{MRlevel3}, \cite{MahowaldConnective}, so
$$ \tmf \wedge \tmf \not\simeq \bigvee_j \Sigma^{8j} \tmf \wedge \bo_j. $$
Indeed, they observe that in $\tmf \wedge \tmf$ the homology summands
$$ \Sigma^8 \tmf \wedge \bo_1, \quad \text{and} \quad \Sigma^{16} \tmf \wedge \bo_2 $$ 
are attached non-trivially.  We shall see in Section~\ref{sec:cover} that our methods recover this fact.

\subsection{Short exact sequences relating Brown-Gitler comodules}

The $E_2$ terms of the Adams spectral sequences
\begin{gather*}
\Ext_{A(1)_*}((A\mmod A(1))_*) \Rightarrow \bo_*\bo \\
\Ext_{A(2)_*}((A\mmod A(2))_*) \Rightarrow \tmf_*\tmf
\end{gather*}
split, by (\ref{eq:boextsplit}), (\ref{eq:tmfextsplit}) into summands of the form
\begin{align*}
& \Ext_{A(1)_*}(\ul{\HZ}_j), \\
& \Ext_{A(2)_*}(\bou_j).
\end{align*}
It is therefore desirable to compute the above Ext groups.  In \cite{\BHHM}, this is accomplished inductively by means of a certain exact sequences relating these Brown-Gitler comodules (Lemmas 7.1 and 7.2 of \cite{BHHM}).  

We begin by pointing out that a similar set of exact sequences interrelates the integral Brown-Gitler comodules.

\begin{lem}\label{lem:HZSES}
There are short exact sequences of $A(1)_*$-comodules
\begin{gather}
0 \rightarrow \Sigma^{4j}\ul{\HZ}_j \rightarrow \ul{\HZ}_{2j} \rightarrow \ul{\bo}_{j-1} \otimes (A(1)\mmod A(0))_* \rightarrow 0, 
\label{eq:HZSES1} \\
0 \rightarrow \Sigma^{4j}\ul{\HZ}_j \otimes \ul{\HZ}_1 \rightarrow \ul{\HZ}_{2j+1} \rightarrow \ul{\bo}_{j-1} \otimes (A(1)\mmod A(0))_* \rightarrow 0. \label{eq:HZSES2} 
\end{gather}
\end{lem}

\begin{proof}
The proof is almost identical to that given in \cite{BHHM}.  On the level of basis elements, the map
\begin{gather*}
\si{4j} \ul{\HZ}_j \rightarrow \ul{\HZ}_{2j}
\end{gather*}
is given by
\begin{align*}
\xib_1^{2i_1} \xib_2^{i_2} \cdots & \mapsto \xib_1^a \xib_2^{2i_1} \xib_3^{i_2} \cdots,
\end{align*}
where the integer $a$ is $4j - wt(\xib_2^{2i_1} \xib_3^{i_2} \cdots)$. The map 
\begin{gather*}
\si{4j} \ul{\HZ}_j \otimes \ul{\HZ}_1 \rightarrow \ul{\HZ}_{2j+1}
\end{gather*}
is determined by
\begin{align*}
\xib_1^{2i_1} \xib_2^{i_2} \cdots \otimes  1 & \mapsto (\xib_1^a \xib_2^{2i_1} \xib_3^{i_2} \cdots) \cdot 1, \\
\xib_1^{2i_1} \xib_2^{i_2} \cdots \otimes  \xib_1^2 & \mapsto (\xib_1^a \xib_2^{2i_1} \xib_3^{i_2} \cdots) \cdot \xib_1^2, \\
\xib_1^{2i_1} \xib_2^{i_2} \cdots \otimes   \xib_2 & \mapsto (\xib_1^a \xib_2^{2i_1} \xib_3^{i_2} \cdots) \cdot \xib_2.
\end{align*}
We abbreviate this by writing 
\begin{align}\label{eq:basismap}
\xib_1^{2i_1} \xib_2^{i_2} \cdots \otimes  \{ 1, \xib_1^2, \xib_2 \}& \mapsto (\xib_1^a \xib_2^{2i_1} \xib_3^{i_2} \cdots) \cdot \{1, \xib_1^2, \xib_2 \}.
\end{align}
The maps
\begin{gather*}
\ul{\HZ}_{2j} \rightarrow \ul{\bo}_{j-1} \otimes (A(1)\mmod A(0))_* , \\
\ul{\HZ}_{2j+1} \rightarrow \ul{\bo}_{j-1} \otimes (A(1)\mmod A(0))_*
\end{gather*}
are given by
$$ \xib_1^{4i_1+2\epsilon_1} \xib_2^{2i_2+\epsilon_2} \xib_3^{i_3} \cdots \mapsto
\begin{cases}
\xib_1^{4i_1} \xib_2^{2i_2} \xib_3^{i_3} \cdots \otimes \xib_1^{2\epsilon_1} \xib_2^{\epsilon_2}, & 
wt(\xib_1^{4i_1} \xib_2^{2i_2} \xib_3^{i_3} \cdots) \le 4j-4, \\
0, & \mr{otherwise},
\end{cases}
 $$
where $\epsilon_s \in \{0,1\}$. The proof is now a direct computation.
\end{proof}

The exact sequences of \cite{\BHHM} which relate the different $\bo$-Brown-Gitler comodules take the form:
\begin{gather}
0\to \Sigma^{8j} \ul{\bo}_j \to \ul{\bo}_{2j}\to (A(2)\mmod A(1))_* \otimes \ul{\tmf}_{j-1}  \to \Sigma^{8j+9} \ul{\bo}_{j-1} \to 0 \label{eq:boSES1},
\\
0 \to \Sigma^{8j} \ul{\bo}_j \otimes \ul{\bo}_1 \to \ul{\bo}_{2j+1}\to (A(2)\mmod A(1))_* \otimes \ul{\tmf}_{j-1} \to 0 \label{eq:boSES2}
\end{gather}

\begin{rmk}
Technically speaking, as is addressed in \cite{\BHHM}, the comodules $(A(2)\mmod A(1))_* \otimes \ul{\tmf}_{j-1}$ in the above exact sequences have to be given a slightly different $A(2)_*$-comodule structure from the standard one arising from the tensor product.  However, this different comodule structure ends up being $\Ext$-isomorphic to the standard one.  As we are only interested in Ext groups, the reader can safely ignore this subtlety.
\end{rmk}

We briefly recall how the maps in the exact sequences (\ref{eq:boSES1}) and (\ref{eq:boSES2}) are defined.
On the level of basis elements, the maps
\begin{gather*}
\si{8j} \ul{\bo}_j \rightarrow \ul{\bo}_{2j}, \\
\si{8j} \ul{\bo}_j \otimes \ul{\bo}_1 \rightarrow \ul{\bo}_{2j+1}
\end{gather*}
are given respectively by
\begin{align*}
\xib_1^{4i_1} \xib_2^{2i_2} \xib_3^{i_3} \cdots & \mapsto \xib_1^a \xib_2^{4i_1} \xib_3^{2i_2} \xib_4^{i_3} \cdots,
\\
\xib_1^{4i_1} \xib_2^{2i_2} \xib_3^{i_3} \cdots \otimes  \{ 1, \xib_1^4, \xib_2^2, \xib_3 \} & \mapsto (\xib_1^a \xib_2^{4i_1} \xib_3^{2i_2} \xib_4^{i_3} \cdots) \cdot \{1, \xib_1^4, \xib_2^2, \xib_3 \},
\end{align*}
where $a$ is taken to be $8j - wt(\xib_2^{4i_1} \xib_3^{2i_2} \xib_4^{i_3} \cdots)$. Here, we are using the notation introduced in \eqref{eq:basismap}.
The maps
\begin{gather}
\ul{\bo}_{2j} \rightarrow (A(2)\mmod A(1))_* \otimes \ul{\tmf}_{j-1}  , \label{eq:boSESmap1} \\
\ul{\bo}_{2j+1} \rightarrow (A(2)\mmod A(1))_* \otimes \ul{\tmf}_{j-1} \label{eq:boSESmap2}
\end{gather}
are given by
\begin{multline*}
\xib_1^{8i_1+4\epsilon_1} \xib_2^{4i_2+2\epsilon_2} \xib_3^{2i_3 + \epsilon_3} \xib_4^{i_4}\cdots \mapsto
\\
\begin{cases}
\xib_1^{8i_1} \xib_2^{4i_2} \xib_3^{2i_3} \xib_4^{i_4} \cdots \otimes \xib_1^{4\epsilon_1} \xib_2^{2\epsilon_2} \xib_3^{\epsilon_3}, & 
wt(\xib_1^{8i_1} \xib_2^{4i_2} \xib_3^{2i_3} \xib_4^{i_4} \cdots) \le 8j-8, \\
0, & \mr{otherwise},
\end{cases}
\end{multline*}
where $\epsilon_s \in \{0,1\}$.  The only change from the integral Brown-Gitler case is that while the map (\ref{eq:boSESmap2}) is surjective, the map (\ref{eq:boSESmap1}) is not.  The cokernel is spanned by the submodule
$$ \FF_2\{ \xib_1^4 \xib_2^2 \xib_3 \} \otimes \Sigma^{8j-8} \bou_{j-1} \subset  (A(2)\mmod A(1))_* \otimes \ul{\tmf}_{j-1}. $$
We therefore have an exact sequence
$$ \ul{\bo}_{2j} \rightarrow (A(2)\mmod A(1))_* \otimes \ul{\tmf}_{j-1} \rightarrow \Sigma^{8j+9} \bou_{j-1} \rightarrow 0. $$

\begin{rmk}
The authors do not know if analogues of the exact sequences (\ref{eq:HZSES1}), (\ref{eq:HZSES2}), (\ref{eq:boSES1}), (\ref{eq:boSES2}) exist in general the Brown-Gitler comodules $N_i(j)$.  Culver \cite{Culver} constructs analogs of these in the context of $BP\bra{2}$-cooperations in \cite{Culver}.
\end{rmk}


\section{Motivation: analysis of $\bo_*\bo$}\label{sec:bo}

In analogy with the four perspectives described in the introduction, there are four primary perspectives on the ring of cooperations for real $K$-theory.
\begin{enumerate}
\item There is a decomposition (at the prime $2$)
\[
\bo \s \bo \simeq \Bigvee{j \ge 0} \Sigma^{4j} \bo \s \HZ_j,
\]
\item There is an isomorphism $\KO_*\KO \cong \KO_* \otimes_{\KO_0} \KO_0\KO$, and $\KO_0\KO$ is isomorphic to a subring of the ring of numerical functions.
\item $K(1)$-locally, the ring spectrum $(\KO \wedge \KO)_{K(1)}$ is given by the function spectrum
$$ (\KO \wedge \KO)_{K(1)} \simeq \Map(\ZZ_2^\times/\{\pm 1\}, \KO^\wedge_2). $$
\item By evaluation on Adams operations, $\KO_*\KO$ injects into a product of copies of $\KO$:  
$$ \KO \wedge \KO \hookrightarrow \prod_{i \in \ZZ} \KO. $$
\end{enumerate}
In this section we will recall results from the literature which relate these four perspectives.  Our discussion of $\bo_*\bo$ will frame our subsequent treatment of $\tmf_*\tmf$.


\subsection{The Adams spectral sequence for $\bo_*\bo$}

In this section we will compute the Adams spectral sequence
\begin{equation}\label{eq:boboASS}
 \Ext_{A(1)_*}((A \mmod A(1))_*) \Rightarrow \bo_*\bo. 
 \end{equation}
The splitting (\ref{eq:boextsplit}) reduces the understanding of the Adams $E_2$-term for $\bo\wedge \bo$ to an understanding of $\Ext_{A(1)_*}(\ul{\HZ}_j)$.

Define
$$ \frac{\Ext_{A(1)_*}(X)}{\text{$v_1$-{tor}}} := \mr{Image}\left(\Ext_{A(1)_*}(X) \rightarrow v_1^{-1}\Ext_{A(1)_*}(X)\right). $$
The following lemma follows from a simple induction (for instance, using the algebraic Atiyah-Hirzebruch spectral sequence), using the fact that $\ul{\HZ}_1$ is given by the following cell diagram.
$$
\xymatrix@R-2em@C-2em{
\xib_2 & \circ \ar@{-}[d]^{\sq^1}
\\
\xib_1^2 & \circ \ar@/^1pc/@{-}[dd]^{\sq^2}
\\ \\
1 & \circ
}
$$
\begin{lem}\label{lem:HZ1^i}
We have 
$$
\frac{\Ext_{A(1)_*}(\ul{\HZ}_1^{\otimes i})}{\text{$v_1$-tor}} \cong
\begin{cases}
\Ext(\bo^{\bra{i}}), & \text{$i$ even}, \\
\Ext(\bsp^{\bra{i-1}}), & \text{$i$ odd}.
\end{cases}
$$
Here, $X^{\bra{i}}$ denotes the $i$th Adams cover.
\end{lem}

We deduce the following well known result (cf. \cite[Thm.~2.1]{LellmannMahowald}).

\begin{prop}\label{lem:HZ_i}
For a non-negative integer $j$, denote by $\alpha(j)$ the number of 1's in the dyadic expansion of $j$. Then
$$
\frac{\Ext_{A(1)_*}(\ul{\HZ}_j)}{\text{$v_1$-tor}} \cong
\begin{cases}
\Ext(\bo^{\bra{2j-\alpha(j)}}), & \text{$j$ even}, \\
\Ext(\bsp^{\bra{2j-\alpha(j)-1}}), & \text{$j$ odd}.
\end{cases}
$$
\end{prop}

\begin{proof}
This may be established by induction on $j$ using the short exact sequences of Lemma~\ref{lem:HZSES}, by augmenting Lemma~\ref{lem:HZ1^i} with the following facts.
\begin{enumerate}
\item All $v_0$-towers in $\Ext_{A(1)_*}(\ul{\HZ}_i)$ are $v_1$-periodic.  This can be seen as $\Ext_{A(1)_*}(\ul{\HZ}_j)$ is a summand of $\Ext(\bo \wedge \bo)$, and after inverting $v_0$, the latter has no $v_1$-torsion.  Explicitly we have
$$ v_0^{-1} \Ext(\bo \wedge \bo) = \FF_2[v^{\pm 1}_0, u^2, v^2]. $$ 
\item We have
\begin{align*} \frac{\Ext_{A(1)_*}((A(1)\mmod A(0))_* \otimes \ul{\bo}_j)}{\text{$v_0$-tors}} & \cong
\frac{\Ext_{A(0)_*}(\ul{\bo}_j)}{\text{$v_0$-tors}} \\
& \cong \FF_2[v_0]\{ 1, \xi_1^4, \ldots,  \xi_1^{4j}\}.
\end{align*}
This follows from the fact that 
$$ \frac{\Ext_{A(0)_*}(\ul{\HZ}_j)}{\text{$v_0$-tors}} \cong \FF_2[v_0], $$
which, for instance, can be established by induction using the short exact sequences of Lemma~\ref{lem:HZSES}.
\end{enumerate} 
\end{proof}

It turns out the $v_1$-torsion is all concentrated in Adams filtration $0$ (see, for instance, \cite{boresolutions}).  It follows that for dimensional reasons, the only possible non-trivial differentials in spectral sequence (\ref{eq:boboASS}) go from $v_1$-torsion classes to $v_1^4$-periodic classes.  This is not possible, so we deduce

\begin{cor}
The Adams spectral sequence for $\bo_*\bo$ collapses at $E_2$.
\end{cor}

\subsection{The cooperations of $\KU$ and $\bu$}

In order to put the ring of cooperations for $\bo$ in the proper setting, we briefly review the story for $\bu$.
We begin by recalling the Adams-Harris determination of $\KU_*\KU$ \cite[Sec.~II.13]{Adams}.  We have an arithmetic square
$$
\xymatrix{
\KU \wedge \KU \ar[r] \ar[d] & (\KU \wedge \KU)^\wedge_2 \ar[d] \\
(\KU \wedge \KU)_\QQ \ar[r] & ((\KU \wedge \KU)^\wedge_2)_\QQ,
}
$$
which results in a pullback square after applying $\pi_*$
$$
\xymatrix{
\KU_* \KU \ar[r] \ar[d] & \Map^c(\ZZ_2^\times, \pi_* \KU^\wedge_2) \ar[d] \\
\QQ[u^{\pm 1},v^{\pm 1}] \ar[r] & \Map^c(\ZZ_2^\times, \QQ_2[u^{\pm 1}]).
}
$$
Setting $w = v/u$, the bottom map in the above square is given on homogeneous polynomials by
$$ f(u,v) = u^n f(1,w) \mapsto \left(\lambda \mapsto u^n f(1,\lambda) \right). $$
We therefore deduce that $\KU_*\KU = \KU_* \otimes_{\KU_0} \KU_0\KU$, and continuity implies that
$$ \KU_0\KU = \{ f(w) \in \QQ[w^{\pm 1}] \: : \: f(k) \in \ZZ_{(2)}, \text{for all $k \in \ZZ^{\times}_{(2)}$}\}. $$

Note that we can perform a similar analysis for $\KU_*\bu$: since $\bu$ and $\KU$ are $K(1)$-locally equivalent, applying $\pi_*$ to the arithmetic square yields a pullback square with the same terms on the right hand edge
$$
\xymatrix{
\KU_* \bu \ar[r] \ar[d] & \Map^c(\ZZ_2^\times, \pi_* \KU^\wedge_2) \ar[d] \\
\QQ[u^{\pm 1},v] \ar[r] & \Map^c(\ZZ_2^\times, \QQ_2[u^{\pm 1}]).
}
$$
Consequently $\KU_*\bu = \KU_* \otimes_{\KU_0} \KU_0\bu$, with
$$ \KU_0\bu = \{ g(w) \in \QQ[w] \: : \: g(k) \in \ZZ_{(2)}, \text{for all $k \in \ZZ^{\times}_{(2)}$}\}. $$

Consider the related space of \emph{$2$-local numerical polynomials}:
$$ \mr{NumPoly}_{(2)} := \{ h(x) \in \QQ[x] \: : \: h(k) \in \ZZ_{(2)}, \text{for all $k \in \ZZ_{(2)}$}\}. $$
The theory of numerical polynomials states that $\mr{NumPoly}_{(2)}$ is the free $\ZZ_{(2)}$-module generated by the basis elements
$$ h_n(x) := \binom{x}{n} = \frac{x(x-1)\cdots (x-n+1)}{n !}. $$

We can relate $\KU_0\bu$ to $\mr{NumPoly}_{(2)}$ by a change of coordinates.  A function on $\ZZ^\times_{(2)}$ can be regarded as a function on $\ZZ_{(2)}$ via the change of coordinates
\begin{align*}
\ZZ_{(2)} & \xrightarrow{\approx} \ZZ_{(2)}^\times \\
k & \mapsto 2k+1. 
\end{align*}
Observe that
\begin{align*}
\frac{k(k-1)\cdots (k-n+1)}{n !} & = 
\frac{2k(2k-2)\cdots (2k-2n+2)}{2^n n !} \\
& = 
\frac{((2k+1)-1)((2k+1)-3)\cdots ((2k+1)-(2n-1))}{2^n n !}.
\end{align*}
We deduce that a $\ZZ_{(2)}$-basis for $\KU_0\bu$ is given by
\[
g_n(w) = \frac{(w-1)(w-3)\ldots(w-(2n-1))}{2^nn!}.
\]
(Compare with \cite[Prop.~17.6(i)]{Adams}.)

From this we deduce a basis of the image of the map
$$ \bu_*\bu \hookrightarrow \KU_* \KU, $$
as we now explain.
In \cite[p. 358]{Adams} it is shown that this image is the ring
\[
\frac{\bu_*\bu}{\text{$v_1$-tor}} = (\KU_*\bu \cap \Q[u,v])_{\mr{AF} \geq 0},
\]
where $\mr{AF} \geq 0$ means the elements of Adams filtration $\geq 0$. Since the elements $2$, $u$, and $v$ have Adams filtration $1$, this image is equivalently described as
$$ 
\frac{\bu_*\bu}{\text{$v_1$-tor}} = \KU_* \bu \cap \ZZ_{(2)}[u/2, v/2].
$$
To compute a basis for this image 
we need to calculate the Adams filtration of the elements of the basis $\{ g_n(w) \}$ for $\KU_0\bu$. 
Since $w$ has Adams filtration $0$ we need only compute the $2$-divisibility of the denominators of the functions $g_n(w)$. As usual in this subject, for an integer $k \in \Z$ let $\nu_2(k)$ be the largest power of $2$ that divides $k$ and let $\al(k)$ be the number of $1$'s in the binary expansion of $k$. Then
\[
\nu_2(n!) = n - \al(n)
\]
and so
\[
\mr{AF}(g_n) = \al(n) - 2n.
\]

The following is a list of the Adams filtration of the first few basis elements:

\begin{center}
\begin{tabular}{|l|r|r|}
\hline
$n$ & binary & $\mr{AF}(g_n)$ \\
\hline
$0$ & $0$ & $0$ \\
$1$ & $1$ & $-1$ \\
$2$ & $10$ & $-3$ \\
$3$ & $11$ & $-4$ \\
$4$ & $100$ & $-7$ \\
$5$ & $101$ & $-8$ \\
$6$ & $110$ & $-10$ \\
$7$ & $111$ & $-11$ \\
$8$ & $1000$ & $-15$ \\
\hline
\end{tabular}
\end{center}

It follows (compare with \cite[Prop.~17.6(ii)]{Adams})
 that the image of $\bu_*\bu$ in $\KU_* \KU$ is the free module:
$$ \frac{\bu_*\bu}{\text{$v_1$-tor}} = \ZZ_{(2)}\{ 2^{\max(0, 2n-m-\al(n))} u^m g_n(w)\: : \: n \ge 0, m \ge n \}. $$

The Adams chart in Figure \ref{fig:bubu} illustrates how the the Mahler basis can be used to identify $\bu_*\bu/v_1-tors$ as a $\bu_*$-module inside of $\KU_*\KU$.  Namely:
\begin{enumerate}
\item Start with the Mahler basis $g_n(w)$ (on the negative $s$-axis).
\item Draw $u$-towers on each of the $g_n(w)'s$.
\item For each $m \ge n$, add a $v_0$-tower on $u^mg_n(w)$, starting in non-negative Adams filtration. 
\end{enumerate}
Restricted to the first quadrant, this gives the $v_1$-torsion-free summand of the Adams spectral sequence for $\bu_*\bu$. 

\begin{figure}
\centering
	\caption{$\bu_*\bu$}
	\label{fig:bubu}
	\includegraphics[height = .5\textheight, trim = 1cm 13cm 4cm 2cm]{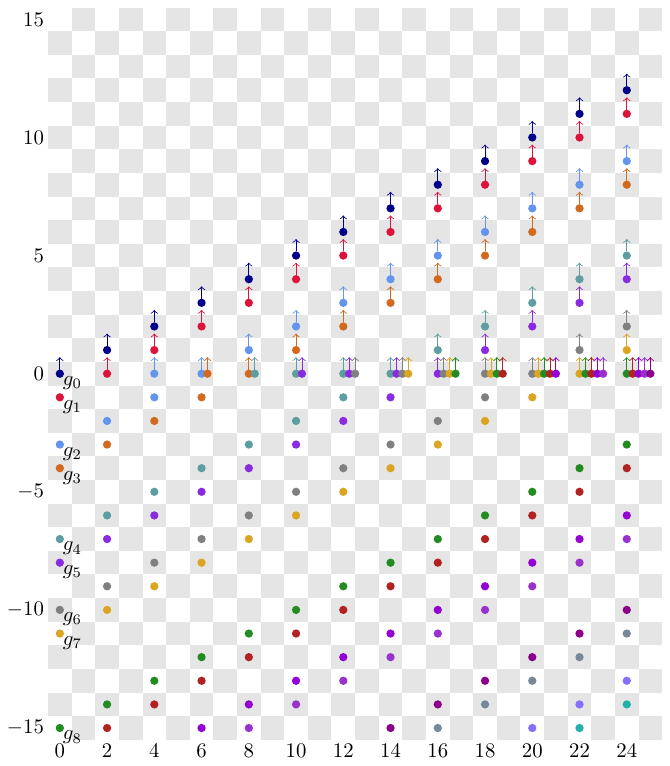}
\end{figure}

\subsection{The cooperations of $\KO$ and $\bo$}

Adams and Switzer computed $\KO_*\KO$  along similar lines \cite[Sec.~II.17]{Adams}.  There is an arithmetic square
$$
\xymatrix{
\KO \wedge \KO \ar[r] \ar[d] & (\KO \wedge \KO)^\wedge_2 \ar[d] \\
(\KO \wedge \KO)_\QQ \ar[r] & ((\KO \wedge \KO)^\wedge_2)_\QQ,
}
$$
which results in a pullback when applying $\pi_*$
$$
\xymatrix{
\KO_* \KO \ar[r] \ar[d] & \Map^c(\ZZ_2^\times/\{\pm 1\}, \pi_* \KO^\wedge_2) \ar[d] \\
\QQ[u^{\pm 2},v^{\pm 2}] \ar[r] & \Map^c(\ZZ_2^\times/\{ \pm 1\}, \QQ_2[u^{\pm 2}]).
}
$$
(One can use the fact that $\KU^\wedge_2$ is a $K(1)$-local $C_2$-Galois extension of $\KO^\wedge_2$ to identify the upper right hand corner of the above pullback.)
Continuing to let $w = v/u$, the bottom map in the above square is given by
$$ f(u^2,v^2) = u^{2n} f(1,w^2) \mapsto \left([\lambda] \mapsto u^{2n} f(1,\lambda^2) \right). $$
We therefore deduce that $\KO_*\KO = \KO_* \otimes_{\KO_0} \KO_0\KO$, with
$$ \KO_0\KO = \{ f(w^2) \in \QQ[w^{\pm 2}] \: : \: f(\lambda^2) \in \ZZ^\times_{2}, \text{for all $[\lambda] \in \ZZ_2^{\times}/\{\pm 1\}$}\}. $$

Again, $\KO_*\bo$ is similarly determined: since $\bo$ and $\KO$ are $K(1)$-locally equivalent, applying $\pi_*$ to the arithmetic square yields a pullback square with the same terms on the right hand edge:
$$
\xymatrix{
\KO_* \bo \ar[r] \ar[d] & \Map^c(\ZZ_2^\times/\{ \pm 1\}, \pi_* \KO^\wedge_2) \ar[d] \\
\QQ[u^{\pm 2},v^2] \ar[r] & \Map^c(\ZZ_2^\times/\{\pm 1\}, \QQ_2[u^{\pm 2}]).
}
$$
We therefore deduce that $\KO_*\bo = \KO_* \otimes_{\KO_0} \KO_0\bo$, with
$$ \KO_0\bo = \{ f(w^2) \in \QQ[w^2] \: : \: f(\lambda^2) \in \ZZ_{2}, \text{for all $[\lambda] \in \ZZ^{\times}_{2}/\{\pm 1 \}$}\}. $$

To produce a basis of this space of functions we use the $q$-Mahler bases developed in \cite{Conrad}, which we promptly recall. 
First note that there is an exponential isomorphism
\[
\Z_2 \lra{\cong} \Z_{2}^{\times}/\{\pm 1\}:k \mapsto [3^k].
\]
Taking $w = 3^k$, we have $w^2 = 9^k$, or in other words, the functions $f(w^2)$ that we are concerned with can be regarded as functions on $2\Z_2$. They take the form
\[
f(9^k):2\Z_2 \cong 1+8\Z_2 \lra{} \Z_2,
\]
where $1+8\Z_2 \subset \Z_{2}^{\times}$ is the image of $2\Z_2$ under the isomorphism given by $3^k$.

To obtain a $q$-Mahler basis as in \cite{Conrad} with $q = 9$ it is important that $\nu_2(9-1)>0$. 
The $q$-Mahler basis is a basis for numerical polynomials with domain restricted to $2\Z_2$.
In the notation of \cite{Conrad} we have that
\[
f(9^k) = \sum_{n\geq 0}c_n \binom{k}{n}_9, 
\]
where $c_n \in \Z_{(2)}$ are coefficients and
\[
\binom{k}{n}_9 = \frac{(9^k-1)(9^k-9)\cdots (9^k - 9^{n-1})}{(9^n-1)(9^n-9)\cdots (9^n - 9^{n-1})}.
\]

Let us set
\begin{align}\label{eq:f_n(w^2)}
f_n(w^2) = \frac{(w^2-1)(w^2-9)\cdots (w^2 - 9^{n-1})}{(9^n-1)(9^n-9)\cdots (9^n - 9^{n-1})};
\end{align}
then any $f\in \KO_0\bo$ is given by
\[
f(w^2) = \sum_{n}c_n f_n(w^2), \qquad c_n \in \Z_{(2)}, 
\]
i.e. a basis for $\KO_0 \bo$ is given by the set $\{ f_n (w^2) \}_{n \ge 0}$.

As in the $\KU$-case, it turns out that the image of $\bo_* \bo$ in $\KO_* \KO$ is given by
$$ \frac{\bo_*\bo}{\text{$v_1$-tor}} = (\KO_*\bo \cap \QQ[u^2, v^2])_{\AF \ge 0}. $$
In order to compute a basis for this we once again need to know the Adams filtration of $f_n$. One can show that
\begin{eqnarray*}
\nu_2((9^n-1)(9^n-9)\cdots (9^n - 9^{n-1})) & = & \nu_2(n!)+3n \\
& = & 4n - \al(n).
\end{eqnarray*}
It follows that we have
\begin{align*}
\frac{\bo_*\bo}{\text{$v_1$-tor}} = & 
\: \ZZ_{(2)}\{ 2^{\max(0, 4n-2m-\al(n))} u^{2m} f_n(w^2)\: : \: n \ge 0, \: m \ge n, \: m \equiv 0 \mod 2 \} \\
& \oplus \ZZ_{(2)}\{ 2^{\max(0, 4n-2m-1-\al(n))} 2u^{2m} f_n(w^2)\: : \: n \ge 0, \: m \ge n, \: m \equiv 0 \mod 2 \} \\
& \oplus \ZZ/2\left\{ u^{2m} f_n(w^2) \eta^c \: : \: 
\begin{array}{l}
n \ge 0, \: m \ge n, \: m \equiv 0 \mod 2, \\
c \in \{1,2\}, \: \alpha(n)-4n+2m+c \ge 0  
\end{array}
\right\}.
\end{align*}

Here is a list of the Adams filtration of the first several elements in the $q$-Mahler basis:
\begin{center}
\begin{tabular}{|l|r|r|}
\hline
$n$ & $f_n$ in terms of $g_i$ & $\mr{AF}(f_n)$ \\
\hline
$0$ & $g_0$ & $0$ \\
$1$ & $g_2 + g_1$ & $-3$ \\
$2$ & $\frac{1}{15}g_4 + \frac{2}{15}g_3 + \frac{1}{15} g_2$ & $-7$ \\
\hline
\end{tabular}
\end{center}

With this information we can now give the Adams chart (Figure \ref{fig:bobo}) of $\bo_*\bo$ modulo $v_1$-torsion.

\begin{figure}
\centering
	\caption{$\bo_*\bo$}
	\label{fig:bobo}
	\includegraphics[height = .3\textheight, trim = 5cm 17.5cm 3cm 2cm]{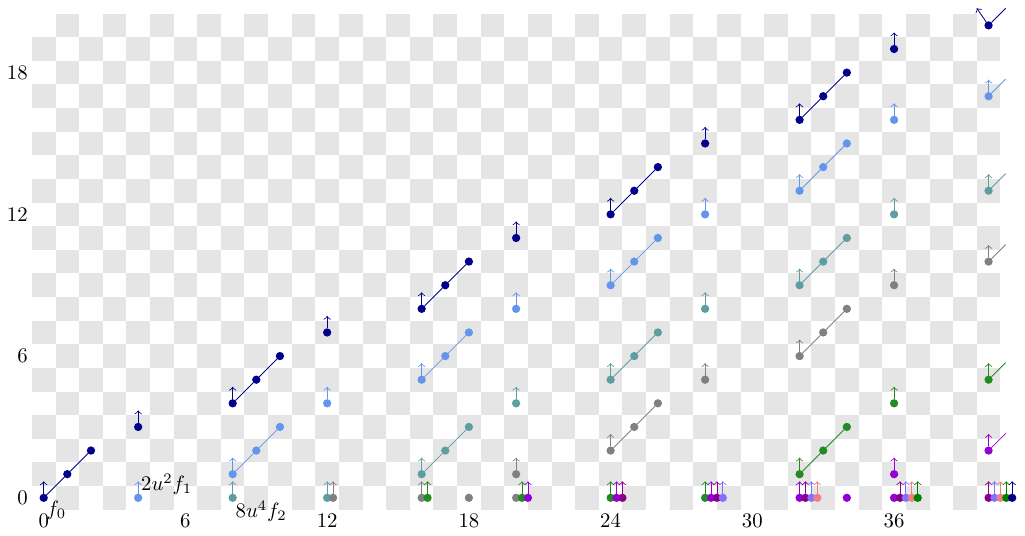}
\end{figure}

\subsection{Calculation of the image of $\bo_* \HZ_j$ in $\KO_* \KO$}
We now compute the image (on the level of Adams $E_\infty$-terms) of the composite
$$ \bo_*\HZ_j \rightarrow \bo_*\bo \rightarrow \KO_*\KO. $$
Since $v_1^{-1}\bo_*\Sigma^{4j}\HZ_j \cong \KO_*$, it suffices to determine the image of the generator 
$$ e_{4j} \in \bo_{4j}(\Sigma^{4j}\HZ_j). $$ 
Because the maps
$$ \bo \wedge \Sigma^{4j} \HZ_j \rightarrow \bo \wedge \bo $$
are constructed to be $\bo$-module maps,
everything else is determined by $2$ and $v_1$, i.e. $u$-multiplication. 
Consider the commutative diagram induced by the maps $\bo \rightarrow \bu$, $\bu \rightarrow \HF2$, and $\BP \rightarrow \bu$
$$
\xymatrix{
\bo \wedge \Sigma^{4j} \HZ_j \ar[r] \ar[d] & 
\bo \wedge \bo \ar[r] \ar[d] & 
\bu \wedge \bu \ar[d] & 
\BP \wedge \BP \ar[ld] \ar[l] 
\\
\HF2 \wedge \Sigma^{4j} \HZ_j \ar[r] & 
\HF2 \wedge \bo \ar[r] & 
\HF2 \wedge \HF2. 
}
$$
On the level of homotopy groups the bottom row of the above diagram takes the form
$$ \FF_2\{\bar{\xi}_1^{4j}, \ldots \} \hookrightarrow \FF_2[\bar{\xi}_1^4, \bar{\xi}_2^2, \bar{\xi}_3, \ldots ] \hookrightarrow \FF_2[\bar{\xi}_1, \bar{\xi}_2, \bar{\xi}_3, \ldots ].
$$ 
Since we have
\begin{align*}
\bo_{*}\Sigma^{4j}\HZ_j & \rightarrow (\HF2)_*\Sigma^{4j}\HZ_j \\
e_{4j} & \mapsto \bar{\xi}_1^{4j},
\end{align*}
it suffices to find an element $b_j \in \bo_{4j}\bo$ such that
\begin{align*}
\bo_{*}\bo & \rightarrow (\HF2)_*\bo \\
b_i & \mapsto \bar{\xi}_1^{4j}.
\end{align*}
Clearly we can take $b_0 = 1 \in \bo_0\bo$.
Note that we have
\begin{align*}
\BP_{*}\BP & \rightarrow (\HF2)_*\HF2 \\
t_1 & \mapsto \bar{\xi}_1^{2}.
\end{align*}
From the equation
$$ \eta_R(v_1) = v_1 + 2t_1 $$
and the fact that the map $\BP_*\BP \to \bu_*\bu$ is one of Hopf algebroids,
we deduce that we have
\begin{align*}
\BP_{*}\BP & \rightarrow \bu_*\bu \\
t_1 & \mapsto \frac{v-u}{2} = ug_1(w).
\end{align*}
Hence we get that
\begin{align*}
\bu_{*}\bu & \rightarrow (\HF2)_*\HF2 \\
\frac{v-u}{2} & \mapsto \bar{\xi}_1^2
\end{align*}
and thus
\begin{align*}
\bu_{*}\bu & \rightarrow (\HF2)_*\HF2 \\
\left(\frac{v^2-u^2}{4}\right)^{j} & \mapsto \bar{\xi}_1^{4j}.
\end{align*}
Since
$$ 2^{2j - \alpha(j)}u^{2j}f_j(w^2) = \left(\frac{v^2-u^2}{4}\right)^{j} \quad \text{modulo terms of higher $\AF$} $$
by \eqref{eq:f_n(w^2)} we see that we have
\begin{align*}
\bo_{*}\bo & \rightarrow (\HF2)_*\bo \\
2^{2j - \alpha(j)}u^{2j}f_j(w^2) & \mapsto \bar{\xi}_1^{4j},
\end{align*}
so that we can take
$$ b_j = 2^{2j - \alpha(j)}u^{2j}f_j(w^2). $$

We have therefore arrived at the following well-known theorem (see \cite[Cor.~2.5(a)]{LellmannMahowald}).

\begin{thm}\label{thm:HZjImage}
The image of the map
$$
\frac{\Ext(\bo \wedge \Sigma^{4j} \ul{\HZ}_j)}{\text{$v_1$-tors}}  \rightarrow 
 \frac{\Ext(\bo \wedge bo)}{\text{$v_1$-tors}}
$$
is the submodule
\begin{align*}
& \FF_2[v_0]\{ v_0^{\max(0, 4j-2m-\al(j))} u^{2m} f_j(w^2)\: : \: m \ge j, \: m \equiv 0 \mod 2 \} \\
& \oplus \FF_2[v_0] \{ v_0^{\max(0, 4j-2m-1-\al(j))} v_0 u^{2m} f_j(w^2)\: : \: m \ge j, \: m \equiv 0 \mod 2 \} \\
& \oplus \FF_2\left\{ u^{2m} f_j(w^2) \eta^c \: : \: 
\begin{array}{l}
m \ge j, \: m \equiv 0 \mod 2, \\
c \in \{1,2\}, \: \alpha(j)-4j+2m+c \ge 0  
\end{array}
\right\}.
\end{align*}
\end{thm}

\begin{rmk}
For each $j \geq 0$, this theorem describes a submodule of $\frac{\Ext(\bo \wedge bo)}{\text{$v_1$-tors}}$. These submodules are represented by the different colors in Figure \ref{fig:bobo}. 
\end{rmk}

\subsection{The embedding into $\prod \KO$}

The final step is to consider the maps of $\KO$-algebras given by the composite
$$ \td{\psi}^{3^k}: \KO \wedge \KO \xrightarrow{1 \wedge \psi^{3^k}} \KO \wedge \KO \xrightarrow{\mu} 
\KO $$
where $\psi^{3^k}$ is the $3^k$-th Adams operation. Together, they result in a map of $\KO$-algebras
$$ \KO \wedge \KO \xrightarrow{\prod \td{\psi}^{3^k}} \prod_{k \in \ZZ} \KO. $$

\begin{rmk}
The map above has a modular interpretation. Let $\M_{fg}$ denote the moduli stack of formal groups, and let
\[
(\Spec\Z)\mmod C_2 \rightarrow \M_{fg}
\]
classify $\hat{\mathbb{G}}_m$ with the action of $[-1]$. This map equips $(\Spec \Z)\mmod C_2$ with a sheaf of $\Einf$-rings, such that the derived global sections are $\KO$; the reader is referred to the appendix of \cite{LawsonNaumann2} for details. The spectrum $\KO \wedge \KO$ is the global sections of the pullback
\[
\left(\Spec \Z \times_{\M_{fg}} \Spec \Z\right)\mmod (C_2 \times C_2).
\]
For $k \in \ZZ$ we may consider the map of stacks
\[
(\Spec \Z)\mmod C_2 \rightarrow \left(\Spec\Z \times_{\M_{fg}} \Spec\ZZ \right)\mmod (C_2 \times C_2)
\]
sending $\hat{\mathbb{G}}_m$ to the object $[3^k]:\hat{\mathbb{G}}_m \rightarrow \hat{\mathbb{G}}_m$. As $k$ varies this induces the map $\prod \td{\psi}^{3^k}$.
\end{rmk}

\begin{prop}
The map
$$ \KO_* \KO \xrightarrow{\prod \td{\psi}^{3^k}} \prod_{k \in \ZZ} \KO_* $$
is an injection.
\end{prop}

\begin{proof}
Consider the diagram
$$
\xymatrix{
\KO_* \KO \ar[r]^{\prod \td{\psi}^{3^k}} \ar[d] & \prod_{k \in \ZZ} \KO_*  \ar[d] 
\\
(\KO_* \KO)^{\wedge}_{2} \ar[r]^{\prod \td{\psi}^{3^k}} \ar@{=}[d] &  \prod_{k \in \ZZ} (\KO_*)^{\wedge}_2 \ar@{=}[d] 
\\
\Map^c(\ZZ_2^\times/\{\pm 1\}, (\KO_*)^\wedge_2) \ar[r] 
& \Map(3^\ZZ, (\KO_*)^\wedge_2), 
} 
$$
where the bottom horizontal map is the map induced from the inclusion of groups 
$$ 3^\ZZ \hookrightarrow \ZZ_2^\times/\{ \pm 1 \}. $$
The vertical maps are injections, since 
$$ \bigcap_i 2^i \KO_*\KO = 0, \quad \rm{and} \quad \bigcap_i 2^i \KO_* = 0. $$
The bottom horizontal map is an injection since $3^\ZZ$ is dense in $\ZZ_2^\times/\{\pm 1\}$.
The result follows.
\end{proof}

We investigated the Brown-Gitler wedge decomposition
$$ \bigvee_j \bo \wedge \Sigma^{4j} \HZ_j \xrightarrow{\simeq} \bo \wedge \bo, $$
and we now end this section by explaining how the map
$$ \KO \wedge \KO \xrightarrow{\prod \td{\psi}^{3^k}} \prod_{k \in \ZZ} \KO $$
is compatible with the above decomposition.

\begin{prop}
The composites
$$ \bo \wedge \HZ_j \rightarrow \bo \wedge \bo \rightarrow \KO \wedge \KO \xrightarrow{\td{\psi}^{3^j}} \KO $$
are equivalences after inverting $v_1$.
\end{prop}

\begin{proof}
This follows from the fact that $f_j(9^j) = 1.$
\end{proof}

\begin{rmk}
In fact, the ``matrix" representing the composite
$$
\bigvee_j \bo \wedge \HZ_j \rightarrow \bo \wedge \bo \rightarrow \KO \wedge \KO \xrightarrow{\prod \td{\psi}^{3^k}} \prod_{k \in \ZZ} \KO
$$ 
is upper triangular, as we have
\begin{equation}\label{eq:UT}
f_j(9^k) = 
\begin{cases}
0, & k < j, \\
1, & k = j. \\
\end{cases}
\end{equation}
This is related to a result of Barker and Snaith \cite{BarkerSnaith} in the following way.  They prove that with respect to the decomposition
\begin{equation}\label{eq:bubodecomp}
 \bu \wedge \bo \simeq \bigvee_j \Sigma^{4j} \bu \wedge H\ZZ_j 
 \end{equation}
the automorphism
$$ 1 \wedge \psi^3: \bu \wedge \bo \rightarrow \bu \wedge \bo $$
is represented by a matrix conjugate to 
$$
\begin{pmatrix}
1 & 1 & 0 & 0  & \cdots \\
0  & 9 & 1 & 0  & \cdots \\
0 & 0  & 9^2 & 1 & \cdots \\
\vdots  & \vdots  & \vdots  & \ddots & \ddots
\end{pmatrix}.
$$
Therefore $1 \wedge \psi^{3^k}$ corresponds to a matrix of the form
\begin{equation}\label{eq:matrixreloaded}
\begin{pmatrix}
\smash[b]{\overbrace{* \cdots *}^{k-1}} & 1 & 0 & 0 & \cdots \\
 \vdots  & \vdots  & \vdots  &  \vdots  & \\
 \vdots  & \vdots  & \vdots  &  \vdots  & \\
\end{pmatrix}.
\end{equation}
Using the fact that the composite
$$ \bu \wedge \bo \rightarrow \bu \wedge \bu \xrightarrow{\mu} \bu $$
corresponds to projection on the $j = 0$ summand of (\ref{eq:bubodecomp}), it follows that (\ref{eq:UT}) is consistent with the top row of the matrix (\ref{eq:matrixreloaded}).
\end{rmk}

%% file: tmfcoop_review.tex

\section{Recollections on topological modular forms}\label{sec:review}

\subsection{Generalities} \emph{In this subsection, we work integrally.}
The remainder of this paper is concerned with determining as much information as we can about the cooperations in the homology theory $\tmf$ of connective topological modular forms, following our guiding example of $\bo$. 
Even more than in the $\bo$ case, an extensive cast of characters will play supporting roles. First of all, we will extensively use the periodic spectrum $\TMF$, which is the analogue of $\KO$. In particular, we will use the fact that this periodic form of topological modular forms arises as the global sections of the Goerss-Hopkins-Miller sheaf of ring spectra $\mathcal{O}^{top}$ on the moduli stack of smooth elliptic curves $\M$. As the associated homotopy sheaves are
\[\pi_{k}\mathcal{O}^{top}=\begin{cases} \omega^{\otimes k/2}, \text{ if } $k$ \text{ is even},\\
0,  \text{ if } $k$ \text{ is odd},  \end{cases} \]
there is a descent spectral sequence 
\[ H^s(\M, \omega^{\otimes t}) \Rightarrow \pi_{2t-s} \TMF.\]

Morally, the connective $\tmf$ should arise as global sections of an analogous sheaf on the moduli stack of all cubic curves (i.e. allowing nodal and cuspidal singularities); however, this has not been formally carried out. Nevertheless, $\tmf$ can be constructed as an $\Einf$-ring spectrum from $\TMF$ as a result of the gap in the homotopy of a third, non-connective and non-periodic, version of topological modular forms associated to the compactification of $\M$.

Rationally, every smooth elliptic curve $C/S$ is locally isomorphic to a cubic of the form \[ y^{2}= x^{3}-27c_{4}x-54c_{6},\] with the discriminant $\Delta=c_{4}^{3}-c_{6}^{2}$ invertible. Here $c_i$ is a section of the line bundle $\omega^{\otimes i}$ over the \'etale map $S\to \M$ classifying $C$. This translates to the fact that $\M_{\Q}\cong \Proj \Q[c_4,c_6][\Delta^{-1}]$, which in turn implies that $(\TMF_{*})_{\Q}=\Q[c_{4},c_{6}][\Delta^{-1}].$ The connective version has $(\tmf_{*})_{\Q}=\Q[c_{4},c_{6}]$.

The spectrum of topological modular forms is, of course, not complex orientable, and just like in the case of $\bo$, we will need the aid of a related complex orientable spectrum. The periodic spectrum $\TMF$ admits ring maps to several families of orientable (as well as non-orientable) spectra which come from the theory of elliptic curves. Namely, an elliptic curve $C$ is an abelian group scheme, and in particular it has a subgroup scheme $C[n]$ of points of order $n$ for any positive integer $n$. When $n$ is invertible, $C[n]$ is locally isomorphic to the constant group $(\Z/n)^{2}$. Based on this observation, there are various additional structures that one can assign to an elliptic curve. In this work we will be concerned with two types, the so-called $\Gamma_1(n)$ and $\Gamma_0(n)$ level structures.

A $\Gamma_1(n)$ level structure on an elliptic curve $C$ is a specification of a point $P$ of (exact) order $n$ on $C$, whereas a $\Gamma_0(n)$ level structure is a specification of a cyclic subgroup $H$ of $C$ of order $n$. The corresponding moduli problems are denoted $\Msl{n}$ and $\Mcl{n}$. Assigning to the pair $(C,P)$ the pair $(C,H_{P})$, where $H_P$ is the subgroup of $C$ generated by $P$, determines an \'etale map of moduli stacks
\[ g: \Msl{n}\to \Mcl{n}. \]
Moreover, there are two morphisms
\[f,q: \Mcl{n}\to \M[1/n] \]
which are \'etale; $f$ forgets the level structure whereas $q$ quotients $C$ by the level structure subgroup. Composing with $g$, we obtain analogous maps from $\Msl{n}$. We can take sections of $\mathcal{O}^{top}$ over the forgetful maps and obtain ring spectra $\TMF_1(n)$ and $\TMF_0(n)$, ring maps $\TMF[1/n]\to \TMF_0(n)\to \TMF_1(n)$  as well as maps of descent spectral sequences
\[\xymatrix{
 H^*(\M[1/n], \omega^{\otimes *} ) \ar@{=>}[r]\ar[d] &\pi_* \TMF[1/n] \ar[d]\\
 H^*(\M_?(n), \omega^* ) \ar@{=>}[r] & \pi_* \TMF_?(n), }\]
obtained by pulling back. 
In particular, for any odd integer $n$ we have such a situation $2$-locally.  

We use the ring map $f:\TMF[1/n]\to \TMF_0(n)$ induced by the forgetful $f:\Mcl{n} \to \M[1/n]$ to equip $\TMF_0(n)$ with a $\TMF[1/n]$-module structure.
With this convention, the map $q:\TMF[1/n]\to \TMF_0(n)$ induced by the quotient map on the moduli stacks does not respect the $\TMF[1/n]$-module structure. However, one can uniquely extend $q$ to 
\begin{align}\label{eq:mapPsi_n}
\xymatrix{
\TMF[1/n] \ar[r]^{q}\ar[d] & \TMF_0(n).\\
\TMF[1/n] \wedge \TMF[1/n] \ar@{-->}[ru]_{\Psi_n}
} 
 \end{align}
Another way to define $\Psi_n$ is as the composition of $f\wedge q$
with the multiplication on $\TMF_0(n)$.

Finally, we will be interested in the morphism
\[
  \phi_{[n]}:\M[1/n]\to \M[1/n],
\]
which is the \'{e}tale map induced by the multiplication-by-$n$ isogeny on an
elliptic curve, and the induced map $\phi_{[n]}:\TMF[1/n]\to
\TMF[1/n]$ is an Adams operation on $\TMF[1/n]$.

In Section \ref{sec:approxlevel} below, we will make heavy use of the maps $\Psi_3$ and $\Psi_5$. Their usefulness is due to the relative ease with which their behavior on non-torsion homotopy groups can be computed. 

\begin{rmk}
There is a subtlety in defining the maps
\begin{align*}
 q: & \TMF[1/n] \rightarrow \TMF_0(n), \\
 \phi_{[n]}: & \TMF[1/n] \rightarrow \TMF[1/n]
 \end{align*}
which is glossed over in the above discussion.  The definition of the map $q$ presupposes a canonical identification of the sections of $\mc{O}^{top}$ on the \'etale opens $f$ and $q$, and the definition of the map $\phi_{[n]}$ somehow associates a map of spectra to an isogeny of elliptic curves.
The \emph{real} origin of these maps of spectra comes from Lurie's generalization of the Goerss-Hopkins-Miller Theorem (see \cite{TAF}), which actually presents the $p$-completions of the sheaf $\mc{O}^{top}$ as a sheaf on the \'etale site of the moduli stack of height two $1$-dimensional $p$-divisible groups $\M_{pd}$.
The $p$-torsion of an elliptic curve $C$ gives a $p$-divisible group $C[p^\infty]$.
Let
$$ u: \M_?(n)^{\wedge}_{p} \rightarrow \M_{pd} $$
denote the map which forgets level structure and outputs the $p$-divisible group of the underlying elliptic curve (where $(p,n) = 1$).  The Serre-Tate theorem implies this map is \'etale, and $\TMF_?(n)^{\wedge}_p$ is the associated spectrum of sections.  Given a cyclic subgroup $H$ of order $n$, the isogeny
$$ C \rightarrow C/H $$
induces an isomorphism of associated $p$-divisible groups, and hence gives a $2$-cell making the following diagram of stacks homotopy commute:
$$
\xymatrix@R-2em@C-2em{
\M_0(n)^\wedge_p \ar[dd]_q \ar[drr]^u && \\
& \Downarrow & \M_{pd} \\
\M^{\wedge}_p \ar[urr]_u 
}
$$
This induces a map on sections
$$ q: \TMF^{\wedge}_p \rightarrow \TMF_0(n)^{\wedge}_p $$
(see, for instance, \cite[Ch. 5]{TMF}).
The map
$$ q: \TMF[1/n] \rightarrow \TMF_0(n) $$
is then obtained by constructing the map rationally, and assembling over an arithmetic square.
The map
$$ \phi_{[n]} : \TMF[1/n] \rightarrow \TMF[1/n] $$
is obtained using the diagram
$$
\xymatrix@R-2em@C-2em{
\M^\wedge_p \ar@{=}[dd] \ar[drr]^u && \\
& \Downarrow & \M_{pd} \\
\M^{\wedge}_p \ar[urr]_u 
}
$$
induced by the isogeny
$$ [n]: C \rightarrow C. $$
A different perspective on these maps between $\TMF$-spectra can be found in \cite{betacong}, but that treatment also secretly relies on Lurie's generalized Goerss-Hopkins-Miller theorem. The reader uncomfortable with relying on unpublished work could also obtain the morphisms $q$ and $\phi_{[n]}$ using the obstruction theoretic construction of $\TMF$ described in \cite[Ch. 12]{TMF}: the isogenies induce isomorphisms on formal groups, and the functoriality of the Goerss-Hopkins-Miller theorem gives maps on the $K(2)$ localizations of $\TMF$. An explicit map of $\theta$-algebras corresponding to the respective isogenies gives, via $K(1)$-local $E_\infty$ obstruction theory, a map of the $K(1)$-localizations of $\TMF$.  These assemble via chromatic fracture to give a map on the $p$-completions of $\TMF$ for $(p,n) = 1$, and these then assemble via the arithmetic square to give the desired maps.
\end{rmk}

\subsection{Details on $\tmf_1(3) $ as $\BP\bra{2}$}\label{sec:tmf13BP2}
\emph{We return to the convention that everything is $2$-local.}
The significance of $\bu$ in the computation of $\bo_*\bo$ was that at the prime $2$, $\bu$ is a truncated Brown-Peterson spectrum $\BP\bra{1}$ with a ring map $\bo \to \bu$ which upon $K(1)$-localization becomes the inclusion of homotopy fixed points
$(\KU\hat{_2})^{hC_2}\to \KU\hat{_2}$. In particular, the image of $\KO\hat{_2} \to \KU\hat{_2}$ in homotopy is describable as certain invariant elements.
 By work of Lawson-Naumann \cite{LawsonNaumann}, we know that there is a $2$-primary form of $\BP\bra{2}$ obtained from topological modular forms; this will be our analogue of $\bu$ in the $\tmf$-cooperations case.

Lawson-Naumann study the ($2$-local) compactification of the moduli stack $\Msl{3}$. Given an elliptic curve $C$ (over a $2$-local base), it is locally isomorphic to a Weierstrass curve of the form
\[ y^2+a_1xy+a_3 y = x^3 +a_4x+a_6. \]
A point $P=(r,s)$ of order $3$ is an inflection point of such a curve; transforming the curve so that the given point $P$ is moved to have coordinates $(0,0)$ puts $C$ in the form
\begin{align}\label{eq:Ca1a3}
 y^2+a_1 xy +a_3 y = x^3.
 \end{align} 
 This is the universal equation of an elliptic curve together with a $\Gamma_1(3)$ level structure.
 The discriminant of this curve is $\Delta = (a_1^3-27a_3)a_3^3$, and $\Msl{3} \simeq \Proj \Z_{(2)}[a_1,a_3][\Delta^{-1}]$. Consequently, $\pi_* \TMF_1(3) = \Z_{(2)}[a_1,a_3][\Delta^{-1}]$. Lawson-Naumann show that the compactification $\bar{\mathcal{M}}_1(3)\simeq \Proj\Z_{(2)}[a_1,a_3]$ also admits a sheaf of $\Einf$-ring spectra, giving rise to a non-connective and non-periodic spectrum $\Tmf_1(3)$ with a gap in its homotopy allowing to take a connective cover $\tmf_1(3)$ which is an $\Einf$-ring spectrum with
 \[\pi_* \tmf_1(3) =\Z_{(2)}[a_1,a_3]. \]
 This spectrum is complex oriented such that the composite map of graded rings
 \[ \Z_{(2)} [v_1,v_2] \subset \BP_* \to (\MU_{(2)})_* \to \tmf_1(3)_* \]
 is an isomorphism \cite[Theorem 1.1]{LawsonNaumann}, where the $v_i$ are Hazewinkel generators. Of course, the map $\BP_* \to \tmf_1(3)_*$ classifies the $p$-typicalization of the formal group associated to the curve \eqref{eq:Ca1a3}, which starts as \cite[IV.2]{Silverman}, \cite{sage}:
 \begin{align*} 
 F(X,Y) &= X + Y -a_1X Y -2 a_3 X^3 Y -3 a_3 X^2 Y^2 +
-2 a_3 X Y^3 \\
&-2 a_1 a_3 X^4 Y -a_1 a_3 X^3 Y^2 
-a_1 a_3 X^2 Y^3 -2 a_1 a_3 X Y^4 + O(X,Y)^6.
\end{align*}
We used Sage to compute the logarithm of this formal group law, 
from which we read off the coefficients $l_i$ \cite[A2.1.27]{Ravenel} in front of $X^{2^i}$ as
\begin{align*}
l_1&= \frac{a_1}{2}, \qquad l_2 = \frac{a_1^3+2a_3}{4},\\
l_3 & = \frac{ a_1^7 + 30 a_1^4 a_3 + 30 a_1 a_3^2}{8}\dots.
\end{align*}

Now the formula \cite[A2.1.1]{Ravenel} $\displaystyle{ pl_n =\sum_{0\leq i <n } l_i v_{n-i}^{2^i} }$ (in which $l_0$ is understood to be $1$) allows us to recursively compute the map $\BP_* \to \tmf_1(3)_*$. For the first few values of $n$, we have that
\begin{align*}
v_1 \mapsto a_1, \qquad 
v_2 \mapsto a_3, \qquad
v_3 & \mapsto 7a_1a_3(a_1^3+a_3)\dots.
\end{align*}

We can do even more with this orientation of $\tmf_1(3)$, as 
 \[ \BP_* \BP \to \tmf_1(3)_* \tmf_1(3)_{\QQ}\]
 is a morphism of Hopf algebroids.
Recall that $\BP_*\BP=\Z_{(2)}[v_1,v_2,\dots][t_1,t_2,\dots]$ with $v_i$ and $t_i$ in degree $2(2^i-1)$ and the right unit is
$ \eta_R:\BP_* \to \BP_*\BP $ determined by the fact \cite[A2.1.27]{Ravenel} that \[\eta_R(l_n) = \sum_{0\leq i \leq n} l_i t_{n-i}^{2^i} \]
with $l_0=t_0=1$ by convention. On the other hand, \[\tmf_1(3)_* \tmf_1(3)_{\Q} = \Q[a_1,a_3,\ba_1, \ba_3]\] and the right unit $\tmf_1(3)_* \to \tmf_1(3)_*\tmf_1(3) $ sends $a_i$ to $\ba_i$. With computer aid from Sage, we can recursively compute the images of each $t_i$ in $\tmf_1(3)_* \tmf_1(3)$. As an example, we include here the first three values
\begin{equation}\label{eq:tiformulas}
\begin{aligned}
t_1 & \mapsto \frac{1}{2}({\ba_1-a_1}),\\
t_2 & \mapsto \frac{1}{8}( 4 \ba_3 + 2 \ba_1^3 - a_1\ba_1^2 + 2 a_1^2\ba_1 - 4a_3 - 3 a_1^3 ), \\
t_3 & \mapsto \frac{1}{128} (480 \ba_1 \ba_3^2 - 16 a_1\ba_3^2 + 480 \ba_1^4 \ba_3 - 16 a_1 \ba_1^3 \ba_3 + 8 a_1^2 \ba_1^2 \ba_3 - 16 a_1^3\ba_1\ba_3 \\
&+ 32 a_1 a_3 \ba_3 + 24 a_1^4 \ba_3 + 16 \ba_1^7 - 4 a_1\ba_1^6 + 4 a_1^2 \ba_1^5 - 4 a_3 \ba_1^4 - 11 a_1^3 \ba_1^4 + 32 a_1 a_3\ba_1^3 \\
&+ 24 a_1^4 \ba_1^3 - 32 a_1^2 a_3 \ba_1^2 - 22 a_1^5 \ba_1^2 + 32 a_1^3 a_3\ba_1 + 20 a_1^6 \ba_1 - 496 a_1 a_3^2 - 508 a_1^4 a_3 - 27 a_1^7 )
\end{aligned}
\end{equation}
and rather than urging the reader to analyze the terms, we simply point out the exponential increase of their number. In Section \ref{subsec:AF} we will use the Adams filtration to extract leading terms from these expressions, allowing us to extract meaningful information from these formulas.

\begin{rmk}
Just as we used $\bu_*\bu$ as a means of porting formulas in $BP_*BP$ to $\bo_*\bo$, so we are using $\tmf_1(3)_*\tmf_1(3)$ to analyze $\tmf_*\tmf$.  The reader might wonder why we do not give a complete analysis of $\tmf_1(3)_*\tmf_1(3)$.  In fact, such an analysis has recently been completed by Culver \cite{Culver}.
\end{rmk}

\subsection{The relationship between $\TMF_1(3)$ and $\TMF $ and their connective versions}
As we mentioned already, the forgetful map $f: \Msl{3}\to \M$ is \'etale; moreover, $f^* \omega = \omega$. As a consequence, we have a \v Cech descent spectral sequence
\[E_1=H^p (\Msl{3}^{\times_{\M}(q+1)}, \omega^{\otimes *}) \Rightarrow H^{p+q} (\M, \omega^{\otimes *}). \]
With it, the modular forms $H^0(\M,\omega^{\otimes *}) $ can be computed as the equalizer of the diagram
 \begin{align}\label{eq:equalizer}
  \xymatrix{ H^0 (\Msl{3} ,\omega^{\otimes *})  \ar@<0.5ex>[r]^-{p_1^*} \ar@<-0.5ex>[r]_-{p_2^*} &H^0(\Msl{3} \times_{\M}\Msl{3}, \omega^{\otimes *} ), }
  \end{align}
in which $p_1$ and $p_2$ are the left and right projection maps. 
The interpretation is that the $\M$-modular forms $MF_*$ are precisely the invariant $\Msl{3}$-modular forms.

To be more explicit, note that $\Msl{3} \times_{\M}\Msl{3}$ classifies tuples $((C,P), (C^\prime, P^\prime), \varphi)$ of elliptic curves with a point of order $3$ and an isomorphism $\varphi:C \to C^\prime$ of elliptic curves which does not need to preserve the level structures. This data is locally given by 
\begin{equation}
\begin{aligned}
C: \qquad & y^2+ a_1xy+a_3y = x^3,\\
C^\prime: \qquad &y^2+ a_1^\prime xy+a_3^\prime y = x^3,\\
\varphi: \qquad & x \mapsto u^{-2}x+ r
& y \mapsto u^{-3} y + u^{-2} sx +t,
\end{aligned}
\end{equation}
such that the following relations hold
\begin{equation}\label{eq:relations}
\begin{aligned}
&sa_1 - 3r + s^2 = 0,\\
&sa_3+(t+rs)a_1 - 3r^2+2st=0,\\
&r^3-ta_3-t^2-rta_1=0 ,\\
&a_1' = \eta_R(a_1) ,\\
&a_3' = \eta_R(a_3).
\end{aligned}
\end{equation}
(Note: For more details on this presentation of $\Msl{3}$, see the beginning of \cite[\S 4]{grpcohcalc}; the relations follow from the general transformation formulas in \cite[III.1]{Silverman} by observing that the coefficients $a_{even}$ must remain zero.  See also \cite{Bauer}, where $\Msl{3}$ is implicitly used to compute the $2$-primary descent spectral sequence for $\tmf$.)

Hence, the diagram \eqref{eq:equalizer} becomes
\[ \Z_{(2)}[a_1,a_3] \rightrightarrows \Z_{(2)}[a_1,a_3][u^{\pm 1},r,s,t]/(\sim) \]
(where $\sim$ denotes the relations \eqref{eq:relations}) with $p_1^*$ being the obvious inclusion and $p_2^*$ determined by
\begin{align*}
a_1 &\mapsto u(a_1+2s),\\
a_3 &\mapsto u^3(a_3+ra_1+2t),
\end{align*}
which is in fact a Hopf algebroid representing $\M_{(2)}$. Note that we do not need to localize at $2$ but only to invert $3$ to obtain this presentation. 

As a consequence of this discussion we can explicitly compute that the modular forms $MF_*$ are the subring of $MF_1(3)_*$ generated by
\begin{align}\label{eq:ModForms}
 c_4 =  a_1^4 - 24a_1a_3 ,\qquad \qquad c_6 = a_1^6 + 36 a_1^3 a_3 -216a_3^2, \qquad\text{and}\qquad \Delta=(a_1^3-27a_3)a_3^3,
\end{align}
which in particular determines the map $\TMF_* \to \TMF_1(3)_*$ on non-torsion elements.

\subsection{Adams filtrations}\label{subsec:AF}
The maps $\BP_* \to \tmf_1(3)_*$ and $\BP_*\BP \to \tmf_1(3)_*\tmf_1(3)$ respect the Adams filtration, which allows us to determine the Adams filtration on the right hand sides.
Recall that \[ AF(v_i) = 1 , \qquad i\geq 0\]
where as usual, $v_0 =2$. 
Consequently, $AF(a_1)=AF(a_3)=1$, which in turn implies via \eqref{eq:ModForms} that 
\begin{align} \label{eq:afs}
AF(c_4)=4, \qquad \qquad AF(c_6)=5, \qquad \text{and} \qquad AF(\Delta)=4. 
\end{align}
More precisely, modulo higher Adams filtration (we use $\sim$ to denote equality modulo terms in higher Adams filtration) we have
\begin{equation}\label{eq:cia1a3}
c_4 \sim a_1^4, \qquad \qquad c_6\sim 216 a_3^2 \sim 8a_3^2, \qquad \qquad 
\Delta \sim a_3^4.
\end{equation}

Note that the Adams filtration of each $t_i$ is zero.

\subsection{Supersingular elliptic curves and $\K(2)$-localizations} 

At the prime $2$, there is a unique isomorphism class of supersingular elliptic curves; one representative is the Weierstrass curve
\begin{align*}
C: \qquad y^2+y=x^3
\end{align*}
over $\F_2$.
Recall that a supersingular elliptic curve is one whose formal completion at the identity section $\hat C$ is a formal group of height two.\footnote{As opposed to an ordinary elliptic curve whose formal completion has height one. These two are the only options.} Under the natural map $\M \to \Mfg$ from the moduli stack of elliptic curves to the one of formal groups sending an elliptic curve to its formal completion at the identity section, the supersingular elliptic curves (in fixed characteristic) are sent to the (unique up to isomorphism, by Cartier's theorem \cite[Appendix B]{Ravenel}) formal group of height two in that characteristic.

Let $\Mss$ denote a formal neighborhood of the supersingular point $C$ of $\M$, and let $\Htwo$ denote a formal neighborhood of the characteristic $2$ point of height two of $\Mfg$. Formal completion yields a map $\Mss\to\Htwo$ which is used to explicitly describe the $\K(2)$-localization of $\TMF$ (or equivalently, $\tmf$) in terms of Morava $\E$-theory.

The formal stack $\Htwo $ has a pro-Galois cover by $\Spf \W(\F_4)[[u_1]] $ for the extended Morava stabilizer group $\G_2$. The Goerss-Hopkins-Miller theorem implies in particular that this quotient description of $ \Htwo $ has a derived version, namely the stack $ \Spf \E_2 \mmod \G_2 $, where $ \E_2 $ is a Lubin-Tate spectrum of height two. As we are working with elliptic curves, we take the Lubin-Tate spectrum associated to the formal group $\hat C$ over $\F_4$, and $\G_2=\Aut_{\F_4}(\hat C) \rtimes Gal(\F_4/\F_2) $. 
 
Let $G$ denote the automorphism group of $C$; it is a finite group of order $48$ given as an extension of the binary tetrahedral group with the Galois group of $\F_4/\F_2$. Then $G$ embeds in $\G_2$ as a maximal finite subgroup and $\Spf \E_2$ is a Galois cover of $\Mss$ for the group $G$. In particular, taking sections of the structure sheaf $\Sh{O}^{top}$ over $\Mss$ gives the $\K(2)$-localization of $\TMF$ which is equivalent to $\E_2^{hG}$. Moreover, we have $\K(2)$-local equivalences

 \[ (\TMF \wedge \TMF)_{\K(2)} \simeq \Hom^c(\G_2/G,\E_2)^{hG} \simeq \prod_{x\in G \backslash (\G_2) /G} \E_2^{h(G \cap xGx^{-1})}.\]
 
 The decomposition on the right hand side is interesting though we will not pursue it further in this work. The interested reader is referred to Peter Wear's explicit calculation of the double cosets in \cite{PeterWear}.

%% file: tmfcoop_ass.tex

\section{The Adams spectral sequence for $\tmf _*\tmf$ and $\bo$-Brown-Gitler modules}\label{sec:ass}

\emph{Recall that we are concerned with the prime $2$, hence everything is implicitly $2$-localized.}

\subsection{Rational calculations}\label{sec:rational}

Recall that we have
$$ \tmf_*\tmf_\QQ \cong \QQ[c_4, c_6, \bar{c}_4, \bar{c}_6]$$
and consider the (collapsing) $v_0$-inverted ASS
$$ \bigoplus_j v_0^{-1} \Ext_{A(2)_*}(\si{8j} \bou_j) \Rightarrow \tmf_* \tmf \otimes \QQ_2. $$
 In this section we explain the decomposition imposed on the $E_\infty$-term of this spectral sequence from the decomposition on the $E_2$-term.  In particular, given a torsion-free element $x \in \tmf_*\tmf$, this will allow us to determine which $\bo$-Brown-Gitler module detects it in the $E_2$-term of the ASS for $\tmf \wedge \tmf$.

Recall from Section~\ref{sec:review} that $\tmf_1(3) \simeq \BP\bra{2}$.  In particular, we have
$$ H^*(\tmf_1(3)) \cong A \mmod  E[Q_0, Q_1, Q_2]. $$
We begin by studying the map between $v_0$-inverted ASS's induced by the map $\tmf \rightarrow \tmf_1(3)$    
$$ 
\xymatrix{
v_0^{-1}\Ext^{*,*}_{A(2)_*}(\FF_2) \ar@{=>}[r] \ar[d] & 
\pi_* \tmf \otimes \QQ_2 \ar[d]
\\
v_0^{-1}\Ext^{*,*}_{E[Q_0, Q_1, Q_2]_*}(\FF_2) \ar@{=>}[r] & 
\pi_* \tmf_1(3) \otimes \QQ_2.
} $$
We have 
$$ v_0^{-1} \Ext^{*,*}_{E[Q_0, Q_1, Q_2]_*}(\FF_2) \cong \FF_2[v_0^{\pm 1}, v_1, v_2], $$
where the $v_i$'s have $(t-s,s)$ bidegrees:
\begin{align*}
\abs{v_0} & = (0, 1), \\
\abs{v_1} & = (2, 1), \\
\abs{v_2} & = (6,1). 
\end{align*}
Recall from Section~\ref{sec:review} that $\pi_* \tmf_1(3)_\QQ = \QQ[a_1, a_3]$, and that
\begin{align*}
v_1 & = [a_1], \\
v_2 & = [a_3]. 
\end{align*}
Of course $\pi_*\tmf_{\QQ} = \QQ[c_4, c_6]$, with corresponding localized Adams $E_2$-term
$$ v_0^{-1} \Ext^{*,*}_{A(2)_*}(\FF_2) \cong \FF_2[v_0^{\pm 1}, c_4, c_6], $$
where the $[c_i]$'s have $(t-s,s)$ bidegrees
\begin{align*}
\abs{[c_4]} & = (8, 4), \\
\abs{[c_6]} & = (12, 5).
\end{align*}
Recall also from Section~\ref{sec:review} that the formulas for $c_4$ and $c_6$ in terms of $a_1$ and $a_3$ imply that the map of $E_2$-terms of spectral sequences above is injective, and is given by
\begin{equation}\label{eq:c4c6ASS}
\begin{split}
[c_4] & \mapsto [a_1^4], \\
[c_6] & \mapsto [8a_3^2].
\end{split}
\end{equation}
Corresponding to the isomorphism
$$ \pi_* \tmf_\QQ \cong \mr{H}\QQ_* \tmf $$
there is an isomorphism of localized Adams $E_2$-terms
$$ v_0^{-1} \Ext_{A(2)}(\FF_2) \cong v_0^{-1} \Ext_{A(0)} ((A\mmod A(2))_*). $$
Since the decomposition
$$ A\mmod A(2)_* \cong \bigoplus_{j} \si{8j} \bou_j $$ 
is a decomposition of $A(2)_*$-comodules, it is in particular a decomposition of $A(0)_*$-comodules, and therefore there is a decomposition
\begin{equation}\label{eq:Qsplit} 
v_0^{-1} \Ext_{A(2)_*} (\FF_2) \cong \bigoplus_j v_0^{-1} \Ext_{A(0)_*}(\si{8j} \bou_j).
\end{equation}

\begin{prop}
Under the decomposition (\ref{eq:Qsplit}), we have
\begin{align*}
v_0^{-1} \Ext_{A(0)_*}(\Sigma^{8j} \bou_j) & = \FF_2[v_0^{\pm 1}]\{ [c_4^{i_1}c_6^{i_2}] \: : \: i_1 + i_2 = j\} \\
& \subset v_0^{-1}\Ext_{A(2)_*}(\FF_2).
\end{align*}
\end{prop}

\begin{proof}
Statement~(2) of the proof of Lemma~\ref{lem:HZ_i} implies that we have
$$ v_0^{-1} \Ext_{A(0)_*}(\bou_j) \cong \FF_2[v_0^{\pm 1}]\{ \xib_1^{4i} \: : \: 0 \le i \le j \}.$$
Using the map (\ref{eq:splittingmap}), we deduce that we have
\begin{align*}
v_0^{-1} \Ext_{A(0)_*}(\si{8j} \bou_j) & \cong \FF_2[v_0^{\pm 1}]\{ \xib_1^{8i_1}\xib_2^{4i_2} \: : \: i_1 + i_2 = j \} \\
& \subset \Ext_{A(0)_*}((A\mmod A(2))_*).
\end{align*}
Consider the diagram:
\begin{equation}\label{eq:Qtmfdiag}
\xymatrix{
H_* \tmf \ar[r] & H_* \tmf_1(3) & \BP_* \BP \ar[l] \ar[d] 
\\
\HZ_*\tmf \ar[u] \ar[d] \ar[r] & \HZ_* \tmf_1(3) \ar[u] \ar[d] & \tmf_1(3)_* \tmf_1(3) \ar[l] \ar[d] 
\\
\rm{H}\QQ_* \tmf \ar[r] & \rm{H}\QQ_* \tmf_1(3) & \tmf_1(3)_* \tmf_1(3)_\QQ. \ar[l] 
}
\end{equation}
The map
$$ \BP_*\BP \rightarrow H_* \tmf_1(3) \cong \FF_2[\xib_1^2, \xib_2^2, \xib_3^2, \xib_4, \ldots ] $$
sends $t_i$ to $\xib_i^2$.  Thus the elements 
\begin{align*}
\xib_1^{8i_1} \xib_2^{4i_2} & \in H_* \tmf, \\
t_1^{4i_1}t_2^{2i_2} & \in \BP_*\BP
\end{align*}
have the same image in $H_* \tmf_1(3)$.
However, using the formulas of Section~\ref{sec:review}, we deduce that the images of $t_1$ and $t_2$ in
$$ \tmf_1(3)_* \tmf_1(3)_\QQ = \QQ[a_1, a_3, \bar{a}_1, \bar{a}_3] $$
are given by
\begin{align*}
t_1 & \mapsto (\bar{a}_1 + a_1)/2,  \\
t_2 & \mapsto ( 4\bar{a}_3 - a_1\bar{a}_1^2 - 4a_3 - a_1^3  ) / 8 + \text{terms of higher Adams filtration}.
\end{align*}
Since the map
$$ \tmf_1(3)_* \tmf_1(3)_\QQ \rightarrow \mr{H}\QQ_* \tmf_1(3) = \QQ[a_1, a_3] $$
of diagram~(\ref{eq:Qtmfdiag}) sends $\bar{a}_i$ to $a_i$ and $a_i$ to zero, we deduce that the image of $t_1$ and $t_2$ in $\mr{H}\QQ_* \tmf_1(3)$ is
 \begin{align*}
t_1 & \mapsto a_1/2,  \\
t_2 & \mapsto a_3/2 + \text{terms of higher Adams filtration}.
\end{align*}
It follows that under the map of $v_0$-localized ASS's induced by the map $\tmf \rightarrow \tmf_1(3)$
$$ v_0^{-1} \Ext_{A(2)_*}(\FF_2) \rightarrow v_0^{-1}\Ext_{E[Q_0, Q_1, Q_2]_*} (\FF_2) $$
we have
$$
\xib_1^{8i_1} \xib_2^{4i_2} \mapsto [a_1/2]^{4i_1} [a_3/2]^{2i_2}.
$$
Therefore, by (\ref{eq:c4c6ASS}), we have the equality (in $v_0^{-1}\Ext_{A(0)_*}((A\mmod A(2))_*)$)
$$ \xib_1^{8i_1} \xib_2^{4i_2} = [c_4/16]^{i_1}[c_6/32]^{i_2} $$
and the result follows.
\end{proof}

Corresponding to the K\"unneth isomorphism for $\mr{H}\QQ$, there is an isomorphism
$$ v_0^{-1}\Ext_{A(0)_*}(M \otimes N) \cong v_0^{-1} \Ext_{A(0)_*}(M) \otimes_{\FF_2[v_0^{\pm 1}]} v_0^{-1}\Ext_{A(0)_*}(N). $$
In particular, since the maps
$$ v_0^{-1}\Ext(\tmf \wedge \si{8j} \bo_j) \rightarrow v_0^{-1} \Ext(\tmf \wedge \tmf) $$
can be identified with the maps
\begin{multline*}
v_0^{-1}\Ext_{A(0)_*}((A\mmod A(2))_*) \otimes_{\FF_2[v_0^{\pm 1}]} v_0^{-1}\Ext_{A(0)_*}(\Sigma^{8j} \bou_j) 
\\
\rightarrow  v_0^{-1}\Ext_{A(0)_*}((A\mmod A(2))_*) \otimes_{\FF_2[v_0^{\pm 1}]} v_0^{-1}\Ext_{A(0)_*}((A\mmod A(2))_*)
\end{multline*}
we have the following corollary.

\begin{cor}
The map
$$ v_0^{-1}\Ext(\tmf \wedge \si{8j} \bo_j) \rightarrow v_0^{-1} \Ext(\tmf \wedge \tmf) $$
obtained by localizing  (\ref{eq:tmfextsplit}) is the canonical inclusion
$$ \FF_2[v_0^{\pm 1}, [c_4], [c_6]] \{ [\bar{c}_4]^{i_1} [\bar{c}_6]^{i_2} \: : \: i_1 + i_2 = j \} \hookrightarrow \FF_2[v_0^{\pm 1}, [c_4], [c_6], [\bar{c}_4], [\bar{c}_6]]. $$
\end{cor}

\subsection{Inductive computation of $\Ext_{A(2)_*}(\bou_j)$}\label{sec:boSES}

The exact sequences (\ref{eq:boSES1}),(\ref{eq:boSES2}) provide an inductive method of computing $\Ext_{A(2)_*}(\bou_j)$ in terms of $\Ext_{A(1)_*}$-computations and $\Ext_{A(2)_*}(\ul{\bo}_1^i)$.  

We give some low dimensional examples. We shall use the shorthand
$$ M \Leftarrow \bigoplus M_i[k_i]  $$
to denote the existence of a spectral sequence
$$  \bigoplus \Ext_{A(2)_*}^{s-k_i,t+k_i}(M_i) \Rightarrow \Ext_{A(2)_*}^{s,t}(M). $$
In the notation above, we shall abbreviate $M_i[0]$ as $M_i$.
We have
\begin{equation}\label{eq:boSES_low}
\begin{split}
\Sigma^{16} \bou_2 & \Leftarrow \Sigma^{16} (A(2)\mmod A(1))_* \oplus \Sigma^{24}\bou_1 \oplus \Sigma^{32} \FF_2[1],
\\
\Sigma^{24} \bou_3 & \Leftarrow
\Sigma^{24} (A(2)\mmod A(1))_* \oplus \Sigma^{32} \bou_1^{2} ,
\\
\si{32} \bou_4 & \Leftarrow
(A(2)\mmod A(1))_* \otimes \left( \si{32} \tmfu_1 \oplus \si{48} \FF_2 \right) \oplus \si{56} \bou_1 \oplus \si{56} \bou_1[1] \oplus \si{64}\FF_2[1],
\\
\si{40}\bou_5 & \Leftarrow 
(A(2)\mmod A(1))_* \otimes \left( \si{40} \tmfu_1 \oplus \si{56} \bou_1 \right)  \oplus \si{64}\bou_1^{2} \oplus \si{72} \bou_1[1],
\\
\si{48} \bou_6 & \Leftarrow 
(A(2)\mmod A(1))_* \otimes \left(\si{48} \tmfu_2 \oplus \si{72} \FF_2 \oplus \si{80} \FF_2[1] \right) 
\\
& \quad \quad \oplus \si{80} \bou_1^2 \oplus \si{88} \bou_1[1] \oplus  \si{96}\FF_2[2],
\\
\si{56} \bou_7 & \Leftarrow 
(A(2)\mmod A(1))_* \otimes \left(\si{56} \tmfu_2 \oplus \si{80} \bou_1 \right) \oplus \si{88} \bou_1^3,
\\
\si{64} \bou_8 & \Leftarrow 
(A(2)\mmod A(1))_* \otimes \left(\si{64}\tmfu_3 \oplus \si{96}\tmfu_1 \oplus \si{112}\FF_2 \oplus \si{104} \FF_2[1] \right)  \\
& \quad \quad \oplus \si{112}\bou_1^2[1] \oplus \si{120}\bou_1 \oplus \si{120} \bou_1[1] \oplus \si{128}\FF_2[1].
\end{split}
\end{equation}

In practice, these spectral sequences tend to collapse.  In fact, in the range computed explicitly in this paper, there are no differentials in these spectral sequences, and the authors have not yet encountered any differentials in these spectral sequences.  These spectral sequences collapse with $v_0$-inverted, for dimensional reasons.

In principle, the exact sequences (\ref{eq:boSES1}) and (\ref{eq:boSES2}) allow one to inductively compute $\Ext_{A(2)_*}(\bou_j)$ given $\Ext_{A(2)_*}(\bou_1^{\otimes k})$, where $\bou_1$ is depicted in Figure \ref{fig:bo1}.
\begin{figure}[h]
$$
\xymatrix@R-2em@C-2em{
\xib_3 & \circ \ar@{-}[d]^{\sq^1}
\\
\xib_2^2 & \circ \ar@/^1pc/@{-}[dd]^{\sq^2}
\\ \\
\xib_1^4 & \circ \ar@{-} `r[dddd] `[dddd]^{\sq^4} [dddd]
\\ \\ \\ \\
1 & \circ
}
$$
\caption{$\bou_1$}\label{fig:bo1}
\end{figure}
The problem is that, unlike the $A(1)$-case, we do not have a  closed form computation of $\Ext_{A(2)_*}(\bou_1^{\otimes k})$.  These computations for $k \le 3$ appeared in \cite{BHHM} (the cases of $k = 0, 1$ appeared elsewhere). We include in Figures \ref{fig:A2andbo12} through \ref{fig:bo5andbo6}  the charts for $\si{8j} \bou_j$, for $0 \leq j\leq 6$, as well as $\Sigma^8\bou_1^2$ in dimensions $\le 64$. 
In these figures, the different contributions to $\bou_j$ coming from the different summands of the $E_1$-term of the spectral sequences \ref{eq:boSES_low} are denoted with different colors.

\begin{figure}
\centering
	\begin{subfigure}{0.49\textwidth}
	\includegraphics[height =\textheight]{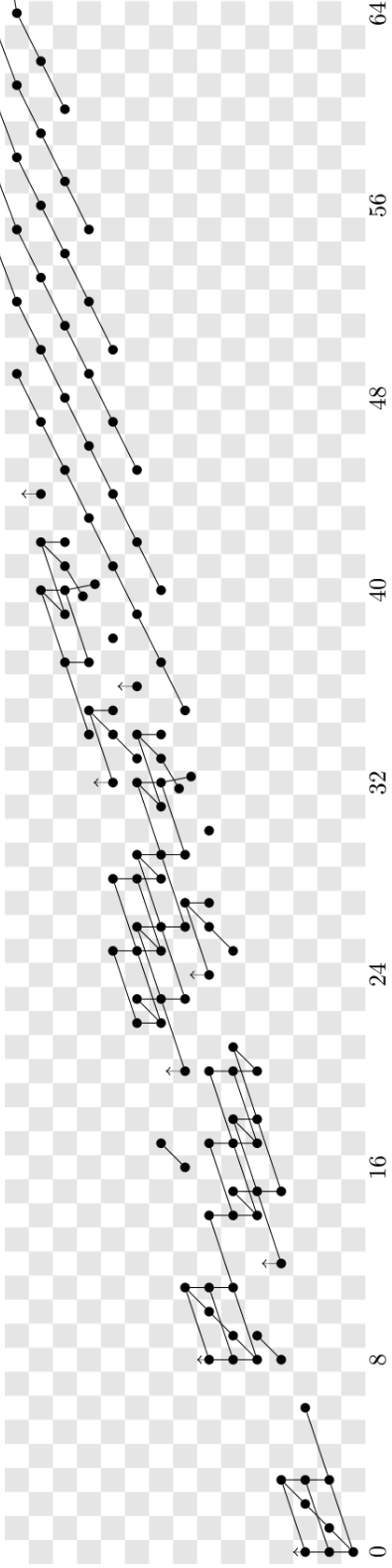}
	\caption{$\bou_0$}
	\end{subfigure}
	\begin{subfigure}{0.49\textwidth}
	\includegraphics[height =\textheight]{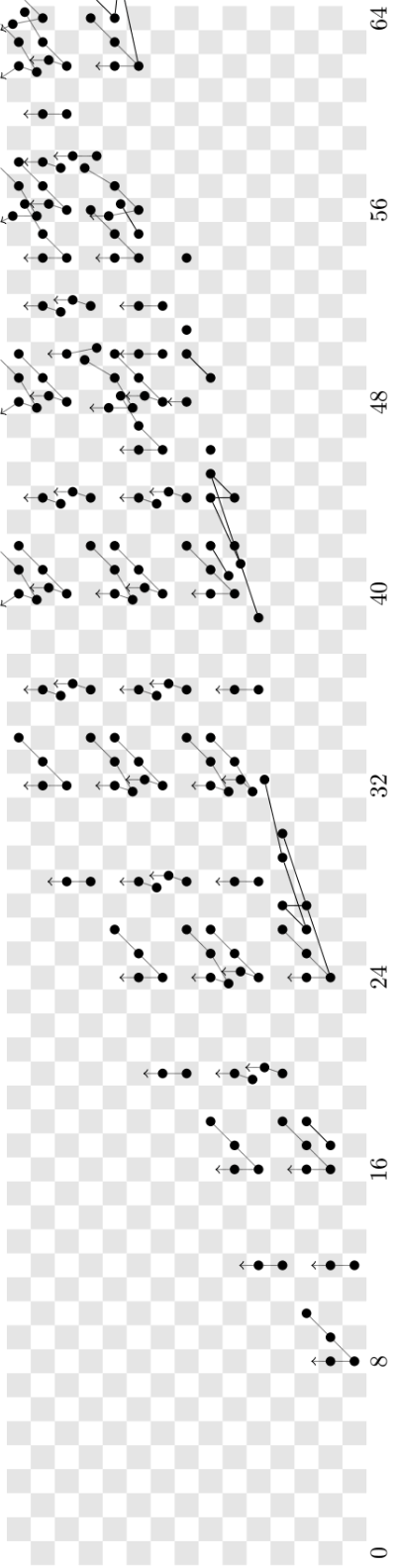}	
	\caption{$\Sigma^8\bou_1^2$}
	\end{subfigure}
\caption{}
\label{fig:A2andbo12}
\end{figure}

\begin{figure}
\centering
	\begin{subfigure}{0.49\textwidth}
	\includegraphics[height =\textheight]{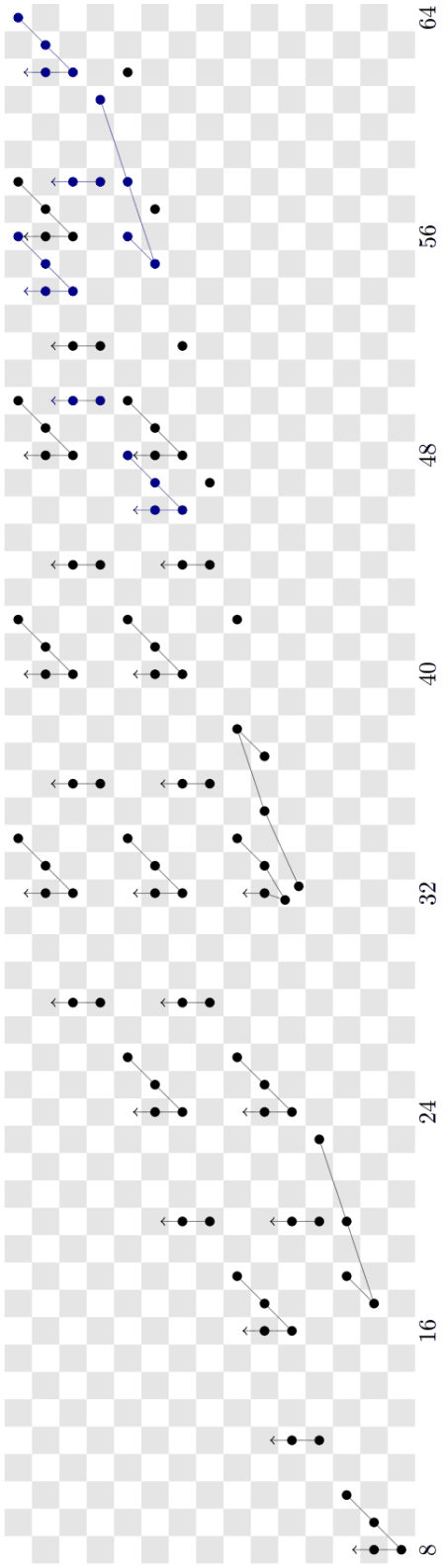}
	\caption{$\Sigma^8\bou_1$}
	\end{subfigure}
	\begin{subfigure}{0.49\textwidth}
	\includegraphics[height =\textheight]{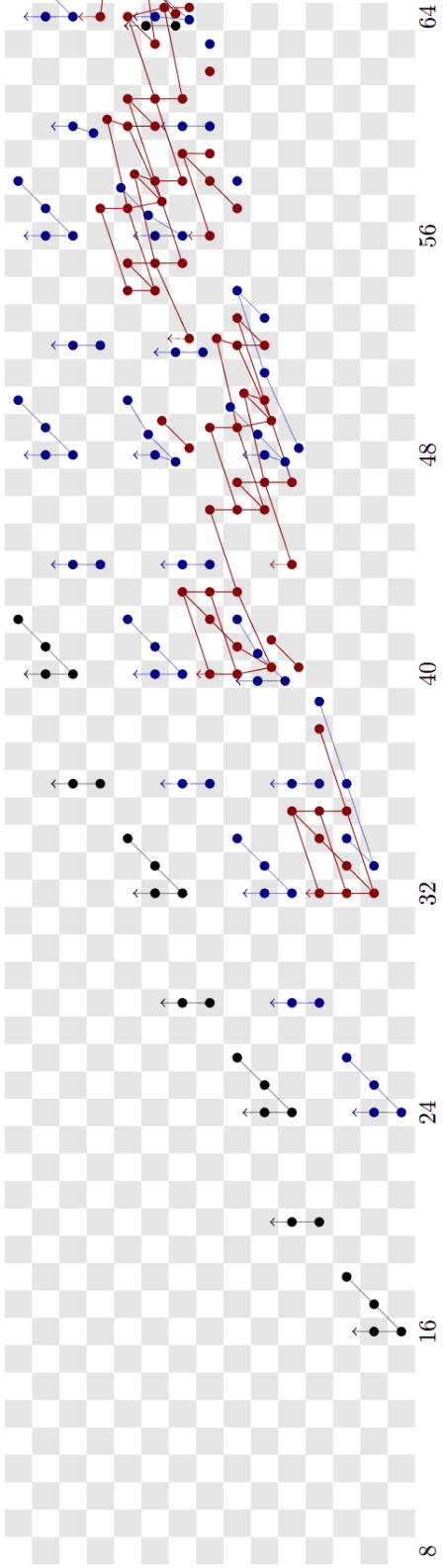}
	\caption{$\Sigma^{16}\bou_2$}
	\end{subfigure}
\caption{}\label{fig:bo1andbo2}
\end{figure}

\begin{figure}
\centering
	\begin{subfigure}{0.49\textwidth}
	\includegraphics[height =\textheight]{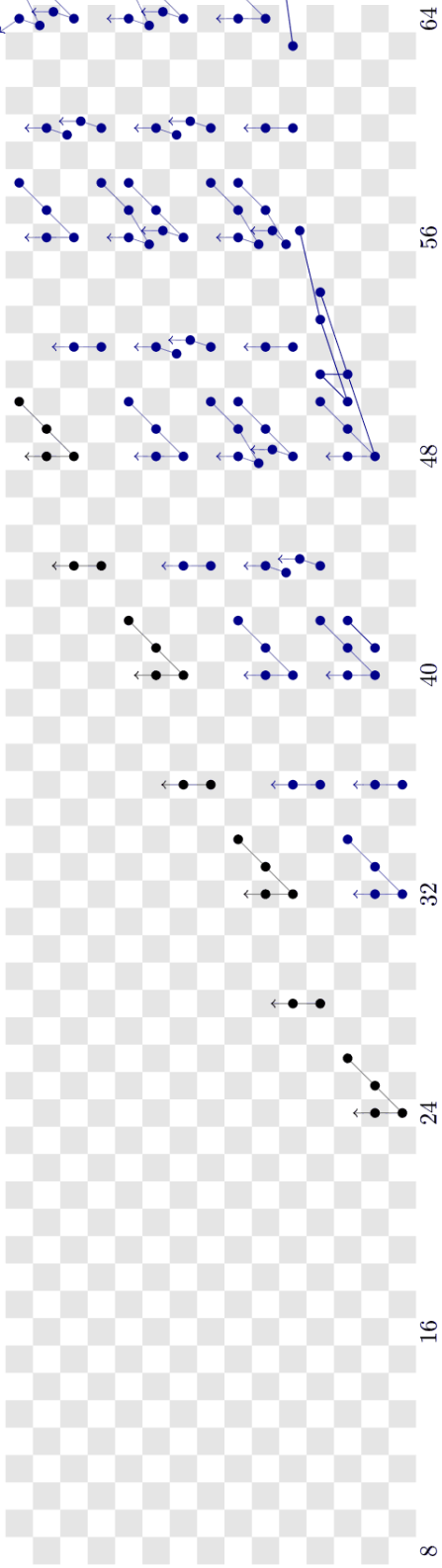}
	\caption{$\Sigma^{24}\bou_3$}
	\end{subfigure}
	\begin{subfigure}{0.49\textwidth}
	\includegraphics[height =\textheight]{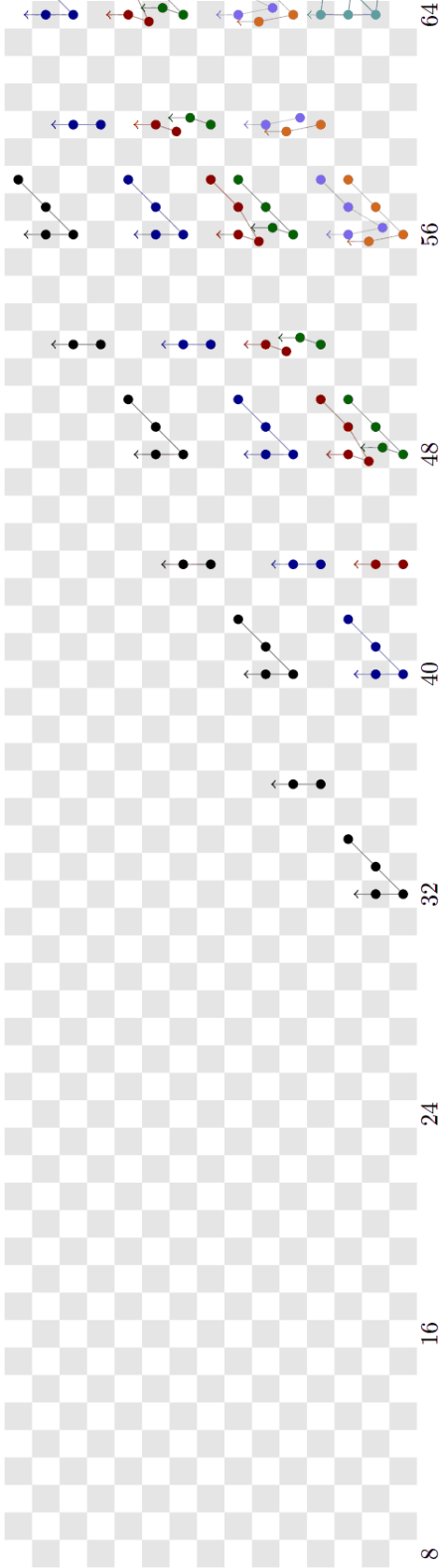}
	\caption{$\Sigma^{32}\bou_4$}
	\end{subfigure}
\caption{}\label{fig:bo3andbo4}
\end{figure}

\begin{figure}
\centering
	\begin{subfigure}{0.49\textwidth}
	\includegraphics[height =\textheight]{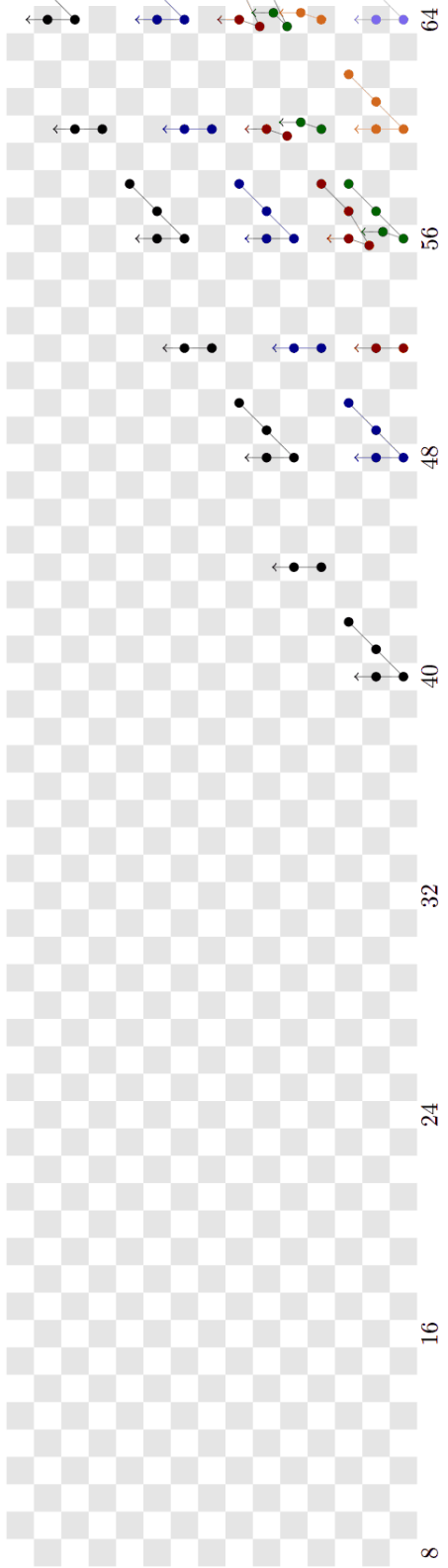}
	\caption{$\Sigma^{40}\bou_5$}
	\end{subfigure}
	\begin{subfigure}{0.49\textwidth}
	\includegraphics[height =\textheight]{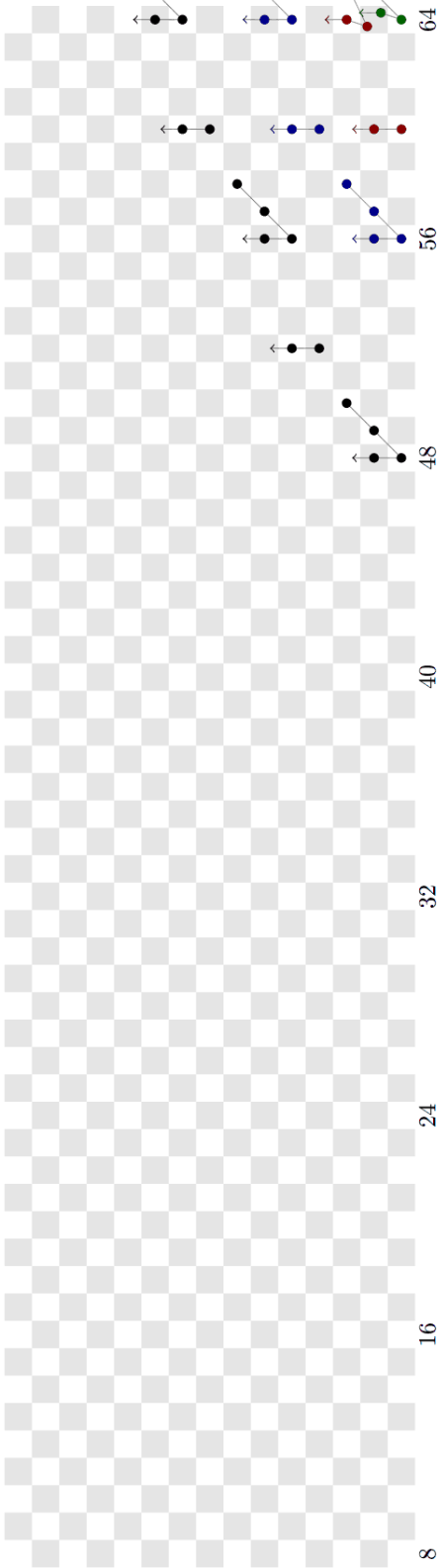}
	\caption{$\Sigma^{48}\bou_6$}
	\end{subfigure}
\caption{}\label{fig:bo5andbo6}
\end{figure}

\subsection{Rational behavior of the exact sequences}\label{sec:rationalgens}

We finish this section with a discussion on how to identify the generators of  $\frac{\Ext_{A(2)_*}(\si{8j} \bou_j)}{v_0-tors}$.
On one hand, the inclusion
$$
\xymatrix@C-1em{
\frac{\Ext_{A(2)_*}(\si{8j} \bou_j)}{v_0-tors} 
\ar@{^{(}->}[r] 
&  
v_0^{-1}\Ext_{A(2)_*}(\si{8j} \bou_j)  \ar@{^{(}->}[d]
&{}\save[]+<2.2cm,0cm>* {=\FF_2[v_0^{\pm 1}, [c_4], [c_6]]\{ \xib^{8i_1}\xib^{4i_2} \: : \: i_1 + i_2  = j \}}
\restore \\
& 
 v_0^{-1}\Ext_{A(2)_*}((A\mmod A(2))_*)
}
$$
discussed in Section~\ref{sec:rational}
informs us that the $h_0$-towers of $\Ext_{A(2)_*}(\si{8j}\bou_j)$ are all generated by
$$ h^k_0 [c_4]^{p} [c_6]^q \xib_1^{8i_1} \xib_2^{4i_2} $$
for appropriate (possibly negative) values of $k$ depending on $i_1, i_2, p,$ and $q$.

The problem is that the terms
\begin{align}
v_0^{-1} \Ext_{A(2)}(\si{16j}(A(2)\mmod A(1))_* \otimes \ul{\tmf}_{j-1}) & \subset v_0^{-1}\Ext_{A(2)_*}(\si{16j} \bou_{2j}),
\label{eq:boSES12} \\
v_0^{-1} \Ext_{A(2)}(\si{16j+8}(A(2)\mmod A(1))_* \otimes \ul{\tmf}_{j-1}) & \subset v_0^{-1}\Ext_{A(2)_*}(\si{16j+8} \bou_{2j+1}) \label{eq:boSES22}
 \end{align}
in the short exact sequences (\ref{eq:boSES1}), (\ref{eq:boSES2}) are not free over $\FF_2[v_0^{\pm 1}, [c_4], [c_6]]$ (however, they are free over $\FF_2[v_0^{\pm 1}, [c_4]]$).

We therefore instead identify the generators of $v_0^{-1}\Ext_{A(2)_*}((A\mmod A(2))_*)$ corresponding to the generators of (\ref{eq:boSES12}) and (\ref{eq:boSES22}) as modules over $\FF_2[v_0^{\pm 1}, [c_4]]$, as well as those generators coming (inductively) from
\begin{align}
v_0^{-1} \Ext_{A(2)_*}(\si{24j} \bou_j) & \subset v^{-1}_0\Ext_{A(2)_*}(\si{16j}\bou_{2j}),
\label{eq:boSES11}\\
v_0^{-1} \Ext_{A(2)_*}(\si{24j+8} \bou_j \otimes \bou_1) & \subset v^{-1}_0\Ext_{A(2)_*}(\si{16j+8}\bou_{2j+1}) \label{eq:boSES21}
\end{align}
in the following two lemmas, whose proofs are immediate from the definitions of the maps in (\ref{eq:boSES1}), (\ref{eq:boSES2}).

\begin{lem}\label{lem:boSES_2}
The summands (\ref{eq:boSES12}) (respectively (\ref{eq:boSES22})) are generated, as modules over $\FF_2[v_0^{\pm 1}, [c_4]]$, by the elements
$$ \xib_1^a \xib_2^{8i_1} \xib_3^{4i_3}, \: \xib_1^{a-8} \xib_2^{8i_1+4} \xib_3^{4i_3} 
\in (A\mmod A(2))_*,
$$
with $i_1 + i_2 \le j-1$ and $a = 16j - 16i_1 - 16i_2$ (respectively $a = 16j+8 - 16i_1 - 16i_2$).
\end{lem}

\begin{lem}\label{lem:boSES_1}
Suppose inductively (via the exact sequences (\ref{eq:boSES1}),(\ref{eq:boSES2})) that the summand
$$ v_0^{-1}\Ext_{A(2)_*}(\si{8j}\bou_j) \subset v_0^{-1}\Ext_{A(2)_*}((A\mmod A(2))_*) $$
has generators of the form
$$ \{ \xib_1^{i_1}\xib_2^{i_2}\ldots \}. $$
Then the summand (\ref{eq:boSES11}) is generated by
$$ \{ \xib_2^{i_1} \xib_3^{i_2} \cdots \} $$
and the summand (\ref{eq:boSES21}) is generated by
$$ \{ \xib_2^{i_1} \xib_3^{i_2} \cdots \} \cdot \{ \xib_1^8, \xib_2^4 \}. $$
\end{lem}

The remaining term
\begin{equation}\label{eq:boSES13}
v_0^{-1}\Ext_{A(2)_*}(\si{24j+8}\bou_{j-1}[1]) \subset v_0^{-1} \Ext_{A(2)_*}(\bou_{2j}) 
\end{equation}
coming from (\ref{eq:boSES1}) is handled by the following lemma.

\begin{lem}\label{lem:boSES13}
Consider the summand
$$
v_0^{-1}\Ext_{A(1)_*}(\si{24j-8}\bou_{j-1}) \subset   
v_0^{-1} \Ext_{A(1)_*}(\si{16j}\ul{\tmf}_{j-1}) \subset v_0^{-1}\Ext_{A(2)_*}(\si{16j} \bou_{2j})
$$
generated as a module over $\FF_2[v_0^{\pm 1}, [c_4]]$ by the generators
$$
\xib^{16}_1 \xib^{8i_1}_2 \xib_3^{4i_2}, \: \xib^{8}_1 \xib^{8i_1+4}_2 \xib_3^{4i_2} 
\in (A\mmod A(2))_*,
$$
with $i_1 + i_2 = j-1$. 
Let $x_i$ ($0 \le i \le j-1$) be the generator of the summand (\ref{eq:boSES13}), as a module over $\FF_2[v_0^{\pm 1}, [c_4], [c_6]]$ corresponding to the generator $\xib_1^{4i} \in \bou_{j-1}.$
Then we have
$$ [c_6]\xib^{8}_1 \xib_2^{8i_1+4} \xib^{4i_2}_3 = v_0^4 x_{i_2} + \cdots $$
in $v_0^{-1}\Ext_{A(2)_*}(\si{16j} \bou_{2j})$, where the additional terms not listed above all come from the summand
$$ v_0^{-1} \Ext_{A(2)_*}(\si{24j}\bou_{j}) \subset v_0^{-1} \Ext_{A(2)_*}(\si{16j} \bou_{2j}). $$
\end{lem}

\begin{proof}
This follows from the definition of the last map in (\ref{eq:boSES1}), together with the fact that with $v_0$ inverted, the cell $\xib_1^{4}\xib_2^2\xib_3 \in (A(2)\mmod A(1))_*$ attaches to the cell $\xib_1^4$ with attaching map $[c_6]/v_0^4$.
\end{proof}

Lemmas~\ref{lem:boSES_2}, \ref{lem:boSES_1}, and \ref{lem:boSES13} give an inductive method of identifying a collection of generators for $v_0^{-1}\Ext_{A(2)_*}(\bou_j)$, which are compatible with the exact sequences (\ref{eq:boSES1}), (\ref{eq:boSES2}).  We tabulate these below for the decompositions arising from the spectral sequences (\ref{eq:boSES_low}). For those summands of the form $(A(2)\mmod A(1))_* \otimes -$ these are generators over $\FF_2[v_0^{\pm 1}, [c_4]]$, for the other summands these are generators over $\FF_2[v_0^{\pm 1}, [c_4], [c_6]]$.
\begin{alignat*}{3}
\bou_0 : & \quad & \FF_2: & \quad && 1 \\
\si{8}\bou_1: && \si{8} \bou_1: &&& \xib_1^{8}, \xib_2^4 \\
\Sigma^{16} \bou_2: 
&& \Sigma^{16} (A(2)\mmod A(1))_*: &&& \xib_1^{16}, \xib_1^8 \xib_2^4 \\
&& \Sigma^{24}\bou_1: &&& \xib_2^8, \xib_3^4 \\
&& \Sigma^{32}\FF_2[1]: &&& v_0^{-4}[c_6] \xib_1^8 \xib_2^4 + \cdots \\
\Sigma^{24} \bou_3: 
&& \Sigma^{24} (A(2)\mmod A(1))_*: &&& \xib_1^{24}, \xib_1^{16} \xib_2^4 \\
&& \Sigma^{32} \bou_1^{2}: &&& \{ \xib_2^{8}, \xib_3^4 \}\cdot\{\xib_1^8, \xib_2^4 \} \\ 
\si{32} \bou_4: 
&& \si{32} (A(2)\mmod A(1))_* \otimes  \tmfu_1: &&&  \xib_1^32, \xib_1^{24}\xib_2^4, \xib_1^{16} \xib_2^8, \xib_1^8 \xib_2^{12}, \xib_1^{16} \xib_3^4, \xib_1^{8} \xib_2^4 \xib_3^4 \\
&& \si{48} (A(2)\mmod A(1))_*: &&& \xib_2^{16}, \xib_2^8\xib_3^4 \\
&& \si{56} \bou_1: &&& \xib_3^8, \xib_4^4 \\
&& \si{64}\FF_2[1]: &&& v_0^{-4}[c_6] \xib_2^8\xib_3^4 + \cdots  \\
&& \si{56} \bou_1[1]: &&& v_0^{-4}[c_6]\xib_1^8 \xib_2^{12} + \cdots , v_0^{-4}[c_6]\xib_1^{8} \xib_2^4 \xib_3^4 + \cdots \\
\si{40}\bou_5: 
&& \si{40} (A(2)\mmod A(1))_* \otimes \tmfu_1: &&& \xib_1^{40}, \xib_1^{32} \xib_2^4, \xib_1^{24} \xib_2^8, \xib_1^{16} \xib_2^{12}, \xib_1^{24} \xib_3^4, \xib_1^{16} \xib_2^4 \xib_3^4 \\
&& \si{56} (A(2)\mmod A(1))_* \otimes\bou_1: &&& \{ \xib_2^{16}, \xib_2^8 \xib_3^4 \}\cdot \{ 
\xib_1^8, \xib_2^4 \} \\
&& \si{64}\bou_1^{2}: &&& \{ \xib_3^8, \xib_4^4 \} \cdot \{ \xib_1^8, \xib_2^4 \}\\
&& \si{72} \bou_1[1]: &&& \{ v_0^{-4}[c_6] \xib_2^8 \xib_3^4 + \cdots \} \cdot \{ \xib_1^8, \xib_2^4 \}\\
\si{48} \bou_6:
&& \si{48} (A(2)\mmod A(1))_* \otimes \tmfu_2: &&& 
\xib_1^{48}, \xib_1^{40} \xib_2^4, \xib_1^{32} \xib_2^8, \xib_1^{24} \xib_2^{12}, \xib_1^{32} \xib_3^4, \xib_1^{24} \xib_2^4 \xib_3^4,  \\
&& &&& \quad \xib_1^{16} \xib_2^{16}, \xib_1^8 \xib_2^{20}, \xib_1^{16} \xib_2^8 \xib_3^4, \xib_1^{8} \xib_2^{12} \xib_3^4, \xib_1^{16} \xib_3^8, \xib_1^8 \xib_2^4 \xib_3^8 \\
&& \si{72} (A(2)\mmod A(1))_*: &&& \xib_2^{24}, \xib_2^{16} \xib_3^4 \\ 
&& \si{80} \bou_1^2: &&& \{ \xib_3^{8}, \xib_4^4 \}\cdot\{\xib_2^8, \xib_3^4 \} \\
&& \si{80} \bou_2[1] &&&  v_0^{-4}[c_6]\xib_1^8 \xib_2^{20} + \cdots ,  v_0^{-4}[c_6]\xib_1^{8} \xib_2^{12} \xib_3^4 + \cdots ,  \\
&& &&& \quad v_0^{-4}[c_6]\xib_1^8 \xib_2^4 \xib_3^8 + \cdots \\
\si{56} \bou_7:
&& \si{56} (A(2)\mmod A(1))_* \otimes  \tmfu_2: &&& 
\xib_1^{56}, \xib_1^{48} \xib_2^4, \xib_1^{40} \xib_2^8, \xib_1^{32} \xib_2^{12}, \xib_1^{40} \xib_3^4, \xib_1^{32} \xib_2^4 \xib_3^4,  \\
&& &&& \quad \xib_1^{24} \xib_2^{16}, \xib_1^{16} \xib_2^{20}, \xib_1^{24} \xib_2^8 \xib_3^4, \xib_1^{16} \xib_2^{12} \xib_3^4, \xib_1^{24} \xib_3^8, \xib_1^{16} \xib_2^4 \xib_3^8 \\
&&\si{80} (A(2)\mmod A(1))_* \otimes \bou_1: &&&  \{ \xib_2^{24}, \xib_2^{16} \xib_3^4 \} \cdot \{ \xib_1^8, \xib_2^4 \} \\ 
&& \si{88} \bou_1^3: &&& \{ \xib_3^{8}, \xib_4^4 \}\cdot\{\xib_2^8, \xib_3^4 \} \cdot \{ \xib_1^8, \xib_2^4 \} \\
\si{64} \bou_8:
&& \si{64} (A(2)\mmod A(1))_* \otimes \tmfu_3: &&& 
\xib_1^{64}, \xib_1^{56} \xib_2^4, \xib_1^{48} \xib_2^8, \xib_1^{40} \xib_2^{12}, \xib_1^{48} \xib_3^4, \xib_1^{40} \xib_2^4 \xib_3^4,  \\
&& &&& \quad \xib_1^{32} \xib_2^{16}, \xib_1^{24} \xib_2^{20}, \xib_1^{32} \xib_2^8 \xib_3^4, \xib_1^{24} \xib_2^{12} \xib_3^4, \xib_1^{32} \xib_3^8, \xib_1^{24} \xib_2^4 \xib_3^8, \\
&& &&& \quad \xib_1^{16} \xib_2^{24}, \xib_1^{8} \xib_2^{28}, \xib_1^{16} \xib_2^{16} \xib_3^4, \xib_1^8 \xib_2^{20} \xib_3^4, \xib_1^{16} \xib_2^{8} \xib_3^8, \xib_1^8 \xib_2^{12} \xib_3^8, \\
&& &&& \quad \xib_1^{16} \xib_3^{12}, \xib_1^8 \xib_2^4 \xib_3^{12} \\
&& \si{96} (A(2)\mmod A(1))_* \otimes \tmfu_1: &&& \xib_2^32, \xib_2^{24} \xib_3^4, \xib_2^{16} \xib_3^8, \xib_2^8 \xib_3^{12}, \xib_2^{16} \xib_4^4, \xib_2^{8} \xib_3^4 \xib_4^4 \\ 
&& \si{112} (A(2)\mmod A(1))_*: &&& \xib_3^{16}, \xib_3^8\xib_4^4 \\
&& \si{120}\bou_1: &&& \xib_4^8, \xib_5^4 \\ 
&& \si{128}\FF_2[1]: &&& v_0^{-4}[c_6] \xib_3^8\xib_4^4 + \cdots \\
&& \si{120} \bou_1[1]: &&& v_0^{-4}[c_6]\xib_2^8 \xib_3^{12} + \cdots , v_0^{-4}[c_6]\xib_2^{8} \xib_3^4 \xib_4^4 + \cdots\\ 
&& \si{104} \bou_3 [1]: &&&
v_0^{-4}[c_6]\xib_1^{8} \xib_2^{28}+\cdots, v_0^{-4}[c_6]\xib_1^8 \xib_2^{20} \xib_3^4+\cdots, \\
&& &&& \quad v_0^{-4}[c_6]\xib_1^8 \xib_2^{12} \xib_3^8 + \cdots
\end{alignat*}

\subsection{Identification of the integral lattice}

Having constructed useful bases of the summands
$$ v_0^{-1} \Ext_{A(2)_*}(\si{8j}\bou_j) \subset v_0^{-1}\Ext_{A(2)_*}(A\mmod A(2)_*) $$
it remains to understand the lattices
$$ \frac{\Ext_{A(2)_*}(\si{8j}\bou_j)}{v_0-tors} \subset v_0^{-1} \Ext_{A(2)_*}(\si{8j}\bou_j). $$
This can accomplished inductively; the rational generators we identified in the last section are compatible with the exact sequences (\ref{eq:boSES1}), (\ref{eq:boSES2}), and $\frac{\Ext_{A(2)_*}}{v_0-tors}$ of the terms in these exact sequences are determined by the $\frac{\Ext_{A(1)_*}}{v_0-tors}$-computations of Section~\ref{sec:bo}, and knowledge of 
$$ \frac{\Ext_{A(2)_*}(\bou_1^k)}{v_0-tors}. $$
Unfortunately the latter requires separate explicit computation for each $k$, and hence does not yield a general answer.

Nevertheless, in this section we will give some lemmas which provide convenient criteria for identifying the $i$ so that given a rational generator $x \in (A\mmod A(2))_*$  
(as in the previous section) we have
$$ v_0^i x \in \frac{\Ext_{A(2)*}((A\mmod A(2))_*)}{v_0-tors} \subset v_0^{-1} \Ext_{A(2)*}((A\mmod A(2))_*). $$

We first must clarify what we actually mean by ``rational generator''.   The generators identified in the last section originate from the exact sequences (\ref{eq:boSES1}), (\ref{eq:boSES12}). More precisely, they come from the generators of $v_0^{-1} \Ext_{A(2)_*}(M)$ where $M$ is given by 
\begin{align*}
\text{Case 1:} & \quad M = \bou_1^k, \\
\text{Case 2:} & \quad M = (A(2)\mmod A(1))_* \otimes \tmfu_j.
\end{align*}

In Case 1, a generator $x$ of $v_0^{-1}\Ext_{A(2)_*}(M)$ is a generator as a module over $\FF_2[v_0^{\pm 1}, [c_4], [c_6]]$, using the isomorphisms
\begin{equation}\label{eq:integralcase2}
\begin{split}
& v_0^{-1}\Ext_{A(2)_*}(\bou_1^k) \\
& \quad \cong v_0^{-1}\Ext_{A_*}((A\mmod A(2))_* \otimes \bou_1^k) \\
& \quad \xrightarrow[\cong]{\alpha} v_0^{-1}\Ext_{A(0)_*}((A\mmod A(2))_* \otimes \bou^k_1) \\
& \quad \cong v_0^{-1}\Ext_{A(0)_*}((A\mmod A(2))_*) \otimes_{\FF_2[v_0^{\pm 1}]}  v_0^{-1}\Ext_{A(0)_*}(\bou_1^k) \\
& \quad \cong \FF_2[v_0^{\pm 1}, [c_4], [c_6]] \otimes_{\FF_2}  
\FF_2\{ 1, \xib_1^4 \}^{\otimes k}.
\end{split}
\end{equation}
The rational generators in this case correspond to the generators 
$$ x \in \{ 1, \xib_1^4 \}^{\otimes k}. $$

In Case 2, a generator $x$ of $v_0^{-1}\Ext_{A(2)_*}(M)$ is a generator as a module over $\FF_2[v_0^{\pm 1}, [c_4]]$ using the isomorphisms
\begin{equation}\label{eq:integralcase1}
\begin{split}
& v_0^{-1}\Ext_{A(2)_*}((A(2)\mmod A(1))_* \otimes \tmfu_j)  \\
& \quad \cong v_0^{-1}\Ext_{A(1)_*}(\tmfu_j) \\
& \quad \cong v_0^{-1}\Ext_{A_*}((A\mmod A(1))_* \otimes \tmfu_j) \\
& \quad \xrightarrow[\cong]{\alpha} v_0^{-1}\Ext_{A(0)_*}((A\mmod A(1))_* \otimes \tmfu_j) \\
& \quad \cong v_0^{-1}\Ext_{A(0)_*}((A\mmod A(1))_*) \otimes_{\FF_2[v_0^{\pm 1}]}  v_0^{-1}\Ext_{A(0)_*}(\tmfu_j) \\
& \quad \cong \FF_2[v_0^{\pm 1}, [c_4]]\{ 1, \xib_1^4 \} \otimes_{\FF_2}  
\FF_2\{  \xib_1^{8i_1} \xib_2^{4i_2} \: : \: i_1 + i_2 \le j \}.
\end{split}
\end{equation}
The rational generators in this case correspond to the generators 
$$ x = \xib_1^{4\epsilon} \otimes \xib_1^{8i_1} \xib_2^{4i_2}. $$

In either case, the maps $\alpha$ in both (\ref{eq:integralcase1}) and (\ref{eq:integralcase2}) arise from surjections of cobar complexes
$$ C^*_{A_*}(N) \rightarrow C^*_{A(0)_*}(N) $$
induced by the surjection
$$ A_* \rightarrow A(0)_*. $$
Thus a term $v_0^i x \in C^*_{A(0)_*}(N)$ representing an element in $v_0^{-1} \Ext_{A(0)_*}(N)$  corresponds (for $i$ sufficiently large) to a term $[\xib_1]^ix + \cdots  \in C^*_{A_*}(N)$.  Then we have determined an element of the integral lattice
$$ \left\lbrack [\xib_1]^ix + \cdots  \right\rbrack \in \frac{\Ext_{A_*}(N)}{v_0-tors} \subset v_0^{-1}\Ext_{A_*}(N). $$

\begin{lem}\label{lem:integralCnu}
Suppose that the $A(2)_*$-coaction on $x \in (A\mmod A(2))_*$ satisfies 
$$ \psi(x) = 1\otimes x + \xib_1^4 \otimes y $$
with $y$ primitive, as in the following ``cell diagram'':
$$
\xymatrix@C-2em@R-2em{
x & \circ \ar@{-} `r[ddd] `[ddd]^{\sq^4} [ddd]
\\ \\ \\ 
y & \circ
}
$$
Then
$$ v_0^3 x \in \frac{\Ext_{A(2)*}((A\mmod A(2))_*)}{v_0-tors} \subset v_0^{-1} \Ext_{A(2)*}((A\mmod A(2))_*) $$
and is represented by
$$
[\xib_1|\xib_1 | \xib_1]x + \left( [\xib_1|\xib_2|\xib_2] + [\xib_1|\xib_1|\xib_1^2\xib_2] + 
[\xib_1|\xib_1\xib_2|\xib_1^2] + [\xib_2|\xib_1^2|\xib_1^2] \right)y
$$
in the cobar complex $C^*_{A(2)_*}((A\mmod A(2))_*)$.
\end{lem}

\begin{proof}
Since the cell complex depicted agrees with $A(2)\mmod A(1)$ through dimension $4$, $\Ext_{A(2)_*}$ of this comodule agrees with $\Ext_{A(1)_*}(\FF_2)$ through dimension $4$.  In particular, $v_0^3 x + \cdots $ generates an $\frac{\Ext_{A(2)_*}}{v_0-tors}$-term in this dimension.  To determine the exact representing cocycle, we note that 
$$ [\xib_1|\xib_2|\xib_2] + [\xib_1|\xib_1|\xib_1^2\xib_2] + 
[\xib_1|\xib_1\xib_2|\xib_1^2] + [\xib_2|\xib_1^2|\xib_1^2] $$
kills $h_0^3 h_2$ in $\Ext_{A(2)_*}(\FF_2)$.
\end{proof}

\begin{ex}
Let $\alpha = \xib^{8j_1}_{i_1} \xib_{i_2}^{8j_2} \cdots $ be a monomial with exponents all divisible by $8$. A typical instance of a set of generators of $(A\mmod A(2))_*$ satisfying the hypotheses of Lemma~\ref{lem:integralCnu} is 
$$
\xymatrix@C-2em@R-2em{
\xib_i^4 \alpha & \circ \ar@{-} `r[ddd] `[ddd]^{\sq^4} [ddd]
\\ \\ \\ 
\xib_{i-1}^8 \alpha & \circ
}
$$
\end{ex}

The following corollary will be essential to relating the integral generators of Lemma~\ref{lem:integralCnu} to $2$-variable modular forms in Section~\ref{sec:2var}.

\begin{cor}\label{cor:integralCnu}
Suppose that $x$ satisfies the hypotheses of Lemma~\ref{lem:integralCnu}.  The image of the corresponding integral generator
$$ v_0^3 x + \cdots \in \Ext_{A(2)_*}((A\mmod A(2)_*)) $$
in $\Ext_{E[Q_0, Q_1, Q_2]_*}((A\mmod E[Q_0, Q_1, Q_2])_*)$ is given by
$$ v_0^3 x + v_0 [a_1]^2 y. $$
\end{cor}

\begin{proof}
Note the equality 
$$ E[Q_0, Q_1, Q_2]_* = \FF_2[\xib_1, \xib_2, \xib_3]/(\xib_1^2, \xib_2^2, \xib_3^2). $$
Therefore
the image of the integral generator of Lemma~\ref{lem:integralCnu} under the map
$$ C^*_{A(2)_*}((A\mmod A(2))_*) \rightarrow  C^*_{E[Q_0, Q_1, Q_2]_*}((A\mmod E[Q_0, Q_1, Q_2])_*) $$
is 
$$ 
[\xib_1|\xib_1 | \xib_1]x + [\xib_1|\xib_2|\xib_2] y
$$
and this represents $v_0^3 x + v_0 [a_1]^2 y$.
\end{proof}

Similar arguments provide the following slight refinement.

\begin{lem}\label{lem:integralCnuHZ1}
Suppose that the $A(2)_*$-coaction on $x \in (A\mmod A(2))_*$ satisfies 
$$ \psi(x) = 1\otimes x + \xib_1^4 \otimes y $$
with $y$ primitive, and that there exists $w$ and $z$ satisfying
$$ \psi(z) = 1\otimes z + \xib_1^2 \otimes y $$
and 
$$ \psi(w) = 1\otimes w + \xib_1 \otimes z + \xib_2 \otimes y $$
as in the following ``cell diagram'':
$$
\xymatrix@C-2em@R-1em{
x & \circ \ar@{-} `r[dddd] `[dddd]^{\sq^4} [dddd] \\
w & \circ  \ar@{-}[d]_{\sq^1} \\ 
z & \circ \ar@/_1pc/@{-}[dd]_{\sq^2} \\ 
\\
y & \circ
}
$$
Then
$$ v_0 x \in \frac{\Ext_{A(2)*}((A\mmod A(2))_*)}{v_0-tors} \subset v_0^{-1} \Ext_{A(2)*}((A\mmod A(2))_*) $$
is represented by
$$
[\xib_1]x + [\xib_1^2]w + \left( [\xib_1^3]+[\xib_2]\right)z + [\xib_1^2 \xib_2]y
$$
in the cobar complex $C^*_{A(2)_*}((A\mmod A(2))_*)$.
\end{lem}

\begin{ex}
Let $\alpha = \xib^{8j_1}_{i_1} \xib_{i_2}^{8j_2} \cdots $ be a monomial with exponents all divisible by $8$.
A typical instance of a set of generators of $(A\mmod A(2))_*$ satisfying the hypotheses of Lemma~\ref{lem:integralCnuHZ1} is 
$$
\xymatrix@C-2em@R-1em{
\xib^{4}_{i} \xib^{4}_{i'} \alpha & \circ \ar@{-} `r[dddd] `[dddd]^{\sq^4} [dddd] \\
(\xib^{8}_{i-1} \xib_{i'+2} + \xib_{i+2} \xib^8_{i'-1}) \alpha & \circ  \ar@{-}[d]_{\sq^1} \\ 
(\xib^{8}_{i-1} \xib^2_{i'+1} + \xib^2_{i+1} \xib^8_{i'-1}) \alpha & \circ \ar@/_1pc/@{-}[dd]_{\sq^2} \\ 
\\
(\xib^{8}_{i-1} \xib^4_{i'} + \xib^4_i \xib^8_{i'-1}) \alpha & \circ
}
$$
\end{ex}

\begin{cor}\label{cor:integralCnuHZ1}
Suppose that $x$ satisfies the hypotheses of Lemma~\ref{lem:integralCnuHZ1}.  The image of the corresponding integral generator
$$ v_0 x + \cdots \in \Ext_{A(2)_*}((A\mmod A(2)_*)) $$
in $\Ext_{E[Q_0, Q_1, Q_2]_*}((A\mmod E[Q_0, Q_1, Q_2])_*)$ is given by
$$ v_0 x + [a_1] z. $$
\end{cor}

%% file: tmfcoop_2var.tex

\section{The image of $\tmf_*\tmf$ in $\TMF_*\TMF_\QQ$: two variable modular forms}\label{sec:2var}

\subsection{Review of Baker-Laures work on cooperations}

\emph{In this brief subsection, we do \emph{not} work $2$-locally, but integrally.}

For $N > 1$, the spectrum $\TMF_1(N)$ is even periodic, with 
$$ \TMF_1(N)_{2*} \cong M_*(\Gamma_1(N))[\Delta^{-1}]_{\ZZ[1/N]}. $$
In particular, its homotopy is torsion-free.  As a result, there is an embedding
\begin{align*}
\TMF_1(N)_{2*} \TMF_1(N) & \hookrightarrow \TMF_1(N)_{2*}\TMF_1(N)_\QQ \\
& \cong M_*(\Gamma_1(N))[\Delta^{-1}]_\QQ \otimes M_*(\Gamma_1(N))[\Delta^{-1}]_\QQ.
\end{align*}
Consider the multivariate $q$-expansion map
$$
 M_*(\Gamma_1(N))[\Delta^{-1}]_\QQ \otimes M_*(\Gamma_1(N))[\Delta^{-1}]_\QQ
\rightarrow \QQ((q, \bar{q})).
$$

In \cite[Thm.~2.10]{Laures}, Laures uses it to determine the image of $\TMF_1(N)_*\TMF_1(N)$ under the embedding above.

\begin{thm}[Laures]\label{thm:Laures}
The multivariate $q$-expansion map gives a pullback
$$
\xymatrix{
\TMF_1(N)_*\TMF_1(N) \ar[r] \ar[d] & 
\TMF_1(N)_*\TMF_1(N)_\QQ \ar[d] 
\\
\ZZ[1/N]((q, \bar{q})) \ar[r] & 
\QQ((q, \bar{q})).
}
$$
Therefore, elements of $\TMF_1(N)_*\TMF_1(N)$ are given by sums
$$ \sum_i f_i \otimes g_i \in  M_*(\Gamma_1(N))[\Delta^{-1}]_\QQ \otimes M_*(\Gamma_1(N))[\Delta^{-1}]_\QQ $$ 
with
$$ \sum_i f_i(q) \otimes g_i(\bar{q}) \in \ZZ[1/N]((q, \bar{q})).$$ 
\end{thm}
We shall let $M^{2-var}_*(\Gamma_1(N))[\Delta^{-1}, \bar{\Delta}^{-1}]$ denote this ring of integral $2$-variable modular forms (meromorphic at the cusps).  We shall denote the subring of those integral $2$-variable modular forms which have holomorphic multivariate $q$-expansions by $M^{2-var}_*(\Gamma_1(N))$.

\begin{rmk}
Baker \cite{MR1307488} showed that in the case of $N = 1$, with $6$ inverted, we have
$$ \TMF_*\TMF[1/6] \cong M^{2-var}_*(\Gamma(1))[1/6, \Delta^{-1}, \bar{\Delta}^{-1}]. $$
Laures's methods also apply to this case.
\end{rmk}

\subsection{Representing $\TMF_*\TMF/tors$ with $2$-variable modular forms}

\emph{From now on, everything is again implicitly $2$-local.}

We now turn to adapting Laures's perspective to identify $\TMF_*\TMF/tors$.  To do this, we use the descent spectral sequence for 
$$ \TMF  \rightarrow \TMF_1(3) . $$  
Let $(B_*, \Gamma_{B_*})$ denote the Hopf algebroid encoding descent from $\Msl{3}$ to $\M$, with
\begin{align*}
B_* & = \pi_* \TMF_1(3)  = \ZZ [a_1, a_3, \Delta^{-1}], \\
\Gamma_{B_*} & = \pi_* \TMF_1(3) \wedge_\TMF \TMF_1(3)  = B_*[r,s,t]/(\sim),
\end{align*}
(see Section~\ref{sec:review}) where $\sim$ denotes the relations (\ref{eq:relations}).
The Bousfield-Kan spectral sequence associated to the cosimplicial resolution
$$ \TMF  \rightarrow \TMF_1(3)  \Rightarrow \TMF_1(3)^{\wedge_{\TMF} 2}   \Rrightarrow \TMF_1(3)^{\wedge_{\TMF} 3}  \cdots $$
yields a Baker-Lazarev spectral sequence \cite{BakerLazarev}
$$ \Ext^{s,t}_{\Gamma_{B_*}}(B_*) \Rightarrow \pi_{t-s} \TMF . $$

We can use parallel methods to construct a Baker-Lazarev spectral sequence for the extension
$$ \TMF \wedge \TMF  \rightarrow \TMF_1(3) \wedge \TMF_1(3) . $$
Let $(B^{(2)}_*, \Gamma_{B^{(2)}_*})$ denote the associated Hopf algebroid encoding descent, with
\begin{align*}
B^{(2)}_* & = \pi_* \TMF_1(3) \wedge \TMF_1(3) ,  \\
\Gamma_{B^{(2)}_*} & = \pi_* (\TMF_1(3)^{\wedge_\TMF 2}  \wedge \TMF_1(3)^{\wedge_\TMF 2}) . 
\end{align*}
The Bousfield-Kan spectral sequence associated to the cosimplicial resolution
$$ \TMF^{\wedge 2}  \rightarrow \TMF_1(3)^{\wedge 2}  \Rightarrow \left( \TMF_1(3)^{\wedge_\TMF 2}\right) ^{\wedge 2}  \Rrightarrow \left( \TMF_1(3)^{\wedge_\TMF 3}\right)^{\wedge 2} \cdots $$
yields a descent spectral sequence
$$ \Ext^{s,t}_{\Gamma_{B^{(2)}_*}}(B^{(2)}_*) \Rightarrow \TMF_{t-s} \TMF . $$

\begin{lem}\label{lem:2varlem1}
The map induced from the edge homomorphism
$$ \TMF_*\TMF /tors \rightarrow \Ext^{0,*}_{\Gamma^{(2)}_{B_*}}(B_*^{(2)}) $$
is an injection.
\end{lem}

\begin{proof}
This follows from the fact that the map
$$ \TMF \wedge \TMF  \rightarrow \TMF \wedge \TMF_\QQ $$
induces a map of descent spectral sequences
$$
\xymatrix{
\Ext^{s,t}_{\Gamma^{(2)}_{B_*}}(B^{(2)}_*) \ar@{=>}[r] \ar[d]  
& \TMF_{t-s} \TMF  \ar[d] 
\\
\Ext^{s,t}_{\Gamma^{(2)}_{B_*}}(B^{(2)}_* \otimes \QQ) \ar@{=>}[r]  
& \TMF_{t-s} \TMF_{\QQ}
}
$$
and the rational spectral sequence is concentrated on the $s  = 0$ line.
\end{proof}

The significance of this homomorphism is that the target is the space of $2$-local two-variable modular forms for $\Gamma(1)$.

\begin{lem}\label{lem:2varlem2}
The $0$-line of the descent spectral sequence for $\TMF_*\TMF $ may be identified with the space of $2$-local two-variable modular forms of level $1$ (meromorphic at the cusp):
$$ \Ext^{0, 2*}_{\Gamma^{(2)}_{B_*}}(B_*^{(2)}) = 
M_*^{2-var}(\Gamma(1))[\Delta^{-1}, \bar{\Delta}^{-1}] . $$ 
\end{lem}

\begin{proof}
This follows from the composition of pullback squares
$$
\xymatrix{
\Ext^{0, *}_{\Gamma_{B^{(2)}_*}}(B_*^{(2)}) \ar@{^{(}->}[r] \ar@{^{(}->}[d] & 
\Ext^{0, *}_{\Gamma_{B^{(2)}_*}}(B_*^{(2)} \otimes \QQ) \ar@{^{(}->}[d]
\\
\TMF_1(3)_*\TMF_1(3)  \ar@{^{(}->}[r] \ar[d] & 
\TMF_1(3)_*\TMF_1(3)_\QQ \ar[d] 
\\
\ZZ ((q, \bar{q})) \ar[r] & 
\QQ((q, \bar{q})).
}
$$
The bottom square is a pullback by Theorem~\ref{thm:Laures}.  Note that since $\TMF_1(3) \wedge_\TMF \TMF_1(3)$ is Landweber exact, $\Gamma_{B_*^{(2)}}$ is torsion-free.  Thus an element of $B_*^{(2)}$ is $\Gamma_{B^{(2)}_*}$-primitive if and only if its image in $B_*^{(2)} \otimes \QQ$ is primitive.  This shows that the top square is a pullback.
\end{proof}

\subsection{Representing $\tmf_*\tmf /tors$ with $2$-variable modular forms}

Recall from Equation \ref{eq:afs} that the Adams filtration of $c_4$ is 4 and the Adams filtration of $c_6$ is 5.  Regarding $2$-variable modular forms as a subring
$$ M^{2-var}_*(\Gamma(1))  \subset \QQ[c_4, c_6, \bar{c}_4, \bar{c}_6], $$
we shall denote by $M^{2-var}_*(\Gamma(1))^{AF \ge 0} $ the subring of 2-variable modular forms with non-negative Adams filtration.
The results of the previous section now easily give the following result.

\begin{prop}
The composite induced by Lemmas~\ref{lem:2varlem1} and \ref{lem:2varlem2}
$$ \tmf_{2*}\tmf /tors \rightarrow \TMF_{2*}\TMF /tors \hookrightarrow M_*^{2-var}(\Gamma(1))[\Delta^{-1}, \bar{\Delta}^{-1}]  $$
induces an injection
$$ \tmf_{2*}\tmf /tors \hookrightarrow M_*^{2-var}(\Gamma(1))^{AF \ge 0}  $$
which is a rational isomorphism.
\end{prop}

\begin{proof}
Consider the commutative cube
$$
\xymatrix@C-3em@R-1em{
\tmf_{2*}\tmf /tors \ar[rr] \ar@{^{(}->}[dd] \ar@{.>}[dr] 
&& \TMF_{2*}\TMF /tors \ar[rd] \ar[dd]|{\phantom{MM}} 
\\
& M^{2-var}_*(\Gamma(1))  \ar[rr] \ar[dd] 
&& M_*^{2-var}(\Gamma(1))[\Delta^{-1}, \bar{\Delta}^{-1}]  \ar[dd] 
\\
\tmf_{2*}\tmf_{\QQ} \ar[rr]|{\phantom{MM}} \ar@{=}[dr] 
&& \TMF_{2*}\TMF_{\QQ} \ar@{=}[rd]  
\\
& M^{2-var}_*(\Gamma(1))_{\QQ} \ar[rr] 
&& M_*^{2-var}(\Gamma(1))[\Delta^{-1}, \bar{\Delta}^{-1}]_{\QQ} .
}
$$
(The dotted arrow exists because the front face of the cube is a pullback.)  The commutativity of the diagram, and the fact that rationally the top face is isomorphic to the bottom face, give an injection
$$ \tmf_{2*}\tmf /tors \hookrightarrow M_*^{2-var}(\Gamma(1))  $$
that is a rational isomorphism.
Since all of the elements of the source have Adams filtration $\ge 0$, this injection factors through the subring
$$ \tmf_{2*}\tmf /tors \hookrightarrow M_*^{2-var}(\Gamma(1))^{AF \ge 0} . $$
\end{proof}

\subsection{Detecting $2$-variable modular forms in the ASS}

\begin{defn}\label{def:detect}
Suppose that we are given a class
$$ x \in \Ext(\tmf \wedge \tmf) $$
and a $2$-variable modular form
$$ f \in M^{2-var}_*(\Gamma(1))^{AF \ge 0} . $$
We shall say that \emph{$x$ detects $f$}
if the image of $x$ in $v_0^{-1}\Ext(\tmf \wedge \tmf)$
detects the image of $f$ in $M^{2-var}_*(\Gamma(1)) \otimes \QQ_2$ in the localized ASS
$$ v_0^{-1}\Ext(\tmf \wedge \tmf) \Rightarrow \tmf_*\tmf \otimes \QQ_2 \cong M^{2-var}_*(\Gamma(1)) \otimes {\QQ_2}. $$
\end{defn}

\begin{rmk} Suppose $x$ as above is a permanent cycle in the unlocalized ASS
$$ \Ext(\tmf \wedge \tmf) \Rightarrow \tmf_*\tmf^\wedge_2, $$
and detects $\zeta \in \tmf_*\tmf^\wedge_2$. If $f$ is the image of $\zeta$ under the map
$$ \tmf_*\tmf^\wedge_2 \rightarrow [M^{2-var}_*(\Gamma(1))^\wedge_2]^{AF \ge 0}, $$
then $x$ detects $f$ in the sense of Definition \ref{def:detect}.
\end{rmk}

Given a class $x \in \Ext(\tmf \wedge \tmf)$, we wish to find a $2$-variable modular form it detects.  To accomplish this, we contemplate the following diagram
\begin{equation}\label{eq:masterdiag1}
\xymatrix@C-1em{
 \frac{\tmf_*\tmf^\wedge_2}{tors} \ar@{^{(}->}[r] \ar@{^{(}->}[d] \ar@/_5pc/[ddd] & 
\frac{\tmf_1(3)_*\tmf_1(3)^\wedge_2}{tors} \ar@{^{(}->}[d] \ar@/^5pc/[ddd] \\
 M^{2-var}_*(\Gamma(1))^\wedge_2 \ar@{^{(}->}[d] \ar@{^{(}->}[r] & 
M^{2-var}_*(\Gamma_1(3))^\wedge_2 \ar@{^{(}->}[d] \\
\QQ_2[c_4, c_6, \bar{c}_4, \bar{c}_6] \ar@{^{(}->}[r] \ar@{=}[d] &
\QQ_2[a_1, a_3, \bar{a}_1, \bar{a}_3] \ar@{=}[d] \\
\tmf_*\tmf \otimes \QQ_2 \ar@{^{(}->}[r] &
\tmf_1(3)_*\tmf_1(3) \otimes \QQ_2 
}
\end{equation}
and the associated ``Ext version'':
\begin{equation}\label{eq:masterdiag2}
\xymatrix@C-1em{
 \frac{\Ext(\tmf \wedge \tmf)}{v_0-tors} \ar@{^{(}->}[r] \ar@{^{(}->}[d] \ar@/_6pc/[ddd] & 
\frac{\Ext(\tmf_1(3)\wedge \tmf_1(3))}{tors} \ar@{^{(}->}[d] \ar@/^6pc/[ddd]^{(2)} \\
 E_0^{AF}M^{2-var}_*(\Gamma(1))^\wedge_2 \ar@{^{(}->}[d] \ar@{^{(}->}[r] & 
E_0^{AF} M^{2-var}_*(\Gamma_1(3))^\wedge_2 \ar@{^{(}->}[d] \\
\FF_2[v_0^{\pm}][[c_4], [c_6], [\bar{c}_4], [\bar{c}_6]] \ar@{^{(}->}[r]^{(1)} \ar@{=}[d] &
\FF_2[v_0^{\pm}][[a_1], [a_3], [\bar{a}_1], [\bar{a}_3]] \ar@{=}[d] \\
v_0^{-1}\Ext(\tmf \wedge \tmf)  \ar@{^{(}->}[r] &
v_0^{-1}\Ext(\tmf_1(3) \wedge \tmf_1(3)) 
}
\end{equation}
Here, $E_0^{AF}M^{2-var}_*$ denotes the associated graded with respect to Adams filtration (AF), where, as usual (see Section~\ref{subsec:AF}), we set
$$ AF(2) = AF(a_1) = AF(a_3) = 1, \quad AF(c_4) = 4, \quad AF(c_6) = 5. $$
As indicated, in both of the above diagrams, all of the arrows are injections.  To determine whether a class $x \in \Ext(\tmf \wedge \tmf)$ detects $f \in M^{2-var}_*(\Gamma(1))$, it suffices to determine whether the image of $x$ in $\Ext(\tmf_1(3) \wedge \tmf_1(3))$ detects the image of $f$ in $M^{2-var}_*(\Gamma_1(3))$.

The following lemma follows immediately from (\ref{eq:cia1a3}).

\begin{lem}\label{lem:cia1a3}
The map $(1)$ of Diagram~(\ref{eq:masterdiag2}) is given by
\begin{align*}
[c_4] & \mapsto [a_1]^4, \\
[c_6] & \mapsto v_0^3[a_3]^2.
\end{align*}
\end{lem}

Given a $2$-variable modular form $f \in M^{2-var}_*(\Gamma(1))$, let
$f(a_i, \bar{a}_i)$ denote its image in
$$ M^{2-var}_*(\Gamma_1(3))\otimes \QQ_2 \cong \QQ_2[a_1, a_3, \bar{a}_1, \bar{a}_3] \cong \tmf_1(3)_*\tmf_1(3) \otimes \QQ_2. $$
and let
$$ [f(a_i, \bar{a}_i)] \in v_0^{-1} \Ext(\tmf_1(3) \wedge \tmf_1(3)) \cong \FF_2[v_0^{\pm 1}, [a_1],[a_3], [\bar{a}_1], [\bar{a}_3]] $$
denote the element which detects it in the (collapsing) $v_0$-localized ASS.

Similarly, let $t_k(a_i, \bar{a}_i)$ denote the images of $t_k$ in $\tmf_1(3)_*\tmf_1(3) \otimes \QQ_2$  (as in Section~\ref{sec:tmf13BP2}), and let $[t_k(a_i, \bar{a}_i)]$ denote the elements of $\Ext$ which detect these images in the $v_0$-localized ASS for $\tmf_1(3)_*\tmf_1(3) \otimes \QQ_2$.

The map $(2)$ of Diagram~(\ref{eq:masterdiag2}) is essentially determined by the following lemma.

\begin{lem}\label{lem:xibiti}
The subalgebra 
$$ \FF_2[\xib_k^2 \: : \: k \ge 1] \subset (A\mmod E[Q_0, Q_1, Q_2])_* = \FF_2[\xib_1^{2}, \xib_2^2, \xib_3^2, \xib_4 \cdots ] $$
is contained in 
$$ \Ext^0_{E[Q_0, Q_1, Q_2]_*}((A\mmod E[Q_0, Q_1, Q_2])_*) = \Ext(\tmf_1(3) \wedge \tmf_1(3)). $$
Furthermore, map (2) of Diagram~(\ref{eq:masterdiag2}) is determined by
$$ \xib_k^2 \mapsto [t_k(a_i, \bar{a}_i)]. $$
\end{lem}

\begin{proof}
The elements $\xib_k^2$ are easily checked to be primitive with respect to the $E[Q_0, Q_1, Q_2]_*$-coaction.  The second part follows from the fact that in the diagram
$$
\xymatrix{
\BP_*\BP \ar[rr] \ar[dr] && H_*H \\
& H_* \tmf_1(3) \ar@{^{(}->}[ur]
}
$$
$t_k$ is mapped to $\xib_k^2$ by the top horizontal map.
\end{proof}

\begin{rmk}
In fact, Lemma~\ref{lem:xibiti} completely determines map (2).  This is because $\Ext(\tmf_1(3) \wedge \tmf_1(3))/v_0-tors$ is generated as an $\FF_2[v_0, a_1, a_3]$-algebra by the elements $\xib_k^2$ (see \cite{Culver}).
\end{rmk}
 
We assemble these observations to give the following convenient criterion for determining when a particular element $z \in \Ext(\tmf \wedge \tmf)$ detects a $2$-variable modular form $f$.  

\begin{prop}\label{prop:detection}
Suppose that we are given an element $z \in \Ext(\tmf \wedge \tmf)$
whose image in 
$$ \Ext(\tmf_1(3) \wedge \tmf_1(3)) = \Ext_{E[Q_0, Q_1, Q_2]_*}((A\mmod E[Q_0, Q_1, Q_2])_*) $$
is given by
$$ \bar{z} = \sum_j \bar{z}_j \xib_1^{2k_{1,j}}\xib_2^{2k_{2,j}} \cdots $$
with $\bar{z}_j \in \Ext(\tmf_1(3))$.
The element $z$ detects a $2$-variable modular form
$$ f \in M^{2-var}_*(\Gamma(1)) ^{AF \ge 0} $$
if and only if 
$$ [f(a_i, \bar{a}_i)] = \sum_j \bar{z}_j [t_1(a_i, \bar{a}_i)]^{k_{1,j}} [t_2(a_i, \bar{a}_i)]^{k_{2,j}} \cdots. $$
\end{prop}

\subsection{Low dimensional computations of $2$-variable modular forms}

Below is a table of generators of $\Ext(\tmf \wedge \tmf)/tors$, as a module over $\FF_2[v_0, [c_4]]$, through dimension 64, with $2$-variable modular forms they detect.  The columns of this table are:
\begin{description}
\item[dim] dimension of the generator,
\item[$\text{\bf bo}_k$] indicates generator lies in the summand $\Ext_{A(2)_*}(\bou_k)$ (see the charts in Section~\ref{sec:ass}),
\item[AF] the Adams filtration of the generator,
\item[cell] the name of the image of the generator in $v_0^{-1}\Ext_{A(2)_*}(\bou_k)$, in the sense of Section~\ref{sec:rationalgens},
\item[form] a two-variable modular form which is detected by the generator in the $v_0$-localized ASS (where $f_k$ are defined below).  
\end{description}
The table below also gives a basis of $M^{2-var}_*(\Gamma(1)) $ as a $\ZZ [c_4]$-module: in dimension $2k$, a form $\alpha g$ in the last column, with $\alpha \in \QQ$ and $g$ a monomial in $\ZZ[c_4, c_6, \Delta, f_k]$ not divisible by 2, corresponds to a generator $g$ of $M^{2-var}_k(\Gamma(1)) $.\footnote{There is one exception: there is a $2$-variable modular form $\td{c_4f_{10}}$ which agrees with $c_4f_{10}$ modulo terms of higher Adams filtration, but which is $2$-divisible.  See Example~\ref{ex:echelon2}.}

\begin{table}[h]
\caption{Table of generators of $\Ext(\tmf \wedge \tmf)/tors$.}\label{tab:gens}
\end{table}
\begin{alignat*}{5}
&	\mr{dim} \quad	&&	\bo_k	\quad &&	\mr{AF}\quad 	&&	\mr{cell}		& \quad &	\mr{form}		\\
&	8	&&	1	&&	0	&&	\xib_1^{8	}	&&	 f_{1	}	\\
&	12	&&	1	&&	3	&&	[8 ]\xib_2^{4	}	&&	2 f_{2	}	\\
&	16	&&	2	&&	0	&&	\xib_1^{16	}	&&	 f_{1}^{2	}	\\
&	20	&&	1	&&	3	&&	[c_6/4 ]\cdot \xib_1^{8	}	&&	2 f_{3	}	\\
&	20	&&	2	&&	3	&&	[8 ]\xib_1^{8} \xib_2^{4	}	&&	2 f_{1}  f_{2	}	\\
&	24	&&	1	&&	4	&&	[c_6/2 ]\cdot  \xib_2^{4	}	&&	 f_{4	}	\\
&	24	&&	2	&&	0	&&	\xib_2^{8	}	&&	 f_{5	}	\\
&	24	&&	3	&&	0	&&	\xib_1^{24	}	&&	 f_{1}^{3	}	\\
&	28	&&	2	&&	3	&&	[8 ]\xib_3^{4	}	&&	2 f_{6	}	\\
&	28	&&	3	&&	3	&&	[8 ]\xib_1^{16} \xib_2^{4	}	&&	2 f_{1}^{2}  f_{2	}	\\
&	32	&&	1	&&	4	&&	[  \Delta  ]\xib_1^{8  	}	&&	   \Delta  f_{1	}	\\
&	32	&&	2	&&	1	&&	[c_6/16 ]\cdot \xib_1^{8} \xib_2^{4}+ [c_4/8 ]\cdot  \xib_2^{8	}	&&	 f_{9	}	\\
&	32	&&	3	&&	0	&&	\xib_1^{8} \xib_2^{8	}	&&	 f_{1}  f_{5	}	\\
&	32	&&	4	&&	0	&&	\xib_1^{32	}	&&	 f_{1}^{4	}	\\
&	36	&&	1	&&	7	&&	[8  \Delta  ]\xib_2^{4  	}	&&	2   \Delta  f_{2	}	\\
&	36	&&	2	&&	3	&&	[c_6/4 ]\cdot \xib_2^{8	}	&&	2 f_{7	}	\\
&	36	&&	3	&&	3	&&	[8 ]\xib_2^{12	}	&&	2 f_{2}  f_{5	}	\\
&	36	&&	3	&&	0	&&	\xib_1^{8} \xib_3^{4}+ \xib_2^{12	}	&&	 f_{10	}	\\
&	36	&&	4	&&	3	&&	[8  ]\xib_1^{24} \xib_2^{4	}	&&	2 f_{1}^{3}  f_{2	}	\\
&	40	&&	2	&&	4	&&	[c_6/2 ]\cdot  \xib_3^{4  	}	&&	 f_{8	}	\\
&	40	&&	3	&&	1	&&	[2 ]\xib_2^{4} \xib_3^{4	}	&&	 f_{11	}	\\
&	40	&&	4	&&	0	&&	\xib_1^{16} \xib_2^{8	}	&&	 f_{1}^{2}  f_{5	}	\\
&	40	&&	5	&&	0	&&	\xib_1^{20	}	&&	 f_{1}^{5	}	\\
&	44	&&	1	&&	7	&&	[   \Delta c_6/4 ]\cdot  \xib_1^{8  	}	&&	2   \Delta  f_{3	}	\\
&	44	&&	2	&&	7	&&	[c_6/4] ( [c_6/16 ]\cdot  \xib_1^{8} \xib_2^{4}+ [c_4/8]\cdot  \xib_2^{8	})	&&	c_{6}  f_{9}/4		\\
&	44	&&	3	&&	3	&&	[c_6/4]\cdot  \xib_1^{8} \xib_2^{8	}	&&	2 f_{1}  f_{7	}	\\
&	44	&&	4	&&	3	&&	[8 ]\xib_1^{8} \xib_2^{12	}	&&	2 f_{1}  f_{2}  f_{5	}	\\
&	44	&&	4	&&	0	&&	\xib_1^{16} \xib_3^{4}+ \xib_1^{8} \xib_2^{12	}	&&	2 f_{13	}	\\
&	44	&&	5	&&	3	&&	[8 ]\xib_1^{32} \xib_2^{4	}	&&	2 f_{1}^{4}  f_{2	}	\\
&	48	&&	1	&&	8	&&	[  \Delta c_6/2 ]\cdot \xib_2^{4 	}	&&	   \Delta  f_{4	}	\\
&	48	&&	2	&&	4	&&	[  \Delta  ]\xib_2^{8  	}	&&	   \Delta  f_{5	}	\\
&	48	&&	3	&&	4	&&	[c_6/2 ]\cdot  \xib_2^{12  	}	&&	 f_{2}  f_{7	}	\\
&	48	&&	3	&&	1	&&	[c_6/16 ]\cdot ( \xib_1^{8} \xib_3^{4}+ \xib_2^{12	})	&&	 f_{14	}	\\
&	48	&&	4	&&	0	&&	\xib_2^{16	}	&&	 f_{5}^{2	}	\\
&	48	&&	4	&&	1	&&	[2 ]\xib_1^{8} \xib_2^{4} \xib_3^{4	}	&&	 f_{1}  f_{11	}	\\
&	48	&&	5	&&	0	&&	\xib_1^{24} \xib_2^{8	}	&&	 f_{1}^{3}  f_{5	}	\\
&	48	&&	6	&&	0	&&	\xib_1^{48	}	&&	 f_{1}^{6	}	\\
&	52	&&	2	&&	7	&&	[8  \Delta  ]\xib_3^{4  	}	&&	2   \Delta  f_{6	}	\\
&	52	&&	3	&&	4	&&	[c_6/2]\cdot  \xib_2^{4} \xib_3^{4  	}	&&	2 f_{15	}	\\
&	52	&&	4	&&	3	&&	[8 ]\xib_2^{8} \xib_3^{4	}	&&	2 f_{5}  f_{6	}	\\
&	52	&&	5	&&	3	&&	[8 ]\xib_1^{16} \xib_2^{12	}	&&	2 f_{1}^{2}  f_{2}  f_{5	}	\\
&	52	&&	5	&&	0	&&	\xib_1^{24} \xib_3^{4}+ \xib_1^{16} \xib_2^{12	}	&&	2 f_{1}  f_{13	}	\\
&	52	&&	6	&&	3	&&	[8 ]\xib_1^{40} \xib_2^{4	}	&&	2 f_{1}^{5}  f_{2	}	\\
&	56	&&	1	&&	8	&&	[ {\Delta^2} \xib_1^{8	}	&&	   \Delta^{2}  f_{1	}	\\
&	56	&&	2	&&	8	&&	[ \Delta]([c_6/2 ]\cdot \xib_1^{8} \xib_2^{4}+ [c_4 ]\cdot  \xib_2^{8	})	&&	8    \Delta  f_{9	}	\\
&	56	&&	3	&&	4	&&	[  \Delta  ]\xib_1^{8} \xib_2^{8	}	&&	   \Delta  f_{5}  f_{1	}	\\
&	56	&&	4	&&	1	&&	[c_6/16 ]\cdot  \xib_1^{8} \xib_2^{12}+ [c_4/8 ]\cdot  \xib_2^{16	}	&&	 f_{5}  f_{9	}	\\
&	56	&&	4	&&	0	&&	\xib_3^{8	}	&&	 f_{16	}	\\
&	56	&&	5	&&	0	&&	\xib_1^{8} \xib_2^{16	}	&&	 f_{1}  f_{5}^{2	}	\\
&	56	&&	5	&&	1	&&	[2 ]\xib_1^{16} \xib_2^{4} \xib_3^{4	}	&&	 f_{1}^{2}  f_{11	}	\\
&	56	&&	6	&&	0	&&	\xib_1^{32} \xib_2^{8	}	&&	 f_{1}^{4}  f_{5	}	\\
&	60	&&	1	&&	11	&&	[8 \Delta^2 ]\cdot  \xib_2^{4	}	&&	2   \Delta^{2}  f_{2	}	\\
&	60	&&	2	&&	7	&&	[  \Delta c_6/4 ]\cdot  \xib_2^{8	}	&&	2   \Delta  f_{7	}	\\
&	60	&&	3	&&	7	&&	[8  \Delta  ]\xib_2^{12	}	&&	2   \Delta  f_{5}  f_{2	}	\\
&	60	&&	3	&&	4	&&	[ \Delta]( \xib_1^{8} \xib_3^{4}+ \xib_2^{12	})	&&	   \Delta  f_{10	}	\\
&	60	&&	4	&&	4	&&	[c_6/2 ]\cdot  \xib_1^{8} \xib_2^{4} \xib_3^{4}+ [c_4 ]\cdot  \xib_2^{8} \xib_3^{4	}	&&	2 f_{6}  f_{9	}	\\
&	60	&&	4	&&	3	&&	[8 ]\xib_4^{4	}	&&	2  f_{17	}	\\
&	60	&&	5	&&	0	&&	\xib_2^{20}+ \xib_1^{8} \xib_2^{8} \xib_3^{4	}	&&	 f_{18	}	\\
&	60	&&	5	&&	3	&&	[8 ]\xib_1^{8} \xib_2^{8} \xib_3^{4	}	&&	2 f_{1}  f_{5}  f_{6	}	\\
&	60	&&	6	&&	3	&&	[8 ]\xib_1^{24} \xib_2^{12	}	&&	2 f_{1}^{3}  f_{2}  f_{5	}	\\
&	60	&&	6	&&	0	&&	\xib_1^{32} \xib_3^{4	}	&&	2 f_{1}^{2}  f_{13	}	\\
&	60	&&	7	&&	3	&&	[8 ]\xib_1^{48} \xib_2^{4	}	&&	2 f_{1}^{6}  f_{2	}	\\
&	64	&&	2	&&	8	&&	[  \Delta c_6/2 ]\cdot  \xib_3^{4	}	&&	   \Delta  f_{8	}	\\
&	64	&&	3	&&	5	&&	[2  \Delta  ]\xib_2^{4} \xib_3^{4  	}	&&	   \Delta  f_{11	}	\\
&	64	&&	4	&&	2	&&	[c_6/16 ]\cdot  \xib_2^{8} \xib_3^{4}+ [c_4/8 ]\cdot  \xib_3^{8	}	&&	 f_{9}^{2	}/2	\\
&	64	&&	5	&&	1	&&	[2 ]\xib_2^{12} ]\xib_3^{4}		&&	 f_{1}  f_{5}  f_{9	}	\\
&	64	&&	5	&&	0	&&	\xib_1^{8} \xib_3^{8	}	&&	 f_{1}  f_{16	}	\\
&	64	&&	6	&&	0	&&	\xib_1^{16} \xib_2^{16	}	&&	 f_{5}^{2}  f_{1}^{2	}	\\
&	64	&&	6	&&	1	&&	[2 ]\xib_1^{24} \xib_2^{4} \xib_3^{4	}	&&	 f_{11}  f_{1}^{3	}	\\
&	64	&&	7	&&	0	&&	\xib_1^{40} \xib_2^{8	}	&&	 f_{1}^{5}  f_{5	}	\\
&	64	&&	8	&&	0	&&	\xib_1^{64	}	&&	 f_{1}^{8	}	
\end{alignat*}

The $2$-variable modular forms $f_k \in M^{2-var}_*(\Gamma(1)) $ in the above table are the generators of $M^{2-var}_*(\Gamma(1)) $ as an $M_*(\Gamma(1)) $-algebra in this range, and are defined as follows.

\begin{align*}
f_{1 } & := ( -\bc_{4}+ c_{4})/ 16 \\
f_{2 } & := ( -\bc_{6}+ c_{6})/ 8 \\
f_{3 } & := ({ 5}  f_{1}  c_{6}+{ 21}  f_{2}  c_{4})/ 8 \\
f_{4 } & := ({ 5}  f_{2}  c_{6}+{ 21}  f_{1}  c_{4}^{2})/ 8 \\
f_{5 } & := ( -f_{1}^{2}  c_{4}+ f_{2}^{2})/ 16 \\
f_{6 } & := ( -c_{4}^{2}  c_{6}+ c_{4}^{2}  c_{6}+{ 544}  f_{2}  c_{4}^{2}+{ 768}  f_{3}  c_{4}+{ 1792}  f_{1}  f_{2}  c_{4})/ 2048 \\
f_{7 } & := ({ 4}  f_{2}  \Delta+ f_{5}  c_{6}+{ 5}  f_{2}  c_{4}^{3}+{ 6}  f_{3}  c_{4}^{2}+{ 5}  f_{1}  f_{2}  c_{4}^{2}+{ 7}  f_{6}  c_{4}+{ 4}  f_{1}^{2}  f_{2}  c_{4})/ 8 \\
f_{8 } & := ({ 4}  f_{1}  c_{4}  \Delta+ f_{6}  c_{6}+{ 5}  f_{1}  c_{4}^{4}+{ 5}  f_{1}^{2}  c_{4}^{3}+{ 7}  f_{5}  c_{4}^{2}+{ 2}  f_{4}  c_{4}^{2}+{ 4}  f_{1}^{3}  c_{4}^{2})/ 8 \\
f_{9 } & := ({ 32}  f_{1}  \Delta+ f_{1}  f_{2}  c_{6}+{ 33}  f_{1}^{2}  c_{4}^{2}+{ 8}  f_{5}  c_{4}+{ 32}  f_{4}  c_{4}+{ 32}  f_{1}^{3}  c_{4})/ 64 \\
f_{10 } & := ({ 2}  f_{2}  c_{4}^{3}+ f_{1}  f_{2}  c_{4}^{2}+{ 2}  f_{6}  c_{4}+{ 3}  f_{1}^{2}  f_{2}  c_{4}+ f_{1}  f_{6}+ f_{2}  f_{5})/ 4 \\
f_{11 } & := ({ 4}  f_{1}  c_{4}  \Delta+{ 11}  f_{1}^{2}  c_{4}^{3}+{ 34}  f_{5}  c_{4}^{2}+{ 28}  f_{4}  c_{4}^{2}+{ 23}  f_{1}^{3}  c_{4}^{2}+{ 4}  f_{9}  c_{4}+ f_{1}  f_{5}  c_{4}+{ 4}  f_{1}^{4}  c_{4} \\
& \qquad +{ 4}  f_{8}+ f_{2}  f_{6})/ 8 \\
f_{12 } & := ( f_{1}  f_{5}  c_{6}+{ 8}  f_{2}  c_{4}^{4}+{ 8}  f_{3}  c_{4}^{3}+{ 8}  f_{1}  f_{2}  c_{4}^{3}+{ 8}  f_{6}  c_{4}^{2}+{ 8}  f_{1}^{2}  f_{2}  c_{4}^{2}+ f_{2}  f_{5}  c_{4})/ 8 \\
f_{13 } & := ({ 8}  f_{3}  \Delta+{ 80}  f_{2}  c_{4}^{4}+{ 56}  f_{3}  c_{4}^{3}+{ 80}  f_{1}  f_{2}  c_{4}^{3}+{ 76}  f_{6}  c_{4}^{2}+{ 55}  f_{1}^{2}  f_{2}  c_{4}^{2 }+{ 4}  f_{10}  c_{4}\\
&\qquad +{ 18}  f_{2}  f_{5}  c_{4}+{ 11}  f_{1}^{3}  f_{2}  c_{4}+{ 4}  f_{12}+ f_{1}^{2}  f_{6}+ f_{1}  f_{2}  f_{5}+{ 4}  f_{1}^{4}  f_{2})/ 8 \\
f_{14 } & := ({ 21}  f_{1}  c_{4}^{2}  \Delta+{ 8}  f_{5}  \Delta+{ 16}  f_{4}  \Delta+{ 20}  f_{1}^{3}  \Delta+ f_{10}  c_{6}+{ 11}  f_{1}  c_{4}^{5}+{ 36}  f_{1}^{2}  c_{4}^{4}+{ 591}  f_{5}  c_{4}^{3} \\
& \qquad +{ 490}  f_{4}  c_{4}^{3}+{ 437}  f_{1}^{3}  c_{4}^{3}+{ 119}  f_{9}  c_{4}^{2}+{ 140}  f_{1}  f_{5}  c_{4}^{2}+{ 75}  f_{1}^{4}  c_{4}^{2}+{ 10}  f_{11}  c_{4}+{ 11}  f_{8}  c_{4} \\
& \qquad +{ 32}  f_{1}^{5}  c_{4}+{ 8}  f_{1}  f_{2}  f_{6})/ 16 \\
f_{15 } & := ({ 4}  f_{6}  \Delta+ f_{1}^{2}  f_{2}  \Delta+{ 76}  f_{2}  c_{4}^{5}+{ 54}  f_{3}  c_{4}^{4}+{ 90}  f_{1}  f_{2}  c_{4}^{4}+{ 73}  f_{6}  c_{4}^{3}+{ 50}  f_{1}^{2}  f_{2}  c_{4}^{3}+{ 3}  f_{10}  c_{4}^{2} \\
& \qquad +{ 8}  f_{7}  c_{4}^{2}+{ 20}  f_{2}  f_{5}  c_{4}^{2}+{ 8}  f_{1}^{3}  f_{2}  c_{4}^{2}+{ 7}  f_{12}  c_{4}+{ 4}  f_{1}  f_{2}  f_{5}  c_{4})/ 8 \\
f_{16 } & := ({ 2}  f_{1}  \Delta^{2}+{ 24}  f_{1}  c_{4}^{3}  \Delta+{ 9}  f_{5}  c_{4}  \Delta+{ 18}  f_{4}  c_{4}  \Delta+{ 4}  f_{1}^{3}  c_{4}  \Delta+{ 2}  f_{9}  \Delta+ f_{1}  f_{5}  \Delta \\
& \qquad +{ 36}  f_{1}^{2}  c_{4}^{5}+{ 480}  f_{5}  c_{4}^{4}+{ 402}  f_{4}  c_{4}^{4}+{ 359}  f_{1}^{3}  c_{4}^{4}+{ 94}  f_{9}  c_{4}^{3}+{ 112}  f_{1}  f_{5}  c_{4}^{3}+{ 55}  f_{1}^{4}  c_{4}^{3} \\
& \qquad +{ 12}  f_{11}  c_{4}^{2}+{ 14}  f_{8}  c_{4}^{2}+{ 20}  f_{1}^{5}  c_{4}^{2}+{ 2}  f_{14}  c_{4}+{ 5}  f_{2}  f_{7}  c_{4}+ f_{5}^{2}  c_{4}+{ 4}  f_{1}^{3}  f_{5}  c_{4}+ f_{1}  f_{14} \\
& \qquad + f_{5}  f_{9}+ f_{1}  f_{2}  f_{7})/ 2 \\
f_{17 } & := ({ 2}  f_{2}  \Delta^{2}+{ 22}  f_{3}  c_{4}^{2}  \Delta+{ 11}  f_{6}  c_{4}  \Delta+ f_{2}  f_{5}  \Delta+{ 19}  f_{9}  c_{4}^{2}  c_{6}+{ 682}  f_{2}  c_{4}^{6}+{ 480}  f_{3}  c_{4}^{5} \\
& \qquad +{ 768}  f_{1}  f_{2}  c_{4}^{5}+{ 648}  f_{6}  c_{4}^{4}+{ 462}  f_{1}^{2}  f_{2}  c_{4}^{4}+{ 30}  f_{10}  c_{4}^{3}+{ 63}  f_{7}  c_{4}^{3}+{ 185}  f_{2}  f_{5}  c_{4}^{3} \\
& \qquad +{ 84}  f_{1}^{3}  f_{2}  c_{4}^{3}+{ 12}  f_{13}  c_{4}^{2}+{ 27}  f_{12}  c_{4}^{2}+{ 29}  f_{1}  f_{2}  f_{5}  c_{4}^{2}+{ 16}  f_{1}^{4}  f_{2}  c_{4}^{2}+{ 4}  f_{15}  c_{4}+{ 4}  f_{5}  f_{6}  c_{4} \\
& \qquad +{ 2}  f_{1}^{2}  f_{2}  f_{5}  c_{4}+ f_{2}  f_{14}+ f_{6}  f_{9})/ 2 \\
f_{18 } & := ({4}  f_{2}  \Delta^{2}+{168}  f_{3}  c_{4}^{2}  \Delta+{96}  f_{6}  c_{4}  \Delta+{8}  f_{2}  f_{5}  \Delta+{168}  f_{9}  c_{4}^{2}  c_{6}+{5880}  f_{2}  c_{4}^{6} \\
& \qquad +{4140}  f_{3}  c_{4}^{5}+{6648}  f_{1}  f_{2}  c_{4}^{5}+{5592}  f_{6}  c_{4}^{4}+{3980}  f_{1}^{2}  f_{2}  c_{4}^{4}+{248}  f_{10}  c_{4}^{3}+{560}  f_{7}  c_{4}^{3} \\
& \qquad +{1586}  f_{2}  f_{5}  c_{4}^{3}+{744}  f_{1}^{3}  f_{2}  c_{4}^{3}+{112}  f_{13}  c_{4}^{2}+{220}  f_{12}  c_{4}^{2}+{265}  f_{1}  f_{2}  f_{5}  c_{4}^{2} \\
& \qquad +{136}  f_{1}^{4}  f_{2}  c_{4}^{2}+{40}  f_{15}  c_{4}+{4}  f_{1}  f_{13}  c_{4}+{34}  f_{5}  f_{6}  c_{4}+{19}  f_{1}^{2}  f_{2}  f_{5}  c_{4}+{8}  f_{1}^{5}  f_{2}  c_{4} \\
& \qquad +{4}  f_{6}  f_{9}+ f_{1}  f_{5}  f_{6}+ f_{2}  f_{5}^{2})/4.
\end{align*}

We shall now indicate the methods used to generate Table~\ref{tab:gens}, and make some remarks about its contents.

The short exact sequences (\ref{eq:boSES1}),(\ref{eq:boSES2}) were used in  Section~\ref{sec:boSES} to give an inductive scheme for computing $\Ext_{A(2)_*}(\bou_k)$, and the charts in that section display the computation through dimension $64$.   In Section~\ref{sec:rationalgens}, these short exact sequences are used to give an inductive scheme for identifying the generators of $v_0^{-1}\Ext_{A(2)_*}(\bou_k)$, and appropriate multiples of these generators generate the image of $\Ext_{A(2)_*}(\bou_k)/tors$ in these localized Ext groups.  These generators are listed in the fourth column of Table~\ref{tab:gens}.  

The two variable modular forms in the last column of Table~\ref{tab:gens} are detected by the generators in the fourth column, in the sense of the previous section.  
In each instance, if necessary, we use Corollary~\ref{cor:integralCnu} or  \ref{cor:integralCnuHZ1} to find the image of the generator in $\Ext(\tmf_1(3) \wedge \tmf_1(3))$ and then apply Proposition~\ref{prop:detection}.
 
The 2-variable modular forms were generated by the following inductive method.  Let $\{z_i\}$ be a basis of 
$$ \frac{\Ext^{*,*+{2n}}(\tmf \wedge \tmf)}{v_0-tors} $$
as an $\FF_2[v_0]$-module.  We wish to produce a basis
$$ \{ f_{z_i} \} \subset MF^{2-var}_n(\Gamma(1)) $$
such that for appropriate $n(z_i) \ge 0$, the $z_i$ detects $2^{n(z_i)}f_{z_i}$.  
Suppose inductively that we have found such $2$-variable modular forms
$f_{z_i}$
for all $z_i$ with Adams filtration (AF) greater than $s$, and let $z \in \{z_i\}$ be a basis element with $AF(z) = s$.
We wish to produce a $2$-variable modular form $f_z$ such that $z$ detects $2^{n(z)}f_z$, and so that
$$ f_z \not\in \FF_2\{f_{z_i} \: : \: AF(z_i) > s \} \subseteq MF^{2-var}_{n}(\Gamma(1))\otimes \FF_2.$$
This will be accomplished by writing a finite sequence of approximations
$$ f^{(0)}_z, f^{(1)}_z, \ldots, f^{(l)}_z $$
with
$$ f^{(j)}_z \in \frac{1}{2^{k-j}} M^{2-var}_*(\Gamma(1)) \subset M^{2-var}_*(\Gamma(1)) \otimes \QQ. $$ 
and
$$ [f^{(j)}_z] = z. $$
We will then take $n(z) := l-k$ and $f_z := \frac{1}{2^{n(z)}}f^{(l)}_z$. 
\begin{description}
\item[Step 1] 
Find an element
$$ f^{(0)}_z \in \QQ[c_4, c_6, \bar{c}_4, \bar{c}_6] $$
so that 
$$ [f^{(0)}_z] = z \in v_0^{-1}\Ext(\tmf \wedge \tmf). $$
Such an $f^{(0)}_z$ can be produced in one of two ways:
\begin{description}
\item[Technique (a)] Find a representative 
$$  \bar{z} = \sum_j \bar{z}_j\xib_1^{2k_{1,j}}\xib_2^{2k_{2,j}} \cdots $$
for the image of $z$ in $\Ext(\tmf_1(3)\wedge \tmf_1(3))$ using  Corollary~\ref{cor:integralCnu} or  \ref{cor:integralCnuHZ1}.  Then by Lemma~\ref{lem:xibiti} we have
$$ \bar{z} = \sum_j \bar{z}_j [t_1(a_i, \bar{a}_i)^{k_{1,j}}t_2(a_i, \bar{a}_i)^{k_{2,j}}\cdots ]. $$
Then use Lemma~\ref{lem:cia1a3} to find $f^{(0)}_z$ so that
$$ [f^{(0)}_z] \mapsto \sum \bar{z}_j [t_1(a_i, \bar{a}_i)^{k_{1,j}}t_2(a_i, \bar{a}_i)^{k_{2,j}}\cdots ] $$
under map (1) of Diagram~\ref{eq:masterdiag2}.

\item[Technique (b)] If $z = \sum_j v_0^{-i_j}[c^{i'_j}_4] [c_6^{i''_j}]z_{1,j}z_{2,j}\cdots $, 
where inductively you already have $2$-variable modular forms $f_{z_{k,j}}$ which $z_{k,j}$ detect, you may also take $f^{(0)}_z$ to be
$$ f^{(0)}_z = \sum_j 2^{-i_j}c^{i'_j}_4 c_6^{i''_j}f_{z_{1,j}}f_{z_{2,j}}\cdots. $$
\end{description}
\item[Step 2]
Write the $q$-expansion of $f^{(0)}_z$ as 
$$ f^{(0)}_z(q,\bar{q}) = g^{(0)}(q,\bar{q})/2^k $$ 
where $g^{(0)}(q,\bar{q})$ is the $q$-expansion of 2-integral 2-variable modular form.
\item[Step 3] Write $g^{(0)}(q, \bar{q})$ as a linear combination of the $q$-expansions of the $2$-variable modular forms of Adams filtration greater than $s+k$ already produced mod $2$:
$$ g^{(0)}(q, \bar{q}) \equiv \sum_i h_i(q, \bar{q}) \mod 2. $$
\item[Step 4] Set
$$ f^{(1)}_z = f^{(0)}_z + \frac{1}{2^k}\sum_i h_i. $$
Then 
$$ [f^{(1)}_z] = [f^{(0)}_z] = z $$
and 
$$ f^{(1)}_z(q,\bar{q}) = g^{(1)}(q,\bar{q})/2^{k-1} $$
where $g^{(1)}$ is a $2$-integral $2$-variable modular form.
\item[Step 5] Repeat steps 3 and 4 to inductively produce $f^{(i)}_z$.
\end{description}

We explain all of this by working it through some low degrees:
\begin{description}
\item[$\mbf{f_1}$]  The corresponding generator of $\Ext^{0,8}_{A(2)_*}(\Sigma^8 \bou_1)$ is $\xib_1^8$.
Using ``Technique (a)'', we compute the image of $\xib_1^8$ in $\Ext(\tmf_1(3) \wedge \tmf_1(3))$ to be
$$ [t_1(a_i, \bar{a_i})^4] = \left[\frac{\ba_1^4 + a_1^4}{2^4}\right]. $$
Using Lemma~\ref{lem:cia1a3}, we take
$$ f^{(0)}_{\xib_1^8} := \frac{-\bc_4 + c_4}{2^4}. $$
We find that $f^{(0)}_{\xib_1^8}$ 
has an integral $q$-expansion, and therefore take
$$ f_1 := f^{(0)}_{\xib_1^8}. $$

\item[$\mbf{2f_2}$]  The corresponding generator of $\Ext^{3,15}_{A(2)_*}(\Sigma^8 \bou_1)$ is $[8]\xib_2^4$.
Using ``Technique (a)'', we compute (appealing to Corollary~\ref{cor:integralCnu}) its image in $\Ext(\tmf_1(3) \wedge \tmf_1(3))$ to be
$$ [8t_2(a_i, \bar{a}_i)^2 + 2a_1^2t_1(a_i, \bar{a}_i)^4] = \left[2\ba_3^2 + 2a_3^2 \right] $$ 
Using Lemma~\ref{lem:cia1a3}, we take
$$ f^{(0)}_{[8]\xib_2^4}
:= \frac{-\bc_6 + c_6}{4}. $$
We find that $f^{(0)}_{\xib_1^8}$ 
has an integral $q$-expansion.  In fact, 
$$ f^{(0)}_{[8]\xib_2^4}(q,\bar{q}) \equiv 0 \mod 2, $$
so $f^{(1)}_{[8]\xib_2^4} = f^{(0)}_{[8]\xib_2^4}$ and we define
$$ f_2 := f^{(1)}_{[8]\xib_2^4}/2. $$

\item[$\mbf{f_1^2}$]  The corresponding generator of $\Ext^{0,16}_{A(2)_*}(\Sigma^{16}\bou_2)$ is $\xib_1^{16}$.  Since $\xib_1^8$ detects $f_1$, we can simply use ``Technique (b)'' to get
$$ f^{(0)}_{\xib_1^{16}} := f_1^2. $$
The process terminates here, as $f_1^2$ is $2$-integral since $f_1$ is.

\item[$\mbf{2f_1f_2}$] The corresponding generator of $\Ext^{3,23}_{A(2)_*}(\Sigma^{16}\bou_2)$ is $\xib_1^{8}\xib_2^4$.  Again we use ``Technique (b)''.  Since $\xib_1^8$ detects $f_1$ and $[8]\xib_2^4$ detects $2f_2$, $[8]\xib_1^{8}\xib_2^4$ detects $2f_1f_2$.

\item[$\mbf{2f_3}$] The corresponding generator of $\Ext^{3,23}_{A(2)_*}(\Sigma^8 \bou_1)$ is $[c_6/4]\xib_1^8$.
Since $\xib_1^8$ detects $f_1$, we use ``Technique (b)'' to begin with 
$$ f^{(0)}_{[c_6/4]\xib_1^8} := c_6 f_1/4. $$  
This $2$-variable modular form is not $2$-integral, but the form
$$ g^{(0)} := c_6 f_1 $$
is $2$ integral (``Step 2'').
Moving on to ``Step 3'', we find
$$ c_6(q) f_1(q, \bar{q}) \equiv f_2(q, \bar{q}) c_4(q) \equiv 0 \mod 2. $$
We define 
$$ f^{(1)}_{[c_6/4]\xib_1^8} := \frac{c_6 f_1}{2} + \frac{f_2 c_4}{2}.$$
It turns out (``Step 4'')
$$ f^{(1)}_{[c_6/4]\xib_1^8}(q,\bar{q}) \equiv 0 \mod 2. $$
Therefore we define 
$$ f^{(2)}_{[c_6/4]\xib_1^8} := \frac{c_6 f_1 + f_2 c_4}{4}. $$
In fact
$$ 5c_6(q) f_1(q, \bar{q}) + 21f_2(q, \bar{q}) c_4(q) \equiv 0 \mod 8, $$
so we set
$$ f^{(3)}_{[c_6/4]\xib_1^8} := \frac{5c_6 f_1 + 21 f_2 c_4}{4} $$
and 
$$ f_3 := \frac{1}{2}f^{(3)}_{[c_6/4]\xib_1^8}.$$
\end{description}

%% file: tmfcoop_level.tex

\section{Approximating by level structures}\label{sec:approxlevel}

Recall from \S\ref{sec:review} the maps
\[
  \Psi_n:\TMF[1/n]\wedge \TMF[1/n]\to \TMF_0(n)
\]
and
\[
  \phi_{[n]}:\TMF\wedge \TMF[1/n]\to \TMF\wedge \TMF[1/n].
\]
Here $\Psi_n$ is induced by the forgetful and quotient maps $f,q:\mathcal{M}_0(n)\to
\mathcal{M}[1/n]$, while $\phi_{[n]} = 1\wedge [n]$ where
$[n]:\TMF[1/n]\to \TMF[1/n]$ is the ``Adams operation'' associated to
the multiplication by $n$ isogeny on $\mathcal{M}[1/n]$.  For reasons
which will become clear in the next section, 
we are interested in the composite map $\Psi$ given as
\[
\xymatrix{
  \tmf\wedge\tmf \ar[rr]^-\Psi\ar[d] &&\prod\limits_{i\in \ZZ,j\ge
    0}\TMF_0(3^j)\times\TMF_0(5^j),\\
  \TMF\wedge \TMF \ar[urr]_-{\psi}
}
\]
where
\[
  \psi = \prod_{i\in \ZZ,j\ge 0}\Psi_{3^j}\phi_{[3^i]}\times \Psi_{5^j}\phi_{[5^i]}.
\]
We will abuse notation and refer to the composite
\[
  \tmf\wedge\tmf \to \TMF\wedge\TMF \xrightarrow{\Psi_n}
  \TMF_0(n)
\]
(for $(2,n)=1$) as $\Psi_n$ as well; these are the $i=0$ factors of $\Psi$.

In order to study $\Psi_n$ we consider the square
\[\xymatrix{
  \tmf_*\tmf \ar[r]^{\pi_*\Psi_n}\ar[d] &\pi_*\TMF_0(n)\ar[d]\\
  M_*^{2-var}(\Gamma(1)) \ar[r]_-{\psi_n} &M_*(\Gamma_0(n)).
}\]
Here the left-hand vertical map is the composite
\[
  \tmf_*\tmf \to \tmf_*\tmf /tors\hookrightarrow
  M_*^{2-var}(\Gamma(1)) ^{AF\ge 0}\hookrightarrow M_*^{2-var}(\Gamma(1)) ,
\]
and $M_*(\Gamma_0(n))$ is the ring of level $\Gamma_0(n)$-modular
forms.  The bottom horizontal map is also induced by $f$ and $q$; if
we consider a $2$-variable modular form as a polynomial
$p(c_4,c_6,\bar{c}_4,\bar{c}_6)$, then $\psi_n(p) =
p(f^*c_4,f^*c_6,q^*c_4,q^*c_6)$.

We are especially interested in the cases $n=3,5$.  Recall from \cite{MRlevel3}
(or \cite[\S3.3]{Q5}) that $M_*(\Gamma_0(3))$ has a convenient presentation
as a subalgebra of $M_*(\Gamma_1(3))$.  More precisely,
$M_*(\Gamma_1(3))  = \ZZ [a_1,a_3,\Delta^{-1}]$ with
$\Delta = a_3^3(a_1^3-27a_3)$, and $M_*(\Gamma_0(3)) $ is the
subring 
\[
  M_*(\Gamma_0(3))  = \ZZ [a_1^2,a_1a_3,a_3^2,\Delta^{-1}].
\]
Using the formulas from \emph{loc.~cit.}, we may compute
\[
\begin{aligned}
  f^*(c_4) &= a_1^4-24a_1a_3, &q^*(c_4)&= a_1^4 + 216a_1a_3,\\
  f^*(c_6) &= -a_1^6+36a_1^3a_3-216a_3^2, &q^*(c_6)&= -a_1^6 + 540a_1^3a_3 + 5832a_3^2.
\end{aligned}
\]

There are similar formulas for the $n=5$ case which we recall from 
\cite[\S3.4]{Q5}.  Here the ring of $\Gamma_0(5)$-modular forms takes
the form
\[
  M_*(\Gamma_0(5))  = \ZZ [b_2,b_4,\delta,\Delta^{-1}]/(b_4^2=b_2^2\delta-4\delta^2),
\]
where $|b_2| = 2$ and $|b_4| = |\delta| = 4$.  (These are the
algebraic, rather than topological, degrees.)  The discriminant takes
the form
\[
  \Delta = \delta^2b_4-11\delta^3
\]
and we have
\[
\begin{aligned}
  f^*(c_4) &= b_2^2-12b_4+12\delta, &q^*(c_4) &=
  b_2^2+228b_4+492\delta,\\
  f^*(c_6) &= -b_2^3+18b_2b_4-72b_2\delta, &q^*(c_6) &= -b_2^3+522b_2b_4+10008b_2\delta.
\end{aligned}
\]

%% file: tmfcoop_cover.tex

\subsection{Faithfullness of $\psi$}

In this section we will prove the following theorem.

\begin{thm}\label{thm:faithful}
The map on homotopy
$$ \psi_* : \TMF_*\TMF  \rightarrow \prod\limits_{i\in \ZZ,j\ge
    0} \pi_* \TMF_0(3^j) \times \pi_*\TMF_0(5^j)  $$
induced by the map $\psi$ defined in the last section is injective. 
\end{thm}

Theorem~	\ref{thm:faithful} will be proven in two steps.  Consider the following diagram
\begin{equation}\label{diag:faithful}
\xymatrix{
\TMF_*\TMF  \ar[r]^-{\psi_*} \ar[d] 
&  \prod\limits_{i\in \ZZ,j\ge  0} \pi_* \TMF_0(3^j) \times \pi_*\TMF_0(5^j)  \ar[d] 
\\
\pi_*(\TMF \wedge \TMF)_{K(2)} \ar[r]^-{(\psi_{K(2)})_*} 
&  \prod\limits_{i\in \ZZ,j\ge  0} \pi_* \TMF_0(3^j)_{K(2)} \times \pi_*\TMF_0(5^j)_{K(2)}
}
\end{equation}
where the vertical maps are the localization maps.
We will first argue that the left vertical map in (\ref{diag:faithful}) is injective, and we will observe that the same argument shows the right hand vertical map is injective.  Secondly, we will show that the bottom horizontal map of (\ref{diag:faithful}) is injective.  Theorem~\ref{thm:faithful} then follows from the commutativity of (\ref{diag:faithful}) and these injectivity results.

\begin{lem}
The localization map
$$ \TMF_*\TMF  \rightarrow \TMF_*\TMF_{K(2)}$$
is injective.
\end{lem}

\begin{proof}
Since $\TMF \wedge \TMF $ is $E(2)$-local, we have
$$ (\TMF \wedge \TMF)_{K(2)} \simeq \holim_{i,j} \TMF \wedge \TMF \wedge M(2^i, v_1^j) $$
where $(i,j)$ above run over a suitable cofinal range of $\N^+ \times \N^+$.  In order to conclude that there is an isomorphism
$$ \pi_*(\TMF \wedge \TMF)_{K(2)} \cong \TMF_*\TMF^{\wedge}_{(2,c_4)} $$
and for the map
$$ \TMF_*\TMF  \rightarrow \TMF_*\TMF^{\wedge}_{(2,c_4)} $$
to be injective we must show that no element of $\TMF_*\TMF $ is infinitely divisible by elements of the ideal $(2,c_4)$.  Consider the Adams-Novikov spectral sequence for $\TMF_* \TMF $.  This spectral sequences converges since $\TMF \wedge \TMF $ is $E(2)$-local \cite[Thm.~5.3]{HoveySadofsky}.  The $E_1$-term of this spectral sequence is easily seen to not be infinitely divisible by elements of the ideal $(2,c_4)$.  Therefore, any infinite divisibility in $\TMF_*\TMF $ would have to occur through infinitely many hidden extensions.  This would result in elements in negative Adams-Novikov filtration, which is impossible.  
\end{proof}

The same argument shows that the various maps
$$ \pi_*\TMF_0(N)  \rightarrow \pi_*\TMF_0(N)_{K(2)} $$
are injections.  The only remaining step to proving Theorem~\ref{thm:faithful} is to show
the bottom arrow of Diagram~(\ref{diag:faithful}) is an injection.  This is the heart of the matter.

\begin{lem}\label{lem:faithful}
The map
$$ \pi_*(\TMF \wedge \TMF)_{K(2)} \xrightarrow{(\psi_{K(2)})_*} \prod\limits_{i\in \ZZ,j\ge  0} \pi_* \TMF_0(3^j)_{K(2)} \times \pi_*\TMF_0(5^j)_{K(2)} $$
is an injection.
\end{lem}

In order to prove this lemma, we will need the following technical observation.

\begin{lem}\label{lem:open}
Suppose that $G$ is a profinite group, $H$ is a finite subgroup of $G$, and $U$ is an open subgroup of $G$ containing $H$.  Then there is a finite set of open subgroups $U_i \le U$ which contain $H$, and a corresponding finite set $\{y_k\}$ of elements in $G$ such that
\begin{enumerate}  
\item $\{y_k U_{k} \}$ forms an open cover of $G$, and 
\item $H \cap y_kU_ky_{k}^{-1} = H \cap y_kHy_k^{-1}. $
\end{enumerate}
\end{lem}

\begin{proof}
We have
$$ H = \bigcap_{H \le V \le_o U }V $$
(where we use $\le_o$ to denote ``open subgroup'').  Therefore, for each $y \in G$, we have
$$ H \cap yHy^{-1} = \bigcap_{H \le V \le_o U} H \cap yVy^{-1}. $$
Therefore, for each $z \in H$ with $z \not\in yHy^{-1}$, there must be a subgroup $H \le V_z \le_o U$ so that $z \not\in yV_zy^{-1}$.  Define
$$ U_y = \bigcap_z V_z. $$
(If the set of all such $z$ is empty, define $U_y = U$.)
Since $H$ is finite, this is a finite intersection, hence $U_y$ is open.  Note that $U_y$ has the property that $H \le U_y \le_o U$ and
$$ H \cap yU_y y^{-1} = H \cap yHy^{-1}. $$
Consider the cover $\{ yU_y \}_y$ where $y$ ranges over the elements of $G$.  Since $G$ is compact, there is a finite subcover $\{y_k U_{y_k}\}$.  We may therefore take $U_k = U_{y_k}$.
\end{proof}

\begin{proof}[Proof of Lemma~\ref{lem:faithful}]
Let $\MS_2$ denote the second Morava stabilizer group, and let $\bar{E}_2$ denote the version of Morava $E$-theory associated to a height $2$ formal group over $\bar{\FF}_2$.  The spectrum $\bar{E}_2$ admits an action by the group $\MS_2 \rtimes Gal$ where $Gal$ is the Galois group of $\bar{\FF}_2$ over $\FF_2$, and we have
$$ \TMF_{K(2)} \simeq \left( \bar{E}_2^{hG_{24}} \right)^{hGal} $$
where $G_{24}$ is the group of automorphisms of the (unique) supersingular elliptic curve $C$ over $\bar{\FF}_2$.
In \cite{GHMR}, it is shown that this homotopy fixed point description of $\TMF_{K(2)}$ gives rise to the following description of $(\TMF \wedge \TMF)_{K(2)}$
\begin{align*} (\TMF \wedge \TMF)_{K(2)} & \simeq 
\left( \Map^c(\MS_2/G_{24}, \bar{E}_2)^{hG_{24}} \right)^{hGal}.
\end{align*}
There is a subtlety being hidden with the above notation: the Galois group is acting on the continuous mapping spectrum with the conjugation action, where it acts on the source through the left action on
$$ (\MS_2 \rtimes Gal)/(G_{24} \rtimes Gal) \cong \MS_2/G_{24}. $$
For $N$ coprime to $2$, let $\Mss_0(N)(\bar{\FF}_2)$ denote the groupoid whose objects are pairs $(C,H)$ where $C$ is a supersingular elliptic curve over $\bar{\FF}_2$ and $H \le C(\bar{\FF}_2)$ is a cyclic subgroup of order $N$, and whose morphisms are isomorphisms of elliptic curves which preserve the subgroup.  
Then we have
$$ \TMF_0(N)_{K(2)} \simeq \left( \prod_{[C,H] \in \Mss_0(N)(\bar{\FF}_2)} \bar{E}^{h\aut(C,H)}_2 \right)^{hGal}. $$
For a prime $\ell \ne 2$, let $\mr{Isog}^{ss}_\ell(\bar{\FF}_2)$ denote the groupoid whose objects are quasi-isogenies
$$ \phi: C_1 \rightarrow C_2 $$
with $C_1, C_2$ supersingular curves over $\bar{\FF}_2$,
and whose morphisms from $\phi$ to $\phi'$ are pairs of isomorphisms $(\alpha_1, \alpha_2)$ making the following square commute
$$
\xymatrix{
C_1 \ar[r]^{\phi} \ar[d]_{\alpha_1} & C_2 \ar[d]^{\alpha_2} \\
C_1' \ar[r] _{\phi'} & C_2'
}.
$$
It is easy to see that there is an equivalence of groupoids
$$ \coprod_{i \in \ZZ, j \ge 0} \Mss_0(\ell^j)(\bar{\FF}_2) \xrightarrow{\simeq} \mr{Isog}^{ss}_\ell(\bar{\FF}_2) $$
given by sending a pair $(C,H)$ to the quasi-isogeny $\phi$ given by the composite
$$ \phi: C \xrightarrow{[\ell^i]} C \rightarrow C/H. $$
However, since there is a unique supersingular elliptic curve $C$ over $\bar{\FF}_2$, the category $\mr{Isog}^{ss}_\ell(\bar{\FF}_2)$ admits the following alternative description (we actually only need that $C$ is unique up to $\ell$-power isogeny).  
Let $\Gamma_\ell$ denote the group of quasi-isogenies $\phi: C \rightarrow C$ whose order is a power of $\ell$.  There is an inclusion
$$ \Gamma_\ell \hookrightarrow \MS_2 $$
given by associating to a quasi-isogeny $\phi$ the associated automorphism $\widehat{\phi}$ of the formal group $\widehat{C}$.
Then there is a bijection between the isomorphism classes of objects of $\mr{Isog}_\ell^{ss}(\bar{\FF}_2)$ and the double cosets
$$ G_{24} \backslash \Gamma_\ell / G_{24}. $$
Moreover, given an element $[\phi] \in G_{24} \backslash \Gamma_\ell / G_{24}$, the corresponding automorphisms of the associated object $\phi$ in $\mr{Isog}^{ss}_{\ell}(\bar{\FF}_2)$ is the group
$$ G_{24} \cap \phi G_{24} \phi^{-1} \subset \Gamma_\ell. $$
Putting this all together, we have
\begin{align*}
\left( \Map(\Gamma_\ell/G_{24}, \bar{E}_2)^{hG_{24}} \right)^{hGal}
& \simeq \left( \prod_{[\phi] \in G_{24} \backslash \Gamma_\ell / G_{24}} \bar{E}_2^{hG_{24} \cap \phi G_{24} \phi^{-1}} \right)^{hGal} \\
& \simeq \left( \prod_{[\phi] \in \mr{Isog}^{ss}_\ell(\bar{\FF}_2)} \bar{E}_2^{h\aut(\phi)} \right)^{hGal} \\
& \simeq \left( \prod_{i \in \ZZ, j \ge 0} \prod_{[(C,H)] \in \Mss_0(\ell^j)(\bar{\FF}_2)} \bar{E}_2^{h\aut(C,H)} \right)^{hGal} \\
& \simeq \prod_{i \in \ZZ, j \ge 0} \TMF_0(\ell^j)_{K(2)} 
\end{align*}
and under the equivalences described above, the map
$$ \psi_{K(2)}: (\TMF \wedge \TMF)_{K(2)} \rightarrow \prod\limits_{i\in \ZZ,j\ge  0}  \TMF_0(3^j)_{K(2)} \times \TMF_0(5^j)_{K(2)} $$
can be identified with the map
\begin{equation}\label{eq:dense}
\left( \Map^c(\MS_2/G_{24}, \bar{E}_2)^{hG_{24}} \right)^{hGal} \rightarrow 
\left( \Map(\Gamma_3/G_{24} \amalg \Gamma_5/G_{24}, \bar{E}_2)^{hG_{24}} \right)^{hGal}.
\end{equation}
induced by the map 
\begin{equation}\label{eq:densecoprod}
\Gamma_3/G_{24} \amalg \Gamma_5/G_{24} \rightarrow \MS_2/G_{24}.
\end{equation}
In \cite{BehrensLawson} it is shown that the image of the above map is dense.  
Intuitively, one would like to say that this density implies that a continuous function on $\MS_2/G_{24}$ is determined by its restrictions to $\Gamma_3/G_{24}$ and $\Gamma_5/G_{24}$, and this should imply that the map (\ref{eq:dense}) is injective on homotopy.  The difficulty lies in making this argument precise.  

Before we make the argument precise (which is rather technical) we pause to give the reader an idea of the intuition behind the argument.  An element in 
$$ \pi_* \Map^c(\MS_2/G_{24}, \bar{E}_2)^{hG_{24}} $$
is something like a section of a sheaf over $G_{24} \backslash \MS_2/G_{24}$ whose stalk over $[x] \in G_{24}\backslash \MS_2/ G_{24}$ is
$$\pi_* \bar{E}^{hG_{24} \cap xG_{24}x^{-1}}_2. $$
One would like to say a section of this sheaf is trivial if its values on the stalks are trivial.  However, the actual space of continuous maps is a ($K(2)$-local) colimit of maps
\begin{align*}
\Map^c(\MS_2/G_{24}, \bar{E}_2)^{hG_{24}} & \simeq \varinjlim_{G_{24} \le U \le_o \MS_2} \Map(\MS_2/U, \bar{E}_2)^{hG_{24}} \\
& \simeq \varinjlim_{G_{24} \le U \le_o \MS_2} \prod_{[x] \in G_{24} \backslash \MS_2 /U} \bar{E}_2^{hG_{24}\cap x U x^{-1}},
\end{align*}
so an element of the homotopy of the continuous mapping space is actually represented by a kind of locally constant section with constant value over $G_{24}xU$ lying in the group
$$ \pi_* \bar{E}^{hG_{24} \cap xUx^{-1}}_2. $$
The difficulty is that there are only maps
$$ \pi_* \bar{E}^{hG_{24} \cap xUx^{-1}}_2
\rightarrow \pi_* \bar{E}^{hG_{24} \cap xG_{24}x^{-1}}_2
$$
and these maps are not necessarily injections.  The point of Lemma~\ref{lem:open} is that the open cover of $\MS_2$ given by the double cosets $G_{24}xU$ admits a finite refinement, over which 
 the ``constant sections'' have values in one of the stalks, and hence the vanishing of a value at a stalk implies the vanishing of the constant section.  
 
We now make this argument completely precise.
We have
\begin{align*}
\pi_*\left( \Map^c(\MS_2/G_{24}, \bar{E}_2)^{hG_{24}} \right)^{hGal} & \cong \varprojlim_{i,j} \varinjlim_{G_{24} \le U\le_o \MS_2} \left( \frac{\pi_* \Map(\MS_2/U, \bar{E}_2)^{hG_{24}}}{(2^i, v_1^j)} \right)^{Gal} \\ 
& \cong 
\varprojlim_{i,j} \varinjlim_{G_{24} \le U\le_o \MS_2} \left( \prod_{[x] \in G_{24} \backslash \MS_2 / U} \frac{\pi_* \bar{E}^{hG_{24} \cap xUx^{-1}}_2}{(2^i, v_1^j)} \right)^{Gal}
\end{align*}
and
\begin{align*}
\pi_*\left( \Map(\Gamma_\ell/G_{24}, \bar{E}_2)^{hG_{24}} \right)^{hGal} & \cong \varprojlim_{i,j} \left( \frac{\pi_* \Map(\Gamma_\ell/G_{24}, \bar{E}_2)^{hG_{24}}}{(2^i, v_1^j)} \right)^{Gal} \\ 
& \cong 
\varprojlim_{i,j} \left( \prod_{[x] \in G_{24} \backslash \Gamma_\ell / G_{24}} \frac{\pi_* \bar{E}^{hG_{24} \cap xG_{24}x^{-1}}_2}{(2^i, v_1^j)} \right)^{Gal}
\end{align*}
for suitable pairs $(i,j)$.
Consider the natural maps
$$
\phi_\ell : \varinjlim_{G_{24} \le U\le_o \MS_2} \prod_{[x] \in G_{24} \backslash \MS_2 / U} \frac{\pi_* \bar{E}^{hG_{24} \cap xUx^{-1}}_2}{(2^i, v_1^j)}  \rightarrow 
 \prod_{[x] \in G_{24} \backslash \Gamma_\ell / G_{24}} \frac{\pi_* \bar{E}^{hG_{24} \cap xG_{24}x^{-1}}_2}{(2^i, v_1^j)}. 
$$
Lemma~\ref{lem:faithful} will be proven if we can show that
if we are given an open subgroup $G_{24} \le U \le_o \MS_2$
and a sequence in the product
$$ (z_{G_{24}xU})_{[x]} \in \prod_{[x] \in G_{24} \backslash \MS_2 / U} \frac{\pi_* \bar{E}^{hG_{24} \cap xUx^{-1}}_2}{(2^i, v_1^j)} $$
such that
$$ \phi_\ell(z_{G_{24}xU}) = 0 $$
for $\ell = 3,5$, then there is another subgroup $G_{24} \le U' \le_o U$ such that the associated sequence
$$ 
(z_{G_{24}xU'})_{[x]} \in \prod_{[x] \in G_{24} \backslash \MS_2 / U'} \frac{\pi_* \bar{E}^{hG_{24} \cap xU'x^{-1}}_2}{(2^i, v_1^j)}
$$ 
is zero, where $ z_{G_{24}xU'} $ is the restriction to $U^\prime$ of $z_{G_{24}xU}$.  

Suppose that $(z_{G_{24}xU})_{[x]}$ is such a sequence in the kernel of $\phi_3$ and $\phi_5$.  Take a cover $\{ y_k U_k \}$ of $\MS_2$ as in Lemma~\ref{lem:open}, and let $U' = \cap_k U_k$.    
Regarding $\Gamma_3$ and $\Gamma_5$ as subgroups of $\MS_2$, the density of the image of the map (\ref{eq:densecoprod}) implies that the map
$$ \Gamma_3/U' \amalg \Gamma_5/U' \rightarrow \MS_2/U'$$
is surjective.  We therefore may assume without loss of generality that the elements  $y_k$ are either in $\Gamma_3$ or $\Gamma_5$.  We need to show that the associated sequence $(z_{G_{24}xU'})_{[x]}$ is zero.  Take a representative $x$ of a double coset $[x] \in G_{24} \backslash \MS_2 /U'$.  Then $x \in y_kU_k$ for some $k$.  Note that we therefore have
$$ G_{24} \cap x U' x^{-1} \le G_{24} \cap x U_k x^{-1} = G_{24}\cap y_k U_k y_k^{-1} = G_{24} \cap y_k G_{24} y_k^{-1} \le G_{24} \cap xUx^{-1}. $$   
Consider the associated composite of restriction maps
$$ \frac{\pi_*\bar{E}_2^{hG_{24} \cap xUx^{-1}}}{(2^i, v_1^j)} \rightarrow 
\frac{\pi_*\bar{E}_2^{hG_{24} \cap y_kG_{24}y_k^{-1}}}{(2^i, v_1^j)} \rightarrow
\frac{\pi_*\bar{E}_2^{hG_{24} \cap xU'x^{-1}}}{(2^i, v_1^j)}.
$$
The element $z_{G_{24}xU'}$ is the image of $z_{G_{24}xU}$ under the above composite.  However, since $z_{G_{24}xU}$ is in the kernel of $\phi_3$ and $\phi_5$, it follows that the image of $z_{G_{24}xU}$ is zero in
$$ \frac{\pi_*\bar{E}_2^{hG_{24} \cap y_kG_{24}y_k^{-1}}}{(2^i, v_1^j)}. $$
We therefore deduce that $z_{G_{24}xU'}$ is zero, as desired.
\end{proof}

\subsection{Computation of $\Psi_3$ and $\Psi_5$ in low degrees}\label{sec:psicomp}

Using the formulas for $f^*$ and $q^*$ for $\Gamma_0(3)$ and $\Gamma_0(5)$ in the beginning of this section, we now compute the effect of the maps $\Psi_3$ and $\Psi_5$ on a piece of $\tmf \wedge \tmf$.  Using the notation of (\ref{eq:boSES_low}), we have decompositions:
\begin{align*}
\Ext^{*,*}_{A(2)_*}(\Sigma^{16}\bou_2) & \cong \underbrace{\Ext^{*,*}_{A(1)_*}(\Sigma^{16} \FF_2) \oplus \Ext^{*,*}_{A(2)_*}(\Sigma^{24} \bou_1)}_{\Ext^{*,*}_{A(2)}(\Sigma^{16}\td{\bou}_2)} \oplus \underbrace{\Ext^{*,*}_{A(2)_*}(\Sigma^{32}\FF_2[1])}_{\Ext^{*,*}_{A(2)_*}(\Sigma^{16} \td{\td{\bou}}_2)}, \\
\Ext^{*,*}_{A(2)_*}(\Sigma^{24}\bou_3) & \cong \Ext^{*,*}_{A(1)_*}(\Sigma^{24} \FF_2) \oplus \Ext^{*,*}_{A(2)_*}(\Sigma^{32} \bou_1^{2}), \\ 
\Ext^{*,*}_{A(2)}(\Sigma^{32}\bou_4) & \cong \underbrace{\Ext^{*,*}_{A(2)_*}(\si{64}\FF_2[1])}_{\Ext^{*,*}_{A(2)_*}(\Sigma^{32}\td{\bou}_4)} \oplus 
\left(
\begin{array}{c}
\Ext^{*,*}_{A(1)_*}\left( \si{32} \tmfu_1 \oplus \si{48} \FF_2 \right) \\
\oplus 
\Ext^{*,*}_{A(2)_*}(\si{56} \bou_1 \oplus \si{56} \bou_1[1]) 
\end{array}
\right).
\end{align*}
As indicated by the underbraces above, we shall refer to the first piece of $\bou_2$ as $\td{\bou}_2$, and the second piece as $\td{\td{\bou}}_2$, and the first piece of $\bou_4$ as ${\td{\bou}}_4$.

We define a $\tmf_*$-\emph{lattice} of $\pi_*\TMF_0(\ell) $ to be a $\pi_*\tmf $-submodule $I <  \pi_*\TMF_0(\ell)$ which is finitely generated as a $\pi_*\tmf $-module, and has the property that
$$ \Delta^{-1}I = \pi_*\TMF_0(\ell) . $$
Note that the first condition forces $I$ to be concentrated in $\pi_{\ge N}\TMF_0(\ell) $ for some $N$.

We will show that a portion $I_3$ of $\tmf_*\tmf $ detected by
$$ \Ext^{*,*}_{A(2)_*}(\Sigma^{8} \bou_1 \oplus \Sigma^{16} \td{\bou}_2) $$
in the ASS maps isomorphically onto a $\tmf_*$-lattice of $\pi_*\TMF_0(3) $, recovering an observation of Davis, Mahowald, and Rezk \cite{MRlevel3}, \cite{MahowaldConnective}.  Similarly, we will show that a portion $I_5$ of $\tmf_*\tmf $ detected by
$$ \Ext^{*,*}_{A(2)_*}(\Sigma^{16} \td{\td{\bou}}_2 \oplus \Sigma^{24} \bou_3 \oplus \Sigma^{32}\td{\bou}_4) $$
in the ASS maps isomorphically onto a $\tmf_*$-lattice of $\pi_*\TMF_0(5) $.  This is a new phenomenon.

Actually, Davis, Mahowald, and Rezk proved something stronger in \cite{MRlevel3}, \cite{MahowaldConnective}: they showed ($2$-locally) that there is actually a $\tmf$-module 
$$ \td{\tmf}_0(3) := \tmf \wedge (\Sigma^{16}\bo_1 \cup \Sigma^{24} \td{\bo}_2) \cup_\beta \Sigma^{33} \tmf $$
which maps to $\TMF_0(3)$ as a \emph{connective cover}, in the sense that on homotopy groups it gives the aforementioned $\tmf_*$-lattice.  In the last section of this paper we will reprove and strengthen their result, and show that there is also a ($2$-local) $\tmf$-module
$$ \td{\tmf}_0(5) := \Sigma^{32}\tmf \cup \Sigma^{24} \tmf \wedge \bo'_3 \cup \Sigma^{64} \tmf$$
(where $\tmf \wedge \bo_3'$ is a $\tmf$-module whose cohomology is isomorphic to the cohomology of $\tmf \wedge \bo_3$ as an $A$-module)
which maps to $\TMF_0(5)$ as a connective cover, topologically realizing the corresponding $\tmf_*$-lattice of $\pi_*\TMF_0(5)$.

It will turn out that to verify these computational claims, it will suffice to compute the maps
\begin{gather*}
\Psi_3: I_3 \rightarrow \pi_*\TMF_0(3)  \\
\Psi_5: I_5 \rightarrow \pi_*\TMF_0(5) 
\end{gather*}
rationally.  The behavior of the torsion classes will then be forced.

{\it The case of $\TMF_0(3)$.}

Observe that we have
\begin{align*}
& v_0^{-1}\Ext^{*,*}_{A(2)_*}(\Sigma^{8} \bou_1 \oplus \Sigma^{16} \td{\bou}_2) \\
\\
&\quad  = v_0^{-1}\Ext^{*,*}_{A(2)_*}(\Sigma^8\bou_1) \\
& \quad \quad  \oplus v_0^{-1} \Ext^{*,*}_{A(1)_*}(\Sigma^{16} \FF_2) \\
&\quad \quad \oplus v_0^{-1} \Ext^{*,*}_{A(2)_*}(\Sigma^{24}\bou_1) \\
\\
& \quad = \FF_2[v_0^{\pm 1}, [c_4], [\Delta]]\{ [f_1], [f_2], [f_3], [f_4] \} \\
& \quad \quad \oplus \FF_2[v_0^{\pm 1}, [c_4]]\{ [f_1^2],[f_1f_2]\} \\
& \quad \quad \oplus  \FF_2[v_0^{\pm 1}, [c_4], [\Delta]]\{ [f_5], [f_6], [f_7], [f_8] \} .
\end{align*}

Recall that 
$$ M_*(\Gamma_0(3))  = \ZZ [a_1^2, a_1a_3, a_3^2] $$
(regarded as a subring of $\ZZ [a_1, a_3]$).
For a $\Gamma_0(3)$ modular form $f$, we will write
$$ f = 2^i a_1^j a_3^k + \cdots, $$
where we have
\begin{enumerate}
\item $f \equiv 0 \mod (2^{i})$, and
\item $ f \equiv 2^i a_1^j a_3^k \mod (2^{i+1}, a_1^{j+1}). $
\end{enumerate}
We shall refer to $2^i a_1^j a_3^k$ as the \emph{leading term} of $f$.

The forgetful map
$$ f^*: M_*(\Gamma(1))  \rightarrow M_*(\Gamma_0(3))  $$
is computed on the level of leading terms by
\begin{align*}
f^*(c_4) & = a_1^4 + \cdots, \\
f^*(c_6) & = a_1^6 + \cdots, \\
f^*(\Delta) & = a_3^4 + \cdots. 
\end{align*}
Using the formulas for $f^*$ and $q^*$ given in the beginning of this section, we have
\begin{equation}\label{eq:Psi_3}
\begin{array}{lll} 
\Psi_3(f_1 ) =a_1a_3 +\cdots  &\quad & \Psi_3(f_2 ) =a_1^3a_3 +\cdots \\ 
\Psi_3(f_3 ) =a_1a_3^3 +\cdots && \Psi_3(f_4 ) =a_1^3a_3^3 +\cdots \\ 
\Psi_3(f_1^2 ) =a_1^2a_3^2 +\cdots && \Psi_3(f_1 f_2 ) =a_1^4a_3^2 +\cdots \\ 
\Psi_3(f_5 ) =a_3^4 +\cdots &&\Psi_3(f_6 ) =a_3^4a_1^2 +\cdots \\ 
\Psi_3(f_7 ) =a_3^6 +\cdots && \Psi_3(f_8 ) =a_3^6a_1^2 +\cdots .\\
\end{array}
\end{equation}
It follows that on the level of leading terms, the $(\tmf_*)_\QQ$-submodule of $\tmf_*\tmf_\QQ$ given by
\begin{gather*}
\QQ[c_4, \Delta]\{f_1, f_2, f_3, f_4\} \\
\oplus \QQ[c_4]\{f_1^2, f_1f_2\} \\
\oplus \QQ[c_4, \Delta]\{f_5, f_6, f_7, f_8\}
\end{gather*}
maps under $\Psi_3$ to the $(\tmf_*)_\QQ$-lattice given by the ideal 
$$ (I_3)_\QQ := (a_1a_3, a_3^2) \subset M_*(\Gamma_0(3))_{\QQ}$$
expressed as
\begin{gather*} 
\QQ[a_1^4,  a_3^4]\{a_1a_3 , a_1^3a_3 , a_1a_3^3 , a_1^3a_3^3
\} \\
\oplus \QQ[a_1^4]\{a_1^2a_3^2, a_1^4a_3^2\}  \\
\oplus \QQ[a_1^4, a_3^4]\{a_3^4, a_3^4a_1^2, a_3^6, a_3^6a_1^2\}.
\end{gather*}

{\it The case of $\TMF_0(5)$.}

Observe that we have
\begin{align*}
& v_0^{-1}\Ext^{*,*}_{A(2)_*}(\Sigma^{16} \td{\td{\bou}}_2 \oplus \Sigma^{24} \bou_3 \oplus \Sigma^{32}\td{\bou}_4) \\
\\
&\quad  = v_0^{-1}\Ext^{*,*}_{A(2)_*}(\Sigma^{32}\FF_2[1]) \\
& \quad \quad  \oplus v_0^{-1} \Ext^{*,*}_{A(1)_*}(\Sigma^{24} \FF_2) \\
&\quad \quad \oplus v_0^{-1} \Ext^{*,*}_{A(2)_*}(\Sigma^{32}\bou_1^2) \\
&\quad \quad \oplus v_0^{-1} \Ext^{*,*}_{A(2)_*}(\Sigma^{64}\FF_2^1) \\
\\
& \quad = \FF_2[v_0^{\pm 1}, [c_4], [\Delta]]\{ [f_9], [c_6f_9] \} \\
& \quad \quad \oplus \FF_2[v_0^{\pm 1}, [c_4]]\{ [f_1^3],[f_1^2f_2]\} \\
& \quad \quad \oplus  \FF_2[v_0^{\pm 1}, [c_4], [\Delta]]\{ [f_5f_1], [f_5f_2], [f_{10}], [f_{11}], [f_7f_1], [f_7f_2], [f_{14}], [f_{15}] \} \\
& \quad \quad \oplus  \FF_2[v_0^{\pm 1}, [c_4], [\Delta]]\{[f_9^2], [c_6 f_9^2]  \} .
\end{align*}

Recall that 
$$  M_*(\Gamma_0(5))  = \ZZ [b_2,b_4,\delta]/(b_4^2=b_2^2\delta-4\delta^2). $$
For a $\Gamma_0(5)$ modular form $f$, we will write
$$ f = 2^i b_2^j \delta^k b_4^\epsilon + \cdots, $$
where $\epsilon \in \{0,1\}$ and 
\begin{enumerate}
\item $f \equiv 0 \mod (2^{i})$, and
\item $ \quad $
$$ \begin{cases} 
f \equiv 2^i b_2^j (\delta^k + \alpha \delta^{k-1}b_4) \mod (2^{i+1}, b_2^{j+1}), & \epsilon = 0, \\
f \equiv 2^i b_2^j \delta^{k}b_4 \mod (2^{i+1}, b_2^{j+1}), & \epsilon = 1. 
\end{cases}
$$
\end{enumerate}
We shall refer to $2^i b_2^j \delta^k b_4^\epsilon$ as the \emph{leading term} of $f$.

The forgetful map
$$ f^*: M_*(\Gamma(1))  \rightarrow M_*(\Gamma_0(5))  $$
is computed on the level of leading terms by
\begin{align*}
f^*(c_4) & = b_2^2 + \cdots, \\
f^*(c_6) & = b_2^3 + \cdots, \\
f^*(\Delta) & = \delta^3 + \cdots. 
\end{align*}
Unlike the case of $\Gamma_0(3)$, the $M_*(\Gamma(1))$-submodule of $2$-variable modular forms generated by the forms listed above in 
$$ v_0^{-1}\Ext^{*,*}_{A(2)_*}(\Sigma^{16} \td{\td{\bou}}_2 \oplus \Sigma^{24} \bou_3 \oplus \Sigma^{32}\td{\bou}_4) $$
does \emph{not} map nicely into $M_*(\Gamma_0(5))$.  Rather, we choose different generators as listed below.  These generators were chosen inductively (first by increasing degree, and second, by decreasing Adams filtration) by using a row echelon algorithm based on leading terms (see Examples~\ref{ex:echelon1} and \ref{ex:echelon2}).   In every case, a generator named $\td{x}$ agrees with $x$ modulo terms of higher Adams filtration:
 \begin{equation}\label{eq:newforms}
\begin{split}
 \td{f_9 } & =f_9+\Delta f_1+ c_4^2 f_1^2 ,\\
 \td{c_6f_9 } & = c_6 f_9+ c_4 \Delta f_2+ c_4^3 f_1 f_2 ,\\
 \td{f_1^3 } & =f_1^3+f_4+ c_4 f_1^2 ,\\
 \td{f_1^2f_2 } & =f_1^2 f_2 +  c_4 f_3 +  c_4 f_1 f_2 ,\\
 \td{f_5f_1 } & =f_1 f_5+\Delta f_1 ,\\
 \td{f_5f_2 } & =f_5 f_2+\Delta f_2 ,\\
 \td{f_7f_1 } & =f_1 f_7+\Delta f_3+ c_4 f_7+ c_4 \Delta f_2+ c_4^2 f_6+ c_4^3 f_1 f_2+ c_4^4 f_2 ,\\
 \td{f_7f_2 } & =f_2 f_7+\Delta f_4+ c_4 f_8+ c_4^2 \Delta f_1+ c_4^4 f_1^2 ,\\
 \td{f_{14} } & =f_{14}+\Delta f_4+ c_4^3 f_5+ c_4^3 f_4 ,\\
 \td{f_{15} } & =f_{15}+ c_4 \Delta f_3+ c_4^3 f_6+ c_4^4 f_3. \\
\end{split}
\end{equation}
The following forms, while not detected by
$\Ext^{*,*}_{A(2)_*}(\Sigma^{16} \td{\td{\bou}}_2 \oplus \Sigma^{24} \bou_3 \oplus \Sigma^{32}\td{\bou}_4)$, will be needed:
\begin{align*}
 \td{f_1^4 } & =f_1^4+ c_4 f_5+ c_4 f_4+ c_4^2 f_1^2, \\
 \td{f_1^3f_2 } & =f_1^3 f_2+ c_4 f_6+ c_4^2 f_3+ c_4^3 f_2.
\end{align*}
We now define:
\begin{align*}
 \td{f_{10} } & =f_{10}+f_7+ c_4 f_6+ c_4^2 f_1 f_2, \\
 \td{f_{11} } & =f_{11}+f_8+ c_4 \Delta f_1+ c_4^2 f_5, \\
 \td{c_4f_{10} } & = c_4\td{f_{10}}+\td{c_6f_9}+ c_4 \td{f_1^3f_2}+ c_4^2 \td{f_1^2f_2} .\\
\end{align*}
Again, the following forms are not detected by
$\Ext^{*,*}_{A(2)_*}(\Sigma^{16} \td{\td{\bou}}_2 \oplus \Sigma^{24} \bou_3 \oplus \Sigma^{32}\td{\bou}_4)$, but will be needed:
\begin{align*}
 \td{f_1^4f_2 } & =f_1^4 f_2+ c_4 \Delta f_2+ c_4^2 f_6+ c_4^3 f_3+ c_4^4 f_2+ c_4 \td{f_5f_2}, \\
 \td{f_{13} } & =f_{13}+\Delta f_3+ c_4 f_7+ c_4 \Delta f_2+ c_4^2 f_6+ c_4^3 f_3+ c_4^3 f_1 f_2+ c_4^4 f_2+\td{f_7f_1}+\frac{\td{c_4f_{10}}}{2} \\
& \quad \quad +\td{c_6f_9}+ c_4 \td{f_5f_2}+\td{f_1^4f_2}+ c_4^2\td{f_1^2f_2} .\\
\end{align*}
We then define:
\begin{align*}
 \td{f_9^2 } & =\td{f_9}^2 ,\\
 \td{c_4f_9^2 } & = c_4 \td{f_9}^2+\Delta \td{f_7f_2}+ c_4 \Delta \td{f_{11}}+ c_4^2 \Delta \td{f_5f_1}+ c_4^3 \td{f_{14}}+ c_4^5 \td{f_9}+ c_4^5 \td{f_5f_1}+ c_4^5 \td{f_1^4} ,\\
 \td{c_6f_9^2 } & = c_6 \td{f_9^2}+ c_4 \Delta \td{f_7f_1}+ c_4 \Delta \frac{\td{c_4f_{10}}}{2}+ c_4 \Delta \td{c_6f_9}+ c_4^2 \Delta \td{f_5f_2}+ c_4^4 \frac{\td{c_4f_{10}}}{2}+ c_4^4 \td{f_1^4f_2} \\
& \quad \quad + c_4^5 \td{f_1^3f_2}+ c_4^4 \td{f_{13}}.
\end{align*}

Using the formulas for $f^*$ and $q^*$ given in the beginning of this section, we have
\begin{equation}\label{eq:Psi_5}
\begin{array}{lll} 
\Psi_5( \td{f_9 }) = \delta^4 + \cdots &\quad &\Psi_5( \td{c_6f_9 }) = b_2^3\delta^4 + \cdots \\
\Psi_5( \td{f_1^3 }) = b_2^2\delta^2 + \cdots && \Psi_5( \td{f_1^2f_2 }) = b_2^3\delta^2 + \cdots \\
\Psi_5( \td{f_5f_1 }) = \delta^3b_4 + \cdots && \Psi_5( \td{f_5f_2 }) = b_2\delta^3b_4 + \cdots \\
\Psi_5( \td{f_7f_1 }) = b_2\delta^5 + \cdots && \Psi_5( \td{f_7f_2 }) = b_2^2\delta^5 + \cdots \\
\Psi_5( \td{f_{14} }) = \delta^6 + \cdots && \Psi_5( \td{f_{15} }) = b_2\delta^6 + \cdots \\
\Psi_5( \td{f_1^4 }) = b_2^2\delta^2b_4 + \cdots && \Psi_5( \td{f_1^3f_2 }) = b_2^3\delta^2b_4 + \cdots \\
\Psi_5( \td{f_{10} }) = b_2\delta^4 + \cdots && \Psi_5( \td{f_{11} }) = \delta^4b_4 + \cdots \\
\Psi_5( \td{c_4f_{10} }) = 2b_2\delta^4b_4 + \cdots && \Psi_5( \td{f_1^4f_2 }) = b_2^5\delta^3 + \cdots \\
\Psi_5( \td{f_{13} }) = b_2^9\delta + \cdots && \Psi_5( \td{f_9^2 }) = \delta^8 + \cdots \\
\Psi_5( \td{c_4f_9^2 }) = 2\delta^8b_4 + \cdots && \Psi_5( \td{c_6f_9^2 }) = b_2\delta^8b_4 + \cdots .
\end{array}
\end{equation}

\begin{ex}\label{ex:echelon1}
We explain how the above generators were produced by working through the example of $\td{f_{10}}$.  
\begin{description}
\item[Step 1] Add terms to $f_{10}$ of higher Adams filtration to ensure that $\Psi_3(\td{f_{10}}) \equiv 0 \mod 2$.  For example, we compute
$$ \Psi_3(f_{10}) = a_3^6 + \cdots. $$
According to (\ref{eq:Psi_3}), we have $\Psi_3(f_7) = a_3^6 + \cdots $.  Since $f_7$ has higher Adams filtration, we can add it to $f_{10}$ without changing the element detecting it in the ASS, to cancel the leading term of $a_3^6$.  We compute
$$ \Psi_3(f_{10} + f_7)  = a_1^6 a_3^4 + \cdots. $$
Again, using (\ref{eq:Psi_3}), we see that $\Psi_3(c_4f_6)$ (of higher Adams filtration) also has this leading term, so we now compute:
$$ \Psi_3(f_{10}+f_7+c_4 f_6) = a_1^{12}a_3^2+\cdots. $$
We see that $\Psi_3(c_4^2f_1f_2)$ also has this leading term, and 
$$ \Psi_3(f_{10}+f_7+c_4 f_6+c_4^2f_1f_2) \equiv 0 \mod 2. $$
\item[Step 2] Add terms to $f_{10}+f_7+c_4 f_6+c_4^2f_1f_2$ to ensure that the leading term of $\Psi_5(\td{f_{10}})$ is distinct from those generated by elements in lower degree, or higher Adams filtration.  In this case, we compute 
$$ \Psi_5(f_{10}+f_7+c_4 f_6+c_4^2f_1f_2) = b_2\delta^4 + \cdots. $$
By induction we know the leading term of $\Psi_5$ on generators in lower degree and higher Adams filtration, and in particular (\ref{eq:Psi_5}) tells us that this leading term is distinct from leading terms generated from elements of lower degree.  We therefore define
$$ \td{f_{10}} = f_{10}+f_7+c_4 f_6+c_4^2f_1f_2. $$
\end{description}
\end{ex}

\begin{ex}\label{ex:echelon2}
We now explain a subtlety which may arise by working through the example of $\td{c_4f_{10}}$.  
\begin{description}
\item[Step 1] We would normally add terms to $c_4f_{10}$ of higher Adams filtration to ensure that $\Psi_3(\td{c_4f_{10}}) \equiv 0 \mod 2$.  Of course, because we already know that $\Psi_3(\td{f_{10}}) \equiv 0 \mod 2$, we have
$$ \Psi_3(c_4\td{f_{10}}) \equiv 0 \mod 2. $$
\item[Step 2] We now add terms to $c_4 \td{f_{10}}$ to ensure that the leading term of $\Psi_5(\td{c_4 f_{10}})$ is distinct from those generated by elements in lower degree.  In this case, we compute 
$$ \Psi_5(c_4\td{f_{10}}) = b_2^3\delta^4 + \cdots. $$
By induction we know the leading term of $\Psi_5$ on generators in lower degree and higher Adams filtration, but now (\ref{eq:Psi_5}) tells us that 
$$ \Psi_5(\td{c_6 f_9}) = b_2^3 \delta^4 + \cdots.$$  
Since $c_6 f_9$ has higher Adams filtration, we add it to $c_4\td{f_{10}}$ and compute
$$  \Psi_5(c_4\td{f_{10}}+ \td{c_6f_9}) = b_2^5\delta^2b_4. $$
We inductively know that $\Psi_5(\td{f_1^3f_2}) = b_2^3 \delta^2b_4 + \cdots$, and we compute
$$ \Psi_5(c_4\td{f_{10}}+ \td{c_6f_9}+c_4 \td{f_1^3f_2}) = b_2^7\delta^2. $$
We inductively know that $\Psi_5(\td{f_1^2f_2}) = b_2^3 \delta^2 + \cdots$, and we compute
$$ \Psi_5(c_4\td{f_{10}}+ \td{c_6f_9}+c_4 \td{f_1^3f_2}+c_4^2\td{f_1^2f_2}) = 2b_2\delta^4 b_4 + \cdots. $$
In other words, the expression above is congruent to $0$ mod $2$, and therefore the leading term is divisible by $2$!
However, this leading term is distinct from leading terms generated from elements of lower degree, so we define
$$ \td{c_4f_{10}} = c_4\td{f_{10}}+ \td{c_6f_9}+c_4 \td{f_1^3f_2}+c_4^2\td{f_1^2f_2} $$
and record the leading term of $\Psi_5(\td{c_4f_{10}})$ as $2b_2\delta^4 b_4$.
(In fact, the 2-variable modular form $\td{c_4f_{10}}$ is $2$-divisible, and this is why some of the equations in (\ref{eq:newforms}) involve the term $\frac{\td{c_4 f_{10}}}{2}$.)
\end{description}
\end{ex}

In light of the form the leading terms of (\ref{eq:Psi_5}) take, we rewrite
\begin{align*}
& v_0^{-1}\Ext^{*,*}_{A(2)_*}(\Sigma^{16} \td{\td{\bou}}_2 \oplus \Sigma^{24} \bou_3 \oplus \Sigma^{32}\td{\bou}_4) \\
& \quad = \FF_2[v_0^{\pm 1}, [c_4], [\Delta]]\{ [f_9], [c_6f_9] \} \\
& \quad \quad \oplus \FF_2[v_0^{\pm 1}, [c_4]]\{ [f_1^3],[f_1^2f_2]\} \\
& \quad \quad \oplus  \FF_2[v_0^{\pm 1}, [c_4], [\Delta]]\{ [f_5f_1], [f_5f_2], [f_{10}], [f_{11}], [f_7f_1], [f_7f_2], [f_{14}], [f_{15}] \} \\
& \quad \quad \oplus  \FF_2[v_0^{\pm 1}, [c_4], [\Delta]]\{[f_9^2], [c_6 f_9^2]  \}
\end{align*}
in the form
\begin{gather*}
\FF_2[v_0^{\pm 1}, [c_4], [\Delta]]\{ [\td{f_9}], [\td{c_6f_9}] \} \oplus \\
 \FF_2[v_0^{\pm 1}, [c_4]]\{ [\td{f_1^3}],[\td{f_1^2f_2}]\} \oplus \\
  \FF_2[v_0^{\pm 1}, [c_4], [\Delta]]\{ [\td{f_5f_1}], [\td{f_5f_2}],  [\td{f_{11}}], [\td{c_4f_{10}}], [\td{f_7f_1}], [\td{f_7f_2}], [\td{f_{14}}], [\td{f_{15}}] \}
\oplus \FF_2[v_0^{\pm 1}, [\Delta]]\{[\td{f_{10}}]\} \\
\oplus  \FF_2[v_0^{\pm 1}, [c_4], [\Delta]]\{[\td{c_4 f_9^2}], [\td{c_6 f_9^2}]  
\} 
\oplus \FF_2[v_0^{\pm 1}, [\Delta]]\{[\td{f_{9}^2]}\}.
\end{gather*}
It follows from (\ref{eq:Psi_5}) that on the level of leading terms, the $(\tmf_*)_\QQ$ submodule of $\tmf_*\tmf_\QQ$ given by
\begin{gather*}
\QQ[c_4, \Delta]\{ \td{f_9}, \td{c_6f_9} \} \\
\oplus \QQ[c_4]\{ \td{f_1^3},\td{f_1^2f_2}\} \\
\oplus  \QQ[c_4, \Delta]\{ \td{f_5f_1}, \td{f_5f_2},  \td{f_{11}}, \frac{\td{c_4f_{10}}}{2}, \td{f_7f_1}, \td{f_7f_2}, \td{f_{14}}, \td{f_{15}} \} \oplus \QQ[\Delta]\{\td{f_{10}}\} \\
\oplus  \QQ[c_4, \Delta]\{\frac{\td{c_4 f_9^2}}{2}, \td{c_6 f_9^2}  \} \oplus \QQ[\Delta]\{\td{f_{9}^2}\}
\end{gather*}
maps under $\Psi_5$ to the $(\tmf_*)_\QQ$-lattice  
$$ (I_5)_\QQ = \QQ[b_2, \delta^3] \{b_2^2 \delta^2, \delta^3b_4, \delta^4, \delta^4b_4, b_2\delta^5,  \delta^6, \delta^8, \delta^8b_4 \} \subset M_*(\Gamma_0(5))_{\QQ}$$
expressed as
\begin{gather*} 
\QQ[b_2^2, \delta^3]\{ \delta^4, b_2^3\delta^4 \} \\
\oplus \QQ[b_2^2]\{b_2^2\delta^2, b_2^3\delta^2 \} \\
\oplus  \QQ[b_2^2, \delta^3]\{ \delta^3b_4, b_2\delta^3b_4,  \delta^4b_4, b_2\delta^4b_4, b_2\delta^5, b_2^2\delta^5, \delta^6, b_2\delta^6 \} \oplus \QQ[\Delta]\{b_2\delta^4\} \\
\oplus  \QQ[c_4, \Delta]\{\delta^8b_4, b_2\delta^8b_4  \} \oplus \QQ[\Delta]\{\delta^8 \}.
\end{gather*}

\subsection{Using level structures to detect differentials and hidden extensions in the ASS}\label{sec:diffext}

In the previous section we observed that $\Psi_3$ maps a $\tmf_*$-submodule of $\tmf_*\tmf $ detected in the ASS by
$$ \Ext^{*,*}_{A(2)_*}(\Sigma^{8} \bou_1 \oplus \Sigma^{16} \td{\bou}_2) $$
to a $\tmf_*$-lattice $I_3 \subset \pi_*\TMF_0(3) $, and $\Psi_5$ maps a $\tmf_*$-submodule of $\tmf_*\tmf $ detected in the ASS
$$ \Ext^{*,*}_{A(2)_*}(\Sigma^{16} \td{\td{\bou}}_2 \oplus \Sigma^{24} \bou_3 \oplus \Sigma^{32}\td{\bou}_4) $$
to a $\tmf_*$-lattice $I_5 \subset \pi_*\TMF_0(5) $.

We now observe that using the known structure of $\pi_*\TMF_0(3)$ and $\pi_*\TMF_0(5)$, we can deduce differentials in the portion of the ASS detected by 
$$ \Ext^{*,*}_{A(2)_*}(\Sigma^8 \bou_1 \oplus \Sigma^{16}\bou_2 \oplus \Sigma^{24}\bou_3 \oplus \Sigma^{32} \td{\bou}_4). $$

\begin{figure}
\centering
	\includegraphics[height =\textheight]{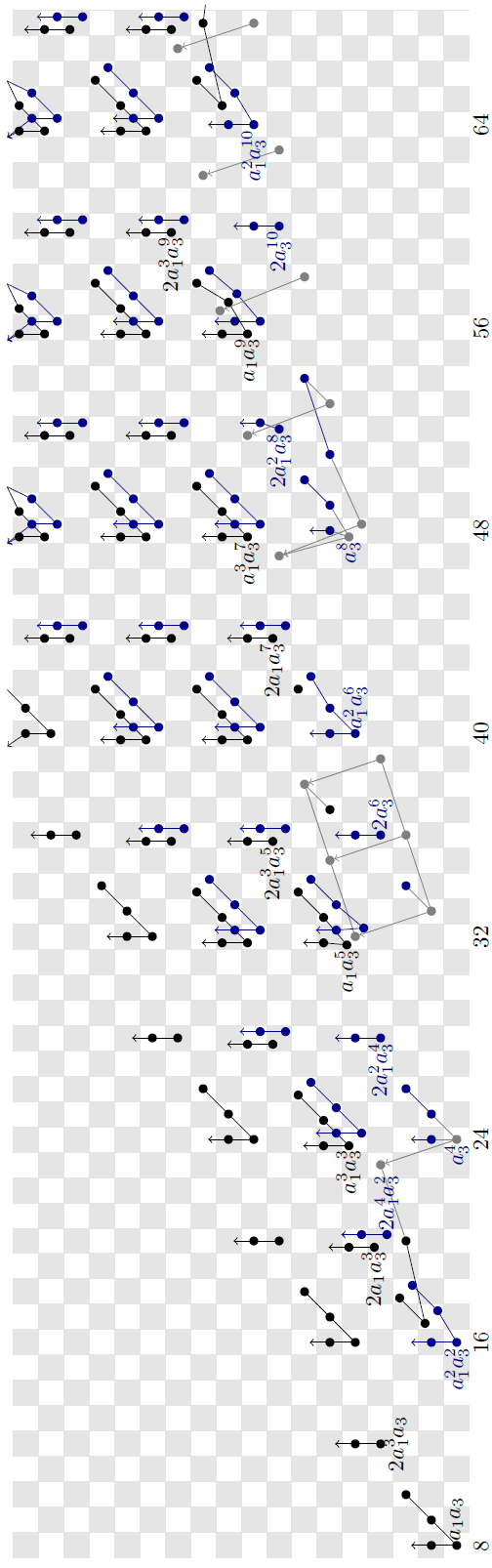}
	\caption{Differentials and hidden extensions in the portion of the ASS for $\tmf_*\tmf$ detected by $\Sigma^8 \bou_1 \oplus \Sigma^{16}\td{\bou}_2$ coming from $\TMF_0(3)$.}\label{fig:tmf3}
\end{figure}

\begin{figure}
\centering
	\includegraphics[height =\textheight]{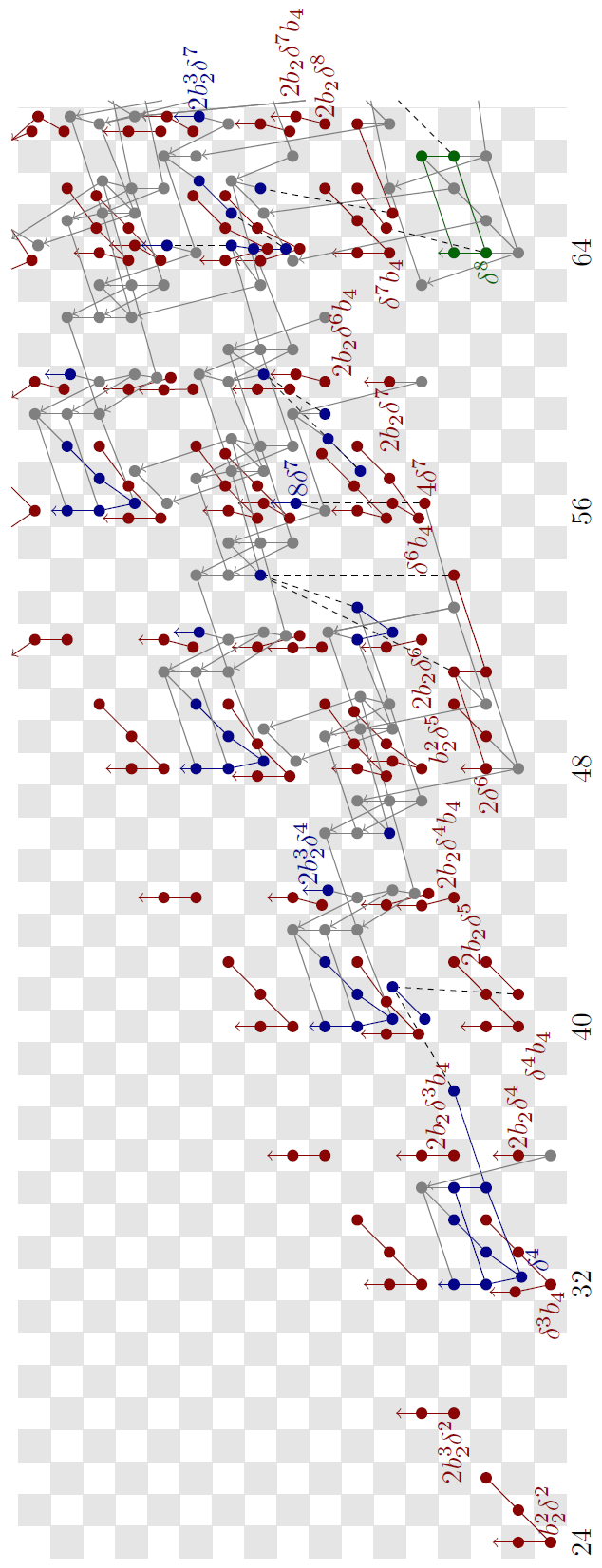}
	\caption{Differentials and hidden extensions in the portion of the ASS for $\tmf_*\tmf$ detected by $\Sigma^{16} \td{\td{\bou}}_2 \oplus \Sigma^{24}\bou_3 \oplus\Sigma^{32}\td{\bou}_4$ coming from $\TMF_0(5)$.}\label{fig:tmf5}
\end{figure}

We begin with $\Sigma^{8} \bou_1 \oplus \Sigma^{16} \td{\bou}_2$.  Figure~\ref{fig:tmf3} displays this portion of the $E_2$-term of the ASS for $\tmf_*\tmf$, with differentials and hidden extensions.  The $v_0^{-1}\Ext_{A(2)}$-generators in the chart are also labeled with $\Gamma_0(3)$-modular forms.  These are the leading terms of the $\Gamma_0(3)$-modular forms that they map to under the map $\Psi_{3}$ (see (\ref{eq:Psi_3})).  The Adams differentials and hidden extensions are all deduced from the behavior of $\Psi_3$ on these torsion-free classes, as we will now explain.  We will also describe how the $h_0$-torsion in this portion of the ASS detects homotopy classes which map isomorphically under $\Psi_3$ onto torsion in $\pi_*\TMF_0(3) $.  We freely make reference to the descent spectral sequence
$$ H^{s}(\Mcl{3} , \omega^{\otimes t}) \Rightarrow \pi_{2t-s}\TMF_0(3) , $$
 as computed in \cite{MRlevel3}.
\begin{description}
\item[Stem 17] We have
$$ \Psi_3(\eta f_4) = \eta a_1^3a_3^3 + \cdots. $$
Mahowald and Rezk \cite {MRlevel3} define a class $x$ in $\pi_{17}\TMF_0(3)$ such that
$$ c_4 x = \eta a_1^3a_3^3 + \cdots.  $$
There is a class $z_{17}$ in $\Ext^{1,18}_{A(2)_*}(\Sigma^8\bou_1)$ such that
$$ [c_4] z_{17} = h_1 [f_4]. $$ 
The class $z_{17}$ is a permanent cycle, and detects an element $y_{17} \in \tmf_{17}\tmf $.
We deduce
\begin{align*}
\Psi_3(y_{17}) & = x, \\
\Psi_3(\eta y_{17}) & = \eta x, \\
\Psi_3(\nu y_{17}) & = \nu x. \\
\end{align*}

\item[Stem 24]  The modular form $a_3^4$ is not a permanent cycle in the descent spectral sequence for $\TMF_0(3)$.  It follows that the corresponding element of $\Ext_{A(2)_*}(\td{\bou}_2)$ must support an ASS differential.  There is only one possible target for this differential.

\item[Stem 33] 
There is a class $z_{33} \in \Ext^{1,34}_{A(2)_*}(\Sigma^{16} \td{\bou}_2)$ satisfying
$$ [c_4] z_{33} = h_1[f_8].$$
There are no possible non-trivial differentials supported by $h_1 z_{33}$.
Dividing both sides of 
$$ \Psi_3(\eta^2 f_8) = \eta^2 a_1^2 a_3^6 + \cdots $$
by $c_4$, we deduce that there is an element $y_{34} \in \tmf_{34}\tmf $ detected by $h_1 z_{33}$ satisfying
$$ \Psi_3(y_{34}) = x^2. $$
Since $x^2$ is not $\eta$-divisible, we deduce that $z_{33}$ must support an Adams differential, and there is only one possible target for such a differential.
Since
$$ \Psi_3(\bar\kappa y_{17}) = \bar\kappa x = \nu x^2 $$
it follows that the element $g y_{17} \in \Ext^{5,42}_{A(2)*}(\bou_1)$ detects $\nu y_{33}$, which maps to $\nu x^2$ under $\Psi_3$.
We then deduce that
$$ \Psi_3(\bra{\eta,\nu,\nu y_{33}}) = \bra{\eta,\nu,\nu x^2} = a_1a_3 x^2. $$

\item[Stem 48] Let $z_{48} \in \Ext^{4,52}_{A(2)_*}(\td{\bou}_2)$ denote the unique non-trivial class with $h_1 z_{48} = 0$, so that $[\Delta f_5] + z_{48}$ is the unique class in that bidegree which supports non-trivial $h_1$ and $h_2$-multiplication.  Note that there is only one potential target for an Adams differential supported by $[\delta f_5]$ or $z_{48}$.  Since $a_3^8$ supports non-trivial $\eta$ and $\nu$ multiplication,  it follows that $[\Delta f_5] + z_{48}$ must be a permanent cycle in the ASS, detecting an element $y_{48} \in \tmf_*\tmf $ satisfying
$$ \Psi_3(y_{48}) = a_3^8. $$
Since $\nu^2 a_3^8$ is not $\eta$-divisible, we conclude that $h_{2,1} z_{48}$ cannot be a permanent cycle.  We deduce using $h_{2,1}$-multiplication (i.e. application of $\bra{\nu,\eta,-}$) that
$$ d_3(h_{2,1}^i z_{48}) = h^{i-1}_{2,1} d_3(h_{2,1}z_{48}) $$
for $i \ge 1$, and that\
$$ d_3(z_{48}) = d_3([\delta f_5]) \ne 0. $$
\end{description}

We now proceed to analyze $\Sigma^{16} \td{\td{\bou}}_2 \oplus \Sigma^{24} \bou_3 \oplus \Sigma^{32}\td{\bou}_4$.  Figure~\ref{fig:tmf5} displays this portion of the $E_2$-term of the ASS for $\tmf_*\tmf$, with differentials and hidden extensions.  The $v_0^{-1}\Ext_{A(2)}$-generators in the chart are also labeled with $\Gamma_0(5)$-modular forms.  These are the leading terms of the $\Gamma_0(5)$-modular forms that they map to under $\Psi_{5}$ (see (\ref{eq:Psi_5})).  As in the case of $\Sigma^8\bo_1 \oplus \Sigma^{16}\td{\bo}_2$, the Adams differentials and hidden extensions are all deduced from the behavior of $\Psi_5$ on these torsion-free classes.  We will also describe how the $h_0$-torsion in this portion of the ASS detects homotopy classes which map isomorphically under $\Psi_5$ onto torsion in $\pi_*\TMF_0(5) $.  We freely make reference to the descent spectral sequence
$$ H^{s}(\Mcl{5} , \omega^{\otimes t}) \Rightarrow \pi_{2t-s}\TMF_0(5) , $$
 as computed in \cite{Q5}, for instance.  Most of the differentials and extensions follow from the fact that the element $[f_9]$ which generates 
$$ \Ext_{A(2)_*}(\Sigma^{16}\td{\td{\bou}}_{2}) \cong \Ext_{A(2)_*}(\Sigma^{32}\FF_2[1]) $$
must be a permanent cycle in the ASS, and that the ASS for $\tmf \wedge \tmf$ is a spectral sequence of modules over the ASS for $\tmf$
$$ \Ext^{*,*}_{A(2)_*}(\FF_2) \Rightarrow \pi_* \tmf^\wedge_2. $$
Below we give some brief explanation for the main differentials and hidden extensions which do not follow from this.
\begin{description}
\item[Stem 36]  We have 
$$ \Psi_5(f_{10}) = b_2 \delta^4 + \cdots. $$
Since $b_2 \delta^4$ is not a permanent cycle in the descent spectral sequence for $\TMF_0(5) $, we deduce that $f_{10}$ must support a differential.  There is only one possibility (taking into account the differential $d_3(h_2 z_{33})$ coming from $\TMF_0(3)$),
$$ d_4([f_{10}]) = h_1^3 [f_9].$$
This is especially convenient, in light of the fact that $\eta^3 \delta^4 = 0$.

\item[Stem 41] The hidden extension follows from dividing
$$ \Psi_5(\eta [\td{f_7f_2}]) = \eta b_2^2\delta^5 + \cdots $$
by $c_4$.

\item[Stem 54] The three hidden extensions to the element $[\kappa c_4 \td{f}_9]$ all follow from the fact that $\nu^2  (2\delta^6)$ is non-trivial, and that
$$ 2  (\nu^2\delta^6) = \eta^2\bar{\kappa}\delta^2. $$

\item[Stem 56] The hidden extension follows from the Toda bracket manipulation
$$ 2\bra{\nu, 2\bar\kappa, 2\td{f_9}} = \bra{2,\nu, 2\bar\kappa}2\td{f_9}. $$

\item[Stem 64] The differential on $[\td{f^2_9}/2]$ follows from the fact that $\delta^8$ is not $2$-divisible.  The hidden extensions follow from the fact that $\eta \delta^8 \ne 0$ and $\nu^2 \delta^8 \ne 0$.

\item[Stem 65] The hidden $\eta$-extension follows from the fact that $\delta^4 \kappa \bar\kappa$ is $\eta$-divisible, and $\nu  (\delta^4 \kappa \bar\kappa) = (2\delta^6)\bar\kappa$. 

\end{description}

\subsection{Connective covers of $\TMF_0(3)$ and $\TMF_0(5)$ in the $\tmf$-resolution}\label{sec:cover}

In this section we will topologically realize the summands
\begin{gather*}
\Ext^{*,*}_{A(2)_*}(\Sigma^{8} \bou_1 \oplus \Sigma^{16} \td{\bou}_2), \\
\Ext^{*,*}_{A(2)_*}(\Sigma^{16} \td{\td{\bou}}_2 \oplus \Sigma^{24} \bou_3 \oplus \Sigma^{32}\td{\bou}_4)
\end{gather*}
of $\Ext(\tmf \wedge \tmf)$, which we showed detect $\tmf_*$-submodules that map to $\tmf_*$-lattices of $\pi_*\TMF_0(3) $ and $\pi_*\TMF_0(5) $ under the maps $\Psi_3$ and $\Psi_5$, respectively.  From now on, everything is implicitly $2$-local.

For the purposes of context, we shall say that a spectrum
$$ X\rightarrow \tmf $$ 
over $\tmf$ is a \emph{$\tmf$-Brown-Gitler spectrum} if the induced map
$$ H_* X \rightarrow H_*\tmf $$
maps $H_* X$ isomorphically onto one of the $A_*$-subcomodules $\ul{\tmf}_i \subset H_*\tmf$ defined in Section~\ref{sec:ass}.  

Not much is known about the existence of $\tmf$-Brown-Gitler spectra, but the most optimistic hope would be that the spectrum $\tmf$ admits a filtration by $\tmf$-Brown-Gitler spectra $\tmf_i$.  The case of $i = 0$ is trivial (define $\tmf_0 = S^0$) and the case of $i = 1$ is almost as easy: a spectrum $\tmf_1$ can be defined to be the $15$-skeleton:
$$ \tmf_1 := \tmf^{[15]} \hookrightarrow \tmf. $$
In light of the short exact sequences
$$ 0 \rightarrow \ul{\tmf}_{i-1} \rightarrow \ul{\tmf}_i \rightarrow \Sigma^{8i} \bou_i \rightarrow 0 $$
one would anticipate that such $\tmf$-Brown-Gitler spectra would be built from $\bo$-Brown-Gitler spectra, so that
$$ \tmf_i \simeq \bo_0 \cup \Sigma^8 \bo_1 \cup \cdots \cup \Sigma^{8i} \bo_i. $$
Davis, Mahowald, and Rezk \cite{MRlevel3}, \cite{MahowaldConnective} nearly construct a spectrum $\tmf_2$; they show that there is a subspectrum
$$ \Sigma^8 \bo_1 \cup \Sigma^{16} \bo_2 \hookrightarrow \br{\tmf} $$
(where $\br{\tmf}$ is the cofiber of the unit $S^0 \rightarrow \tmf$) realizing the subcomodule
$$ \Sigma^8 \bou_1 \oplus \Sigma^{16} \bou_2 \subseteq H_* \br{\tmf}. $$  

We will not pursue the existence of $\tmf$-Brown-Gitler spectra here, but instead will consider the easier problem of constructing the beginning of a potential filtration of $\tmf \wedge \tmf$ by $\tmf$-modules, which we denote $\tmf \wedge \tmf_i$ even though we do not require the existence of the individual spectra $\tmf_i$. We would have
$$\tmf \wedge \tmf_i \simeq \tmf \wedge \bo_0 \cup \Sigma^8 \tmf \wedge \bo_1 \cup \cdots \cup \Sigma^{8i} \tmf \wedge \bo_i,$$ 
such that the map
$$ H_* \tmf \wedge \tmf_i \rightarrow H_* \tmf \wedge \tmf. $$ 
maps $H_* \tmf \wedge \tmf_i$ onto the sub-comodule
$$  (A\mmod A(2))_* \otimes \ul{\tmf}_i \subset H_* \tmf \wedge \tmf. $$
Note that in the case of $i = 0$, we may take
$$ \tmf \wedge \tmf_0 := \tmf \xrightarrow{\eta_L} \tmf \wedge \tmf. $$
Since this is the inclusion of a summand, with cofiber denoted $\br{\tmf}$, it suffices to instead look for a filtration
$$ \tmf \wedge \br{\tmf}_1 \hookrightarrow \tmf \wedge \br{\tmf}_2 \hookrightarrow \cdots \hookrightarrow \tmf \wedge \br{\tmf} $$
of $\tmf$-modules.  Our previous discussion indicates that the cases of $i = 1$ is easy, and now the work of Davis-Mahowald-Rezk fully handles the case of $i = 2$.  In this section we will address the case of $i = 3$, and a ``piece'' of the case of $i = 4$. We state a proposition and two theorems before moving onto their proofs.

\begin{prop}\label{prop:tmf_3} $\quad$
\begin{enumerate}
\item There is a $\tmf$-module
$$ \tmf \wedge \br{\tmf}_3 \simeq \Sigma^8 \tmf \wedge \bo_1 \cup \Sigma^{16} \tmf \wedge \bo_2 \cup \Sigma^{24} \tmf \wedge \bo_3' \hookrightarrow \tmf \wedge \br{\tmf} $$
which realizes the submodule 
$$ (A\mmod A(2))_* \otimes (\Sigma^8 \bou_1 \oplus \Sigma^{16} \bou_2 \oplus \Sigma^{24} \bou_3) \subset H_* \tmf \wedge \br{\tmf} $$
where $\tmf \wedge \bo'_3$ is a $\tmf$-module with 
$$ H_*(\tmf \wedge \bo_3') \cong (A\mmod A(2))_* \otimes \bou_3 $$
(but which may not be equivalent to $\tmf \wedge \bo_3$ as a $\tmf$-module).

\item There is a map of $\tmf$-modules
$$ \Sigma^{63} \tmf \xrightarrow{\alpha} \tmf \wedge \br{\tmf}_3  $$
and an extension
$$
\xymatrix{
\tmf \wedge \br{\tmf}_3 \ar[r] \ar[d] & \tmf \wedge \tmf . \\
\tmf \wedge \br{\tmf}_3 \cup_\alpha \Sigma^{64} \tmf \ar@{.>}_{\iota}[ur]
}
$$

\item There is a modified Adams spectral sequence
$$ \Ext^{*,*}_{A(2)_*}(\Sigma^8 \bou_1 \oplus \Sigma^{16} \bou_2 \oplus \Sigma^{24} \bou_3 \oplus \Sigma^{32}\td{\bou}_4) \Rightarrow \pi_* \tmf \wedge \br{\tmf}_3 \cup_\alpha \Sigma^{64}\tmf, $$
and the map $\iota$ induces a map from this modified ASS to the ASS for $\tmf \wedge \tmf$ such that the induced map on $E_2$-terms is the inclusion of the summand
$$ \Ext^{*,*}_{A(2)_*}(\Sigma^8 \bou_1 \oplus \Sigma^{16} \bou_2 \oplus \Sigma^{24} \bou_3 \oplus \Sigma^{32}\td{\bou}_4) \hookrightarrow \Ext^{*,*}_{A(2)_*}((A\mmod A(2))_*).
$$
\end{enumerate}
\end{prop}

In \cite{MRlevel3},\cite{MahowaldConnective}, Davis, Mahowald, and Rezk construct a map
$$ \Sigma^{32} \tmf \xrightarrow{\beta} \tmf \wedge \br{\tmf}_2$$
such that
cofiber $\tmf \wedge \br{\tmf}_2 \cup_\beta \Sigma^{33} \tmf$ has an ASS with $E_2$-term
$$ E^{*,*}_2 \cong \Ext^{*,*}_{A(2)_*}(\Sigma^8 \bou_1 \oplus \td{\bou}_2) $$
and there is an equivalence
$$ v_2^{-1}( \tmf \wedge \br{\tmf}_2 \cup_\beta \Sigma^{33} \tmf) \simeq \TMF_0(3). $$
What they do not address is how this connective cover is related to $\tmf \wedge \tmf$ and the map $\Psi_3$ to $\TMF_0(3)$.

\begin{thm}\label{thm:tdtmf03}$\quad$
\begin{enumerate}
\item There is a choice of attaching map $\beta$ such that the $\tmf$-module
$$
\td{\tmf}_0(3) := \Sigma^8 \tmf \wedge \br{\tmf}_2 \cup_\beta \Sigma^{33} \tmf  \\ 
$$
fits into a diagram
\begin{equation}\label{eq:tdtmf03}
\xymatrix{
\tmf \wedge \br{\tmf}_2 \ar@{^{(}->}[r] \ar[d] &
\tmf \wedge \tmf \ar[r]^{\Psi_3} &
\TMF_0(3). \\
\td{\tmf}_0(3) \ar@{.>}[urr] 
}
\end{equation}

\item The $E_2$-term of the ASS for $\td{\tmf}_0(3)$ is given by 
$$ E_2^{*,*} = \Ext^{*,*}_{A(2)_*}(\Sigma^{8}\bou_1 \oplus \Sigma^{16}\td{\bou}_2). $$

\item The map $\tmf \wedge \br{\tmf}_2 \rightarrow \td{\tmf}_0(3)$ of Diagram~(\ref{eq:tdtmf03}) induces the projection
$$ \Ext^{*,*}_{A(2)_*}(\Sigma^{8}\bou_1 \oplus \Sigma^{16}\bou_2) \rightarrow \Ext^{*,*}_{A(2)_*}(\Sigma^{8}\bou_1 \oplus \Sigma^{16}\td{\bou}_2)$$
on Adams $E_2$-terms.

\item The map $\td{\tmf}_0(3) \rightarrow \TMF_0(3)$ of Diagram~(\ref{eq:tdtmf03}) makes $\td{\tmf}_0(3)$ a connective cover of $\TMF_0(3)$.
\end{enumerate}
\end{thm}

We also will provide the following analogous connective cover of $\TMF_0(5)$.

\begin{thm}\label{thm:tdtmf05} $\quad$
\begin{enumerate}
\item There is a $\tmf$-module
$$
\td{\tmf}_0(5) := \Sigma^{32} \tmf \cup \Sigma^{24} \tmf \wedge \bo'_3 \cup \Sigma^{64}\tmf 
$$
which fits into a diagram
\begin{equation}\label{eq:tdtmf05}
\xymatrix{
\td{\tmf}_0(5) \ar[dr] \ar@{.>}[r] & \tmf \wedge \br{\tmf}_3 \cup_{\alpha} \Sigma^{64} \tmf \ar[d] \ar[r] &
\tmf \wedge \tmf \ar[r]_{\Psi_5} &
\TMF_0(5).
\\
& \Sigma^{24} \tmf \wedge \bo_3' \cup_{\td{\alpha}} \Sigma^{64}\tmf
} 
\end{equation}

\item There is a modified ASS
$$ \Ext^{*,*}_{A(2)_*}(\Sigma^{16} \td{\td{\bou}}_2 \oplus \Sigma^{24}\bou_3 \oplus \Sigma^{32}\td{\bou}_4) \Rightarrow \pi_*(\td{\tmf}_0(5)).$$

\item The map $\td{\tmf}_0(5) \rightarrow \tmf \wedge \br{\tmf}_3 \cup_\alpha \Sigma^{64}\tmf$ of Diagram~(\ref{eq:tdtmf05}) induces a map of modified ASS's, which on $E_2$-terms is given by the inclusion
$$ \Ext^{*,*}_{A(2)_*}(\Sigma^{16} \td{\td{\bou}}_2 \oplus \Sigma^{24} \bou_3 \oplus \Sigma^{32}\td{\bou}_4) \hookrightarrow \Ext^{*,*}_{A(2)_*}(\Sigma^8 \bou_1 \oplus \Sigma^{16} \bou_2 \oplus \Sigma^{24} \bou_3 \oplus \Sigma^{32}\td{\bou}_4). $$

\item The composite $\td{\tmf}_0(5) \rightarrow \TMF_0(5)$ of Diagram~(\ref{eq:tdtmf05}) makes $\td{\tmf}_0(5)$ a connective cover of $\TMF_0(5)$.
\end{enumerate}
\end{thm}

The remainder of this section will be devoted to proving Proposition~\ref{prop:tmf_3}, Theorem~\ref{thm:tdtmf03}, and Theorem~\ref{thm:tdtmf05}. The proofs of all of these will be accomplished by taking fibers and cofibers of a series of maps, using brute force calculation of the ASS.  These brute force calculations boil down to having low degree computations of the groups $\Ext_{A(2)_*}(\bou_i, \bou_j)$ for various small values of $i$ and $j$.  The computations were performed using R.~Bruner's $\Ext$-software \cite{Bruner}.  The software requires module definition input that completely describes the $A(2)$-module structure of the modules $H^*\bo_i$.  The first author was fortunate to have an undergraduate research assistant, Brandon Tran, generate module files using Sage.

\begin{proof}[Proof of Proposition~\ref{prop:tmf_3}]
Endow $\tmf \wedge \br{\tmf}$ with a minimal $\tmf$-cell structure corresponding to an $\FF_2$-basis of $H_*\br{\tmf}$.   Let $\tmf\wedge \br{\tmf}^{[46]}$ denote the $46$-skeleton of this $\tmf$-cell module, so we have
\begin{equation}\label{eq:Hbrtmf46}
H_*(\tmf \wedge \br{\tmf}^{[46]}) \cong (A\mmod A(2))_* \otimes (  \Sigma^8 \bou_1 \oplus \Sigma^{16} \bou_2 \oplus \Sigma^{24} \bou_3 \oplus \Sigma^{32} \bou^{[14]}_4 \oplus \Sigma^{40}\bou^{[6]}_5).
\end{equation}
We first wish to form a $\tmf$-module $X_1$ with 
\begin{equation}\label{eq:HX1}
 H_* X_1 \cong (A\mmod A(2))_* \otimes (  \Sigma^8 \bou_1 \oplus \Sigma^{16} \bou_2 \oplus \Sigma^{24} \bou_3 \oplus \Sigma^{32} \bou^{[14]}_4 )
\end{equation}
by taking the fiber of a suitable map of $\tmf$-modules
$$ \gamma_5: \tmf \wedge \br{\tmf}^{[46]} \rightarrow \tmf \wedge \Sigma^{40}\bo_5^{[6]}. $$
We use the ASS
$$ \Ext^{s,t}_{A(2)_*}(H_*\br{\tmf}^{[46]}, \Sigma^{40}\bou_5^{{6}}) \Rightarrow [\Sigma^{t-s}\tmf \wedge \br{\tmf}^{[46]},\Sigma^{40}\tmf \wedge \bo_5^{[6]}]_\tmf. $$
The decomposition (\ref{eq:Hbrtmf46}) induces a corresponding decomposition of $\Ext$ groups.  The only non-zero contributions near $t-s = 0$ come from $\Sigma^{40}\bou_5^{[6]}$, $\Sigma^{32} \bou_4^{[14]}$, and $\Sigma^{24} \bou_3$; the corresponding $\Ext$ charts are depicted below.
\begin{center}
	\includegraphics[width = \textwidth]{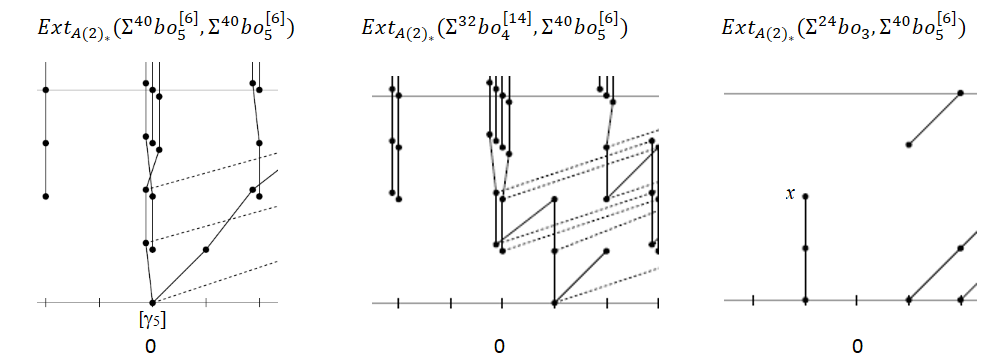}
\end{center}
The generator $[\gamma_5] \in \Ext^{0,0}_{A(2)_*}(\Sigma^{40}\bou^{[6]}_5, \Sigma^{40}\bou_5^{[6]})$ would detect the desired map $\gamma_5$.  We just need to show that this generator is a permanent cycle in the ASS.  As the charts indicate, the only potential target is the non-trivial class 
$$ x \in \Ext^{2,-1+2}_{A(2)_*}(\Sigma^{24}\bou_3, \Sigma^{40}\bou^{[6]}_5). $$
We shall call $x$ the \emph{potential obstruction} for $\gamma_5$; if $d_2(\gamma_5) = x$ then we will say that $\gamma_5$ is \emph{obstructed} by $x$.  The key observation is that in the vicinity of $t-s = 0$, the groups $\Ext_{A(1)_*}(\Sigma^{24}\bou_3, \Sigma^{40}\bou_5^{[6]})$ are depicted below.
\begin{center}
	\includegraphics[width=1.75in]{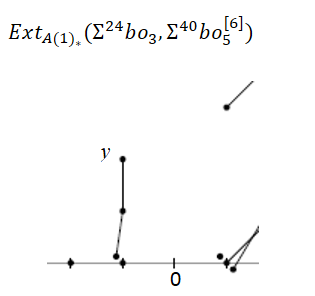}
\end{center}
Under the map of ASS's
$$
\xymatrix{
\Ext^{s,t}_{A(2)_*}(H_*\br{\tmf}^{[46]}, \Sigma^{40}\bou_5^{[6]}) \ar@{=>}[r] \ar[d] & [\Sigma^{t-s}\tmf \wedge \br{\tmf}^{[46]},\Sigma^{40}\tmf \wedge \bo_5^{[6]}]_\tmf \ar[d]^{\bo \wedge_{\tmf} -}
\\
\Ext^{s,t}_{A(1)_*}(H_*\br{\tmf}^{[46]}, \Sigma^{40}\bou_5^{[6]}) \ar@{=>}[r] & [\Sigma^{t-s}\bo \wedge \br{\tmf}^{[46]},\Sigma^{40}\bo \wedge \bo_5^{[6]}]_{\bo} 
}
$$
the potential obstruction $x$ maps to the nonzero class
$$ y \in \Ext^{2,-1+2}_{A(1)_*}(\Sigma^{24}\bou_3, \Sigma^{40}\bou^{[6]}_5). $$
Therefore if $\gamma_5$ is obstructed by $x$, then $y$ is the obstruction to the existence of a corresponding map of $\bo$-modules
$$ \bo \wedge_\tmf \gamma_5: \bo \wedge \br{\tmf}^{[46]} \rightarrow \Sigma^{40}\bo \wedge \bo_5^{[6]}. $$
However, Bailey showed in \cite{Bailey} that there is a splitting of $\bo$-modules
$$ \bo \wedge \tmf \simeq \bigvee_i \Sigma^{8i} \bo \wedge \bo_i. $$
In particular, the map $\bo \wedge_\tmf \gamma_5$ is realized by restricting the splitting map
$$ \bo \wedge \br{\tmf} \rightarrow \Sigma^{40} \bo \wedge \bo_5 $$
to $46$-skeleta (in the sense of $\bo$-cell spectra).  Therefore, $\bo \wedge_\tmf \gamma_5$ is unobstructed, and we deduce that $\gamma_5$ cannot be obstructed.

The $\tmf$-module $\tmf \wedge \br{\tmf}_3$ may then be defined to to be the fiber of a map
$$ \gamma_4: X_1 \rightarrow \Sigma^{32} \bo_4^{[14]} $$
which on homology is the projection on the summand
$$ H_* X_1 \rightarrow \Sigma^{32} (A\mmod A(2))_* \otimes \bou^{[14]}_4 .$$
Again, we use the ASS
$$ \Ext^{s,t}_{A(2)_*}(\br{\ul{\tmf}}_3 \oplus \Sigma^{32} \bou^{[14]}_4, \Sigma^{32}\bou_4^{[14]}) \Rightarrow [\Sigma^{t-s}X_1 ,\Sigma^{32}\bo_4^{[14]}]_\tmf. $$
The $E_2$-term is computed using the decomposition (\ref{eq:HX1}).  The only non-zero contributions come from the following summands:
\begin{center}
	\includegraphics[width=3in]{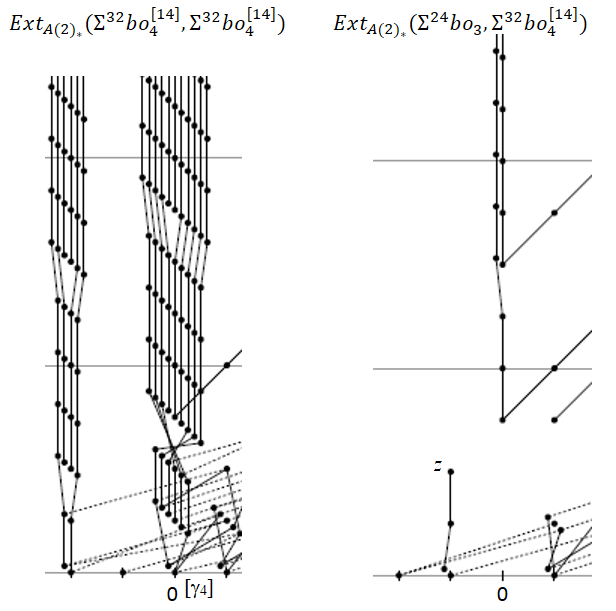}
\end{center}
We discover that the only potential obstruction to the existence of $\gamma_4$ is the non-trivial class
$$ z \in \Ext^{2,-1+2}_{A(2)_*}(\Sigma^{24}\bou_3, \Sigma^{32}\bou^{[14]}_4). $$
Unfortunately we cannot simply imitate the argument in the previous paragraph, because $z$ is in the kernel of the homomorphism
$$  \Ext^{2,-1+2}_{A(2)_*}(\Sigma^{24}\bou_3, \Sigma^{32}\bou^{[14]}_4) 
\rightarrow \Ext^{2,-1+2}_{A(1)_*}(\Sigma^{24}\bou_3, \Sigma^{32}\bou^{[14]}_4).
$$
Nevertheless, a more roundabout approach will eliminate this potential obstruction.  We first observe that there is a map of $\tmf$-modules
$$ \gamma'_4: X_1 \rightarrow \Sigma^{32} (\bo_4)_{[9]}^{[14]} $$
(with $(\bo_4)_{[9]}^{[14]}$ denoting the quotient $\bo_4^{[14]}/\bo_4^{[8]}$),
which on homology is the canonical composite
$$ H_* X_1 \rightarrow \Sigma^{32} (A\mmod A(2))_* \otimes \bou^{[14]}_4 \rightarrow \Sigma^{32}(A\mmod A(2))_* \otimes (\bou_4)_{[9]}^{[14]}.$$
The existence of $\gamma_4'$ is verified by the ASS
$$ \Ext^{s,t}_{A(2)_*}(\br{\ul{\tmf}}_3 \oplus \Sigma^{32}\bou_4^{[14]}, \Sigma^{32}(\bou_4)_{[9]}^{[14]}) \Rightarrow [\Sigma^{t-s}X_1 ,\Sigma^{32}(\bo_4)_{[9]}^{[14]}]_\tmf. $$
The $E_2$-term is computed using the decomposition (\ref{eq:HX1}).  The only non-zero contributions in the vicinity of $t-s = 0$ come from the following summands:
\begin{center}
	\includegraphics[width=3in]{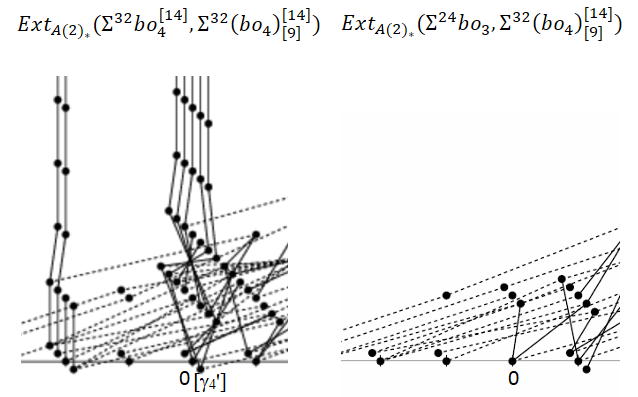}
\end{center}
We see that there are no potential obstructions for the existence of $\gamma'_4$.  Let $X_2$ denote the fiber of $\gamma'_4$, so that we have
\begin{equation}\label{eq:HX2}
 H_* X_2 \cong (A\mmod A(2))_* \otimes (  \Sigma^8 \bou_1 \oplus \Sigma^{16} \bou_2 \oplus \Sigma^{24} \bou_3 \oplus \Sigma^{32} \bou^{[8]}_4 ).
\end{equation}
We instead contemplate the potential obstructions to the existence of a map of $\tmf$-modules
$$ \gamma''_4: X_2 \rightarrow \Sigma^{32} \tmf \wedge \bo_4^{[8]} $$
which on homology induces the projection
$$ H_*X_2 \rightarrow \Sigma^{32} (A\mmod A(2))_* \otimes \bou_4^{[8]}. $$
The $E_2$-term of the ASS
$$ \Ext^{s,t}_{A(2)_*}(\br{\ul{\tmf}}_3 \oplus \Sigma^{32}\bou_4^{[8]}, \Sigma^{32}\bou_4^{[8]}) \Rightarrow [\Sigma^{t-s}X_2 ,\Sigma^{32}\bo_4^{[8]}]_\tmf $$
is computed using the decomposition (\ref{eq:HX2}), and in particular the contribution coming from the summand $\Sigma^{24} \bou_3 \subset \br{\ul{\tmf}}_3$ gives the following classes in the vicinity of $t-s = 0$:  
\begin{center}
	\includegraphics[width=1.5in]{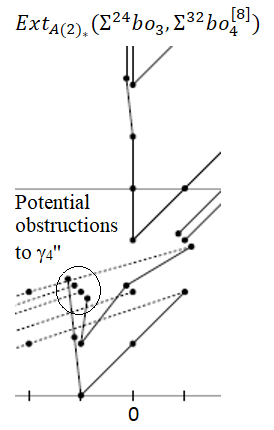}
\end{center}
We see that there are many potential obstructions to the existence of $\gamma''_4$ in 
$$ \Ext_{A(2)_*}^{2,-1+2}(\Sigma^{24}\bou_3, \Sigma^{32}\bou_4^{[8]}). $$
The potential obstructions for the related map
$$ \bo \wedge_\tmf \gamma''_4: X_2 \rightarrow \Sigma^{32} \tmf \wedge \bo_4^{[8]} $$
of $\bo$-modules in the ASS
$$ \Ext^{s,t}_{A(1)_*}(\br{\ul{\tmf}}_3 \oplus \Sigma^{32}\bou_4^{[8]}, \Sigma^{32}\bou_4^{[8]}) \Rightarrow [\Sigma^{t-s}X_2 ,\Sigma^{32}\bo_4^{[8]}]_\bo $$
may be analyzed, and the contribution coming from the summand $\Sigma^{24} \bou_3 \subset \br{\ul{\tmf}}_3$ gives the following classes in the vicinity of $t-s = 0$:  
\begin{center}
	\includegraphics[width=1.5in]{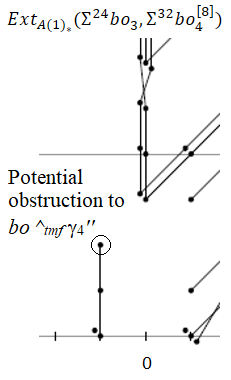}
\end{center}
We see that there is one potential obstruction to the existence of $\bo \wedge_\tmf \gamma''_4$ in 
$$ \Ext_{A(1)_*}^{2,-1+2}(\Sigma^{24}\bou_3, \Sigma^{32}\bou_4^{[8]}). $$
We analyze these potential obstructions through the following zig-zag of ASS's:
$$
\xymatrix{
\Ext^{s,t}_{A(2)_*}(\br{\ul{\tmf}}_3 \oplus \Sigma^{32}\bou_4^{[14]}, \Sigma^{32}\bou_4^{[14]}) \ar@{=>}[r] \ar[d]_r & [\Sigma^{t-s}X_1,\Sigma^{32}\tmf \wedge \bo_4^{[14]}]_\tmf \ar[d]
\\
\Ext^{s,t}_{A(2)_*}(\br{\ul{\tmf}}_3 \oplus \Sigma^{32}\bou_4^{[8]}, \Sigma^{32}\bou_4^{[14]}) \ar@{=>}[r] & [\Sigma^{t-s}X_2,\Sigma^{32}\tmf \wedge \bo_4^{[14]}]_\tmf
\\
\Ext^{s,t}_{A(2)_*}(\br{\ul{\tmf}}_3 \oplus \Sigma^{32}\bou_4^{[8]}, \Sigma^{32}\bou_4^{[8]}) \ar[u]^{i} \ar[d]_{j} \ar@{=>}[r] & [\Sigma^{t-s}X_2,\Sigma^{32}\tmf \wedge \bo_4^{[8]}]_\tmf \ar[u] \ar[d]^{\bo \wedge_{\tmf} -}
\\
\Ext^{s,t}_{A(1)_*}(\br{\ul{\tmf}}_3 \oplus \Sigma^{32}\bou_4^{[8]}, \Sigma^{32}\bou_4^{[8]})  \ar@{=>}[r] & [\Sigma^{t-s}\bo \wedge_\tmf X_2,\Sigma^{32}\bo \wedge \bo_4^{[8]}]_\bo .
}
$$
In the above diagram the potential obstruction $z$ to the existence of $\gamma_4$ maps under $r$ to a non-trivial class, so that if $z$ obstructs $\gamma_4$, then $r(z)$ obstructs the composite  
$$ \gamma_4|_{X_2}: X_2 \rightarrow X_1 \xrightarrow{\gamma_4} \Sigma^{32} \tmf \wedge \bo_4^{[14]}. $$
The key fact to check using Bruner's $\Ext$-software is that in bidegree $(t-s,s) = (-1,2)$, the maps $i$ and $j$ are both surjective, with the same kernel.  It follows that if $\gamma_4|_{X_2}$ is obstructed by $r(z)$, then the map
$$ \bo \wedge_\tmf \gamma''_4: \bo \wedge_\tmf X_2 \rightarrow \Sigma^{32}\bo \wedge \bo_4^{[8]} $$
is obstructed.  We may now avail ourselves of the Bailey splitting of $\bo \wedge \tmf$: the map $\bo \wedge_\tmf \gamma''_4$ is unobstructed, because it is realized by the projection
$$ \bo \wedge_\tmf X_2 \simeq \bo \wedge (\Sigma^8 \bo_1 \vee \Sigma^{16} \bo_2 \vee \Sigma^{24} \bo_3 \vee \Sigma^{32} \bo_4^{[8]})  \rightarrow \Sigma^{32} \bo \wedge \bo_4^{[8]}. $$
We conclude that $z$ cannot obstruct the existence of $\gamma_4$.  We may therefore define $\tmf \wedge \br{\tmf}_3$ to be the fiber of the map $\gamma_4$. 

We now need to show that the $\tmf$-module $\tmf \wedge \br{\tmf}_3$ is built as
$$ \Sigma^8 \tmf \wedge \bo_1 \cup \Sigma^{16} \tmf \wedge \bo_2 \cup \Sigma^{24} \tmf \wedge \bo_3'. $$
In order to establish this decomposition, our first task is to construct a map of $\tmf$-modules
$$ \gamma_3: \tmf \wedge \br{\tmf}_3 \rightarrow \Sigma^{24} \tmf \wedge \bo_3 $$
by analyzing the ASS
$$ \Ext^{s,t}_{A(2)_*}(\br{\ul{\tmf}}_3, \Sigma^{24}\bou_3) \Rightarrow [\Sigma^{t-s}\tmf \wedge \br{\tmf}_3 ,\Sigma^{24}\bo_3]_\tmf. $$
The only contributions in the vicinity of $t-s = 0$ come from the summands $\Sigma^{16} \bou_2$ and $\Sigma^{32} \bou_3$ of $\br{\ul{\tmf}}_3$:
\begin{center}
	\includegraphics[width=3in]{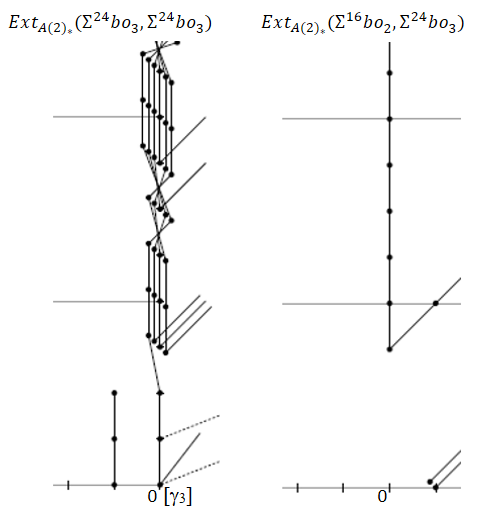}
\end{center}
As we see from the charts above there is a potential obstruction to the existence of $\gamma_3$ in
$$ \Ext^{2,-1+2}_{A(2)_*}(\Sigma^{24}\bou_3, \Sigma^{24}\bou_3). $$
The Bailey splitting does not eliminate this potential obstruction, as
$$ \Ext^{2,-1+2}_{A(1)_*}(\Sigma^{24}\bou_3, \Sigma^{24}\bou_3) = 0. $$
However, by Toda's Realization Theorem \cite{Toda}, \cite[Sec.3]{BhattacharyaEgger}, this potential obstruction also corresponds to the existence of a different ``form'' of the $\tmf$-module $\tmf \wedge \bo_3$, with the same homology.  Since $\Ext^{s,-2+s}_{A(2)_*}(\bou_3, \bou_3) = 0$ for $s \ge 3$, both forms are realized.  It follows that if $\gamma_3$ is obstructed with the standard form, then it is unobstructed for the other form.  Let $\tmf \wedge \bo'_3$ be the unobstructed form, so that there exists a map
$$ \gamma_3 : \tmf \wedge \br{\tmf}_3 \rightarrow \Sigma^{24} \tmf \wedge \bo'_3. $$
The fiber of $\gamma_3$ is $\tmf \wedge \br{\tmf}_2$, where 
$$ \br{\tmf}_2 \simeq \Sigma^8 \bo_1 \cup \Sigma^{16} \bo_2 $$ 
is the spectrum constructed by Davis, Mahowald, and Rezk.  We note that there is a fiber sequence
$$ \Sigma^8 \tmf \wedge \bo_1 \rightarrow \tmf \wedge \br{\tmf}_2 \rightarrow \Sigma^{16}\tmf \wedge \bo_2 $$
since a quick check of $\Ext^{s,-1+s}(\bou_i, \bou_i)$ reveals there are no exotic ``forms'' of $\tmf \wedge \bo_i$ for $i = 1,2$, and $\Sigma^8 \bo_1$ is the $15$-skeleton of the $\tmf$-cell complex $\tmf \wedge \br{\tmf}_2$.

We now must produce the map of $\tmf$-modules
$$ \alpha : \Sigma^{63}\tmf \rightarrow \tmf \wedge \br{\tmf}_3. $$
This just corresponds to an element $\alpha \in \pi_{63} \tmf \wedge\br{\tmf}_3$.  In the ASS
$$ \Ext^{s,t}_{A(2)_*}(\ul{\br{\tmf}}_3) \Rightarrow \pi_{t-s}(\tmf \wedge \br{\tmf}_3) $$
there is a class
$$ x_{63} \in \Ext^{4,4+63}_{A(2)_*}(\Sigma^{24}\bou_3) \subseteq \Ext^{4,4+63}_{A(2)_*}(\ul{\br{\tmf}}_3) $$
(see Figure~\ref{fig:bo3andbo4}).  Moreover, according to Figures~\ref{fig:bo1andbo2} and \ref{fig:bo3andbo4}, there are no possible targets of an Adams differential supported by this class.  Therefore, $x_{63}$ corresponds to a permanent cycle: take $\alpha$ to be the element in homotopy detected by it.  The factorization
$$
\xymatrix{
\tmf \wedge \br{\tmf}_3 \ar[r] \ar[d] & \tmf \wedge \tmf \\
\tmf \wedge \br{\tmf}_3 \cup_\alpha \Sigma^{64} \tmf \ar@{.>}_{\iota}[ur]
}
$$
exists because the element $x_{63}$, when regarded as an element of the ASS
$$ \Ext^{s,t}_{A(2)_*}(H_*\tmf) \Rightarrow \pi_{t-s} (\tmf \wedge \tmf), $$ 
is the target of a differential
$$ d_3([\td{f^2_9}/2]) = x_{63} $$
(see Figure~\ref{fig:tmf5}).

The modified ASS 
$$ \Ext^{*,*}_{A(2)_*}(\Sigma^8 \bou_1 \oplus \Sigma^{16} \bou_2 \oplus \Sigma^{24} \bou_3 \oplus \Sigma^{32}\td{\bou}_4) \Rightarrow \pi_* \tmf \wedge \br{\tmf}_3 \cup_\alpha \Sigma^{64}\tmf $$
is constructed by taking the modified Adams resolution
$$
\xymatrix{
\tmf \wedge \br{\tmf}_3 \cup_\alpha \Sigma^{64}\tmf \ar[d]_{\rho} 
& Y_1 \ar[l] \ar[d] 
& Y_2 \ar[l] \ar[d]
& \cdots, \ar[l]
\\
H \wedge \tmf \wedge \br{\tmf}_3
& H \wedge Y_1 
& H \wedge Y_2 
}
$$
where the map $\rho$ is the composite 
$$ \rho: \tmf \wedge \br{\tmf}_3 \cup_\alpha \Sigma^{64}\tmf \rightarrow H \wedge \tmf \wedge \br{\tmf}_3 \cup_\alpha \Sigma^{64}\tmf \xrightarrow{s} H \wedge \tmf \wedge \br{\tmf}_3, $$
$s$ is a section of the inclusion
$$ H \wedge \tmf \wedge \br{\tmf}_3 \hookrightarrow  H \wedge \tmf \wedge \br{\tmf}_3 \cup_\alpha \Sigma^{64}\tmf, $$
$Y_1$ is the fiber of $\rho$, and $Y_i$ is the fiber of the map
$$ Y_{i-1} \rightarrow H \wedge Y_{i-1}. $$
The map from this modified ASS to the ASS for $\tmf \wedge \tmf$ arises from the existence of a commutative diagram
$$ 
\xymatrix{
\tmf \wedge \br{\tmf}_3 \cup_\alpha \Sigma^{64}\tmf \ar[r]^-{\iota} \ar[d]_\rho
& \tmf \wedge \tmf \ar[d] 
\\
H \wedge \tmf \wedge \br{\tmf}_3 \ar@{^{(}->}[r]  &
H \wedge \tmf \wedge \tmf.
}
$$
(This diagram commutes since the class $[\td{f^2_9}/2]$ killing $x_{63}$ in the ASS for $\tmf \wedge \tmf$ has Adams filtration $1$.)
\end{proof}

\begin{proof}[Proof of Theorem~\ref{thm:tdtmf03}]
Define a $2$-variable modular form
$$
\td{\td{f_9}} = f_9 -\frac{212}{315} c_4 f_4 - \frac{34}{441} c_4 f_5 + \frac{2501}{11025} f_1^2 c_4^2 - 851 f_1 \Delta
$$
so that $[\td{\td{f_9}}] = [f_9]$ and $\Psi_3(\td{\td{f_9}}) = 0$.  (This form was produced by executing an integral variant of the ``row-reduction'' method outlined in Step 1 of Example~\ref{ex:echelon1}.) 
Then we may take the attaching map
$$ \beta : \Sigma^{32} \tmf \rightarrow \tmf \wedge \tmf $$
to be the map of $\tmf$-modules corresponding to the homotopy class
$$ \td{\td{f_9}} \in \pi_{32}\tmf \wedge \tmf. $$
We define
$$
\td{\tmf}_0(3) := \Sigma^8 \tmf \wedge \br{\tmf}_2 \cup_\beta \Sigma^{33} \tmf. 
$$
Since $\Psi_3(\td{\td{f_9}}) = 0$, there is a factorization
$$
\xymatrix{
\tmf \wedge \br{\tmf}_2 \ar@{^{(}->}[r] \ar[d] &
\tmf \wedge \tmf \ar[r]^{\Psi_3} &
\TMF_0(3). \\
\td{\tmf}_0(3) \ar@{.>}[urr] 
}
$$
The rest of the theorem is fairly straightforward given this, and our analysis of $\Psi_3$ in the previous section.
\end{proof}

\begin{proof}[Proof of Theorem~\ref{thm:tdtmf05}]
An analysis of the Adam $E_2$-terms in low dimensions reveals that the only non-trivial attaching map of $\tmf$-modules
$$ \upsilon: \Sigma^{23} \tmf \wedge \bo'_3 \rightarrow \tmf \wedge 
\br{\tmf}_2 $$
must factor as
\begin{equation}\label{eq:upsilon}
\upsilon: \Sigma^{23} \tmf \wedge \bo'_3 \xrightarrow{\upsilon'} \Sigma^{32} \tmf \xrightarrow{\beta} \tmf \wedge \br{\tmf}_2,
\end{equation}
where $\upsilon'$ is the unique non-trivial class in that degree.  The existence of differentials in Figure~\ref{fig:tmf5} from $\bou_3$-classes to $\td{\td{\bou}}_2$-classes implies that in $\tmf \wedge \br{\tmf}_3$, $\tmf \wedge \bo_3$ must be attached non-trivially to $\tmf \wedge \br{\tmf}_2$, and we therefore have
$$ \tmf \wedge \br{\tmf}_3 \simeq \tmf \wedge \br{\tmf}_2 \cup_\upsilon \Sigma^{24} \tmf \wedge \bo'_3. $$
When applied to the factorization (\ref{eq:upsilon}), Verdier's Axiom implies that there is a fiber sequence
$$ \Sigma^{32}\tmf \cup_{\upsilon'} \Sigma^{24} \tmf \wedge \bo'_3  \rightarrow \tmf \wedge \br{\tmf}_3 \rightarrow \td{\tmf}_0(3). $$
Now, an easy check with the ASS reveals that the composite
$$ \Sigma^{63} \tmf \xrightarrow{\alpha} \tmf \wedge \br{\tmf}_3 \rightarrow \td{\tmf}_0(3) $$
is null, from which it follows that there is a lift
$$
\xymatrix{
& \Sigma^{32} \tmf \cup_{\upsilon'} \Sigma^{24} \tmf \wedge \bo_3'  \ar[d] 
\\
\Sigma^{63} \tmf \ar[r]_{\alpha} \ar@{.>}[ur]^{{\alpha'}}  &
\tmf \wedge \br{\tmf}_3.
} 
$$
Define
$$ \td{\tmf}_0(5) := \Sigma^{32}\tmf \cup_{\upsilon'} \Sigma^{24}\tmf \wedge \bo'_3 \cup_{\alpha'} \Sigma^{64} \tmf. $$
Verdier's axiom, applied to the factorization above, gives a fiber sequence
$$ \td{\tmf}_0(5) \rightarrow \tmf \wedge \br{\tmf}_3 \cup_\alpha \Sigma^{64}\tmf \rightarrow \td{\tmf}_0(3). $$
Given our analysis of $\Psi_5$, the rest of the statements of the theorem are now fairly straightforward.
\end{proof}

%% file: tmfcoop.bbl
\newcommand{\etalchar}[1]{$^{#1}$}
\providecommand{\bysame}{\leavevmode\hbox to3em{\hrulefill}\thinspace}
\providecommand{\MR}{\relax\ifhmode\unskip\space\fi MR }
\providecommand{\MRhref}[2]{%
  \href{http://www.ams.org/mathscinet-getitem?mr=#1}{#2}
}
\providecommand{\href}[2]{#2}
\begin{thebibliography}{GHMR05}

\bibitem[Ada74]{Adams}
J.~F. Adams, \emph{Stable homotopy and generalised homology}, University of
  Chicago Press, Chicago, Ill., 1974, Chicago Lectures in Mathematics.
  \MR{0402720 (53 \#6534)}

\bibitem[Bai10]{Bailey}
Scott~M. Bailey, \emph{On the spectrum {$b{\rm o}\wedge{\rm tmf}$}}, J. Pure
  Appl. Algebra \textbf{214} (2010), no.~4, 392--401. \MR{2558747
  (2010i:55011)}

\bibitem[Bak95]{MR1307488}
Andrew Baker, \emph{Operations and cooperations in elliptic cohomology. {I}.
  {G}eneralized modular forms and the cooperation algebra}, New York J. Math.
  \textbf{1} (1994/95), 39--74, electronic. \MR{1307488 (96b:55004)}

\bibitem[Bau08]{Bauer}
Tilman Bauer, \emph{Computation of the homotopy of the spectrum {\tt tmf}},
  Groups, homotopy and configuration spaces, Geom. Topol. Monogr., vol.~13,
  Geom. Topol. Publ., Coventry, 2008, pp.~11--40. \MR{2508200}

\bibitem[BE16]{BhattacharyaEgger}
Prasit Bhattacharya and Phillip Egger, \emph{A class of $2$-local finite
  spectra which admit a $v_2^1$-self map}, arXiv:1608.06250, 2016.

\bibitem[Beh09]{betacong}
Mark Behrens, \emph{Congruences between modular forms given by the divided
  {$\beta$} family in homotopy theory}, Geom. Topol. \textbf{13} (2009), no.~1,
  319--357. \MR{2469520}

\bibitem[BG73]{BrownGitler}
Edgar~H. Brown, Jr. and Samuel Gitler, \emph{A spectrum whose cohomology is a
  certain cyclic module over the {S}teenrod algebra}, Topology \textbf{12}
  (1973), 283--295.

\bibitem[BHHM08]{BHHM}
M.~Behrens, M.~Hill, M.~J. Hopkins, and M.~Mahowald, \emph{On the existence of
  a {$v^{32}_2$}-self map on {$M(1,4)$} at the prime 2}, Homology, Homotopy
  Appl. \textbf{10} (2008), no.~3, 45--84. \MR{2475617 (2009j:55015)}

\bibitem[BL01]{BakerLazarev}
Andrew Baker and Andrej Lazarev, \emph{On the {A}dams spectral sequence for
  {$R$}-modules}, Algebr. Geom. Topol. \textbf{1} (2001), 173--199.
  \MR{1823498}

\bibitem[BL06]{BehrensLawson}
Mark Behrens and Tyler Lawson, \emph{Isogenies of elliptic curves and the
  {M}orava stabilizer group}, J. Pure Appl. Algebra \textbf{207} (2006), no.~1,
  37--49.

\bibitem[BL10]{TAF}
\bysame, \emph{Topological automorphic forms}, Mem. Amer. Math. Soc.
  \textbf{204} (2010), no.~958, xxiv+141. \MR{2640996}

\bibitem[BO16]{Q5}
Mark Behrens and Kyle Ormsby, \emph{On the homotopy of {$Q(3)$} and {$Q(5)$} at
  the prime 2}, Algebr. Geom. Topol. \textbf{16} (2016), no.~5, 2459--2534.
  \MR{3572338}

\bibitem[BR08]{MR2434436}
Andrew Baker and Birgit Richter, \emph{On the cooperation algebra of the
  connective {A}dams summand}, Tbil. Math. J. \textbf{1} (2008), 33--70.
  \MR{2434436 (2009d:55012)}

\bibitem[Bru93]{Bruner}
Robert~R. Bruner, \emph{{${\rm Ext}$} in the nineties}, Algebraic topology
  ({O}axtepec, 1991), Contemp. Math., vol. 146, Amer. Math. Soc., Providence,
  RI, 1993, pp.~71--90. \MR{1224908 (94a:55011)}

\bibitem[BS05]{BarkerSnaith}
Jonathan Barker and Victor Snaith, \emph{{$\psi^3$} as an upper triangular
  matrix}, $K$-Theory \textbf{36} (2005), no.~1-2, 91--114 (2006). \MR{2274160}

\bibitem[CCW01]{ClarkeCrossleyWhitehouse1}
Francis Clarke, M.~D. Crossley, and Sarah Whitehouse, \emph{Bases for
  cooperations in {$K$}-theory}, $K$-Theory \textbf{23} (2001), no.~3,
  237--250. \MR{1857208 (2002h:55019)}

\bibitem[CCW05]{ClarkeCrossleyWhitehouse2}
Francis Clarke, Martin Crossley, and Sarah Whitehouse, \emph{Algebras of
  operations in {$K$}-theory}, Topology \textbf{44} (2005), no.~1, 151--174.
  \MR{2104006 (2006b:55015)}

\bibitem[CJ92]{MR1232203}
Francis Clarke and Keith Johnson, \emph{Cooperations in elliptic homology},
  Adams {M}emorial {S}ymposium on {A}lgebraic {T}opology, 2 ({M}anchester,
  1990), London Math. Soc. Lecture Note Ser., vol. 176, Cambridge Univ. Press,
  Cambridge, 1992, pp.~131--143. \MR{1232203 (94i:55010)}

\bibitem[Coh81]{CohenZeta}
Ralph~L. Cohen, \emph{Odd primary infinite families in stable homotopy theory},
  Mem. Amer. Math. Soc. \textbf{30} (1981), no.~242, viii+92.

\bibitem[Con00]{Conrad}
Keith Conrad, \emph{A {$q$}-analogue of {M}ahler expansions. {I}}, Adv. Math.
  \textbf{153} (2000), no.~2, 185--230. \MR{1770929 (2001i:11140)}

\bibitem[Cul17]{Culver}
Dominic Culver, \emph{On ${BP}\bra{2}$ cooperations}, arXiv:1708.03001, 2017.

\bibitem[DFHH14]{TMF}
Christopher~L. Douglas, John Francis, Andr\'e~G. Henriques, and Michael~A. Hill
  (eds.), \emph{Topological modular forms}, Mathematical Surveys and
  Monographs, vol. 201, American Mathematical Society, Providence, RI, 2014.
  \MR{3223024}

\bibitem[DM10]{MahowaldConnective}
Donald~M. Davis and Mark Mahowald, \emph{Connective versions of {$TMF(3)$}},
  Int. J. Mod. Math. \textbf{5} (2010), no.~3, 223--252.

\bibitem[GHMR05]{GHMR}
P.~Goerss, H.-W. Henn, M.~Mahowald, and C.~Rezk, \emph{A resolution of the
  {$K(2)$}-local sphere at the prime 3}, Ann. of Math. (2) \textbf{162} (2005),
  no.~2, 777--822.

\bibitem[GJM86]{GoerssJonesMahowald}
Paul~G. Goerss, John D.~S. Jones, and Mark~E. Mahowald, \emph{Some generalized
  {B}rown-{G}itler spectra}, Trans. Amer. Math. Soc. \textbf{294} (1986),
  no.~1, 113--132. \MR{819938}

\bibitem[HS99]{HoveySadofsky}
Mark Hovey and Hal Sadofsky, \emph{Invertible spectra in the {$E(n)$}-local
  stable homotopy category}, J. London Math. Soc. (2) \textbf{60} (1999),
  no.~1, 284--302. \MR{1722151 (2000h:55017)}

\bibitem[Lau99]{Laures}
Gerd Laures, \emph{The topological {$q$}-expansion principle}, Topology
  \textbf{38} (1999), no.~2, 387--425.

\bibitem[LM87]{LellmannMahowald}
Wolfgang Lellmann and Mark Mahowald, \emph{The {$b{\rm o}$}-{A}dams spectral
  sequence}, Trans. Amer. Math. Soc. \textbf{300} (1987), no.~2, 593--623.

\bibitem[LN12]{LawsonNaumann}
Tyler Lawson and Niko Naumann, \emph{Commutativity conditions for truncated
  {B}rown-{P}eterson spectra of height 2}, J. Topol. \textbf{5} (2012), no.~1,
  137--168. \MR{2897051}

\bibitem[LN14]{LawsonNaumann2}
\bysame, \emph{Strictly commutative realizations of diagrams over the
  {S}teenrod algebra and topological modular forms at the prime 2}, Int. Math.
  Res. Not. IMRN (2014), no.~10, 2773--2813. \MR{3214285}

\bibitem[Mah77]{MahowaldInf}
Mark Mahowald, \emph{A new infinite family in {${}_{2}\pi_{*}{}^s$}}, Topology
  \textbf{16} (1977), no.~3, 249--256.

\bibitem[Mah81]{boresolutions}
\bysame, \emph{{$b{\rm o}$}-resolutions}, Pacific J. Math. \textbf{92} (1981),
  no.~2, 365--383. \MR{618072 (82m:55017)}

\bibitem[Mil75]{Milgram-connectivektheory}
R.~James Milgram, \emph{The {S}teenrod algebra and its dual for connective
  {$K$}-theory}, Conference on homotopy theory ({E}vanston, {I}ll., 1974),
  Notas Mat. Simpos., vol.~1, Soc. Mat. Mexicana, M\'exico, 1975, pp.~127--158.
  \MR{761725}

\bibitem[MR09]{MRlevel3}
Mark Mahowald and Charles Rezk, \emph{Topological modular forms of level 3},
  Pure Appl. Math. Q. \textbf{5} (2009), no.~2, Special Issue: In honor of
  Friedrich Hirzebruch. Part 1, 853--872.

\bibitem[Rav86]{Ravenel}
Douglas~C. Ravenel, \emph{Complex cobordism and stable homotopy groups of
  spheres}, Pure and Applied Mathematics, vol. 121, Academic Press Inc.,
  Orlando, FL, 1986. \MR{860042 (87j:55003)}

\bibitem[S{\etalchar{+}}14]{sage}
W.\thinspace{}A. Stein et~al., \emph{{S}age {M}athematics {S}oftware ({V}ersion
  x.y.z)}, The Sage Development Team, 2014, {\tt http://www.sagemath.org}.

\bibitem[Sil86]{Silverman}
Joseph~H. Silverman, \emph{The arithmetic of elliptic curves}, Graduate Texts
  in Mathematics, vol. 106, Springer-Verlag, New York, 1986. \MR{MR817210
  (87g:11070)}

\bibitem[Sto14]{grpcohcalc}
Vesna Stojanoska, \emph{Calculating descent for 2-primary topological modular
  forms}, An alpine expedition through algebraic topology, Contemp. Math., vol.
  617, Amer. Math. Soc., Providence, RI, 2014, pp.~241--258. \MR{3243402}

\bibitem[Tod71]{Toda}
Hirosi Toda, \emph{On spectra realizing exterior parts of the {S}teenrod
  algebra}, Topology \textbf{10} (1971), 53--65. \MR{0271933}

\bibitem[Wea]{PeterWear}
Peter Wear, \emph{The double cosets of a 2-adic division algebra}, to appear.

\end{thebibliography}
